\documentclass[a4paper,10pt]{imsart}
\RequirePackage{amsthm,amsmath,amsfonts,amssymb}
\RequirePackage[numbers]{natbib}
\RequirePackage[colorlinks,citecolor=blue,urlcolor=blue]{hyperref}
\RequirePackage{graphicx}
\usepackage{mathtools}
\usepackage{tikz}
\usepackage[T1]{fontenc}

\startlocaldefs
\theoremstyle{plain}
\newtheorem{theorem}{Theorem}[section]
\newtheorem{proposition}[theorem]{Proposition}
\newtheorem{lemma}[theorem]{Lemma}
\newtheorem{corollary}[theorem]{Corollary}
\newtheorem{conjecture}[theorem]{Conjecture}
\theoremstyle{remark}
\newtheorem{remark}[theorem]{Remark}
\DeclareMathOperator{\sign}{sign}

\DeclareMathOperator{\dis}{dist}
\DeclareMathOperator{\isDistr}{\stackrel{d}{=}}

\newcommand{\floor}[1]{\left\lfloor #1 \right\rfloor}

\newcommand{\Z}{\mathbb Z}
\newcommand{\R}{\mathbb R}
\newcommand{\N}{\mathbb N}

\newcommand{\F}{\mathbb F}
\newcommand{\1}{\mathbf 1}
\newcommand{\eps}{\varepsilon}
\newcommand{\dd}{\text{\upshape{d}}}

\makeatletter
\providecommand*{\cupdot}{\mathbin{\mathpalette\@cupdot{}}}
\newcommand*{\@cupdot}[2]{\ooalign{$\m@th#1\cup$\cr\hidewidth$\m@th#1\cdot$\hidewidth}}
\makeatother

\newcommand{\B}{\mathcal B}

\renewcommand{\L}{\mathcal L}
\newcommand{\Ru}{\overline R}
\newcommand{\Rl}{\underline R}
\newcommand{\xiu}{\overline\xi}
\newcommand{\xil}{\underline \xi}
\newcommand{\chiu}{\overline\chi}
\newcommand{\chil}{\underline\chi}
\newcommand{\Aupper}{\overline A}
\newcommand{\Alower}{\underline A}

\newcommand{\Pu}{\overline U}
\newcommand{\Pl}{\underline U}
\newcommand{\U}{\mathcal U}
\newcommand{\oomega}{\overline\omega}
\newcommand{\prev}{\operatorname{prev}}

\endlocaldefs

\begin{document}

\begin{frontmatter}
\title{Anomalous scaling regime for\\one-dimensional Mott variable-range hopping}
\runtitle{Anomalous one-dimensional Mott variable-range hopping}

\begin{aug}
\author[A]{\fnms{David A.}~\snm{Croydon}\ead[label=e1]{croydon@kurims.kyoto-u.ac.jp}\orcid{0000-0002-1520-1468}},
\author[B]{\fnms{Ryoki}~\snm{Fukushima}\ead[label=e2]{ryoki@math.tsukuba.ac.jp}\orcid{0000-0002-7582-6793}}
\and
\author[C]{\fnms{Stefan}~\snm{Junk}\ead[label=e3]{sjunk@tohoku.ac.jp}\orcid{0000-0002-5151-8008}}

\address[A]{Research Institute for Mathematical Sciences, Kyoto University\printead[presep={,\ }]{e1}}

\address[B]{Institute of Mathematics, University of Tsukuba\printead[presep={,\ }]{e2}}

\address[C]{Advanced Institute for Materials Research, Tohoku University\printead[presep={,\ }]{e3}}
\end{aug}

\begin{abstract}
We derive an anomalous, sub-diffusive scaling limit for a one-dimen-sional
version of the Mott random walk. The limiting process can be viewed heuristically as a one-dimensional diffusion with an absolutely continuous speed measure and a discontinuous scale function, as given by a two-sided stable subordinator. Corresponding to intervals of low conductance in the discrete model, the discontinuities in the scale function act as barriers off which the limiting process reflects for some time before crossing. We also discuss how, by incorporating a Bouchaud trap model element into the setting, it is possible to combine this `blocking' mechanism with one of `trapping'. Our proof relies on a recently developed theory that relates the convergence of processes to that of associated resistance metric measure spaces.
\end{abstract}




\end{frontmatter}

\tableofcontents

\section{Introduction}

\subsection{Definition of the model} Mott variable-range hopping is a model of low-temperature conduction in a disordered medium in the Anderson localisation regime. In Mott's original paper, the hopping of electrons between localisation sites was assumed to depend on the spatial and energy separation of sites \cite{mott}. As is set-out precisely below, such a phenomenon can be described by a suitable random walk in a random environment. In this article, we study a one-dimensional version of the Mott random walk in a regime where the inhomogeneity of the environment persists asymptotically, leading to an anomalous, sub-diffusive scaling limit for the discrete process. Our arguments will demonstrate that the sub-diffusivity observed is due to a certain `blocking' mechanism, which can be viewed as a natural counterpart to the `trapping' seen in the Bouchaud trap model.

We start by introducing the model of interest. Let
\[\dots <\omega_{-2}< \omega_{-1}<\omega_0=0<\omega_1<\omega_2<\cdots\]
be the atoms of a homogeneous Poisson process on $\mathbb{R}$ with intensity $\rho\in(0,\infty)$, conditioned to have an atom at zero (i.e.\ sampled according to the relevant Palm distribution). The points $\omega=(\omega_i)_{i\in\mathbb{Z}}$ represent electron localisation sites, and to capture the corresponding energy marks, we suppose $E=(E_i)_{i\in\Z}$ is an independent and identically distributed (i.i.d.) family of random variables on $\mathbb{R}$, independent of $(\omega_i)_{i\in\Z}$. For a given realisation of the environment variables $(\omega, E)$, we define conductances $(c^{\beta,\lambda}(x,y))_{x,y\in \omega}$ by setting
\begin{equation}\label{cdef}
c^{\beta,\lambda}(\omega_i,\omega_j):=\exp\left(-|\omega_i-\omega_j|-\beta U(E_i,E_j)+\lambda(\omega_i+\omega_j)\right),
\end{equation}
where $U\colon \R\times\R\to[0,1]$ is a symmetric function and $\beta\geq 0$, $\lambda\in[0,1)$ are parameters. Note that, in addition to the terms depending on the spatial separation and energy marks (i.e.\ $|\omega_i-\omega_j|$ and $\beta U(E_i,E_j)$, respectively), we include the term $\lambda(\omega_i+\omega_j)$ to model the effect of an external field. The version of the Mott random walk studied here is then the continuous-time Markov chain $X^{\beta,\lambda}=(X^{\beta,\lambda}_t)_{t\geq 0}$ on $\omega$ with generator given by
\begin{align}\label{generator}
(L^{\beta,\lambda}f)(\omega_i)&:=\sum_{j\in\Z}\frac{c^{\beta,\lambda}(\omega_i,\omega_j)}{c^{\beta,\lambda}(\omega_i)}\left(f(\omega_j)-f(\omega_i)\right),
\end{align}
where $c^{\beta,\lambda}(\omega_i):=\sum_{j\in\Z}c^{\beta,\lambda}(\omega_i,\omega_j)$ is the invariant measure. A few remarks are in order. First, since $\lambda$ is assumed to take values in $[0,1)$, the random variables $c^{\beta,\lambda}(\omega_i)$ are readily checked to be almost-surely finite. Second, taking $\beta=0$ for simplicity, we can rewrite the jump-rate from $i$ to $j$ as
\begin{equation*}
\frac{c^{0,\lambda}(\omega_i,\omega_j)}{c^{0,\lambda}(\omega_i)}
=\frac{\exp(-|\omega_i-\omega_j|)}{\sum_{k\in\Z}\exp(-|\omega_i-\omega_k|+\lambda(\omega_k-\omega_i))}e^{\lambda(\omega_j-\omega_i)};
\end{equation*}
the role of $\lambda$ might be clearer in this form. Third, the process $X^{\beta,\lambda}$ is the so-called constant-speed random walk; see Remark~\ref{rem:CSvsVS} for comparison with the related variable-speed random walk with jump rates given by $c^{\beta,\lambda}(\omega_i,\omega_j)e^{-2\lambda \omega_i}$. We write $P^{\beta,\lambda}$ for the law of $X^{\beta,\lambda}$ started from $0$, conditional on $(\omega,E)$; this is the so-called quenched law of $X^{\beta,\lambda}$. The corresponding annealed law is obtained by integrating out the randomness of the environment, i.e.\
\begin{equation}\label{annealed}
\mathbb{P}^{\beta,\lambda}:=\int P^{\beta,\lambda}\left(\cdot\right)\mathbf{P}(\dd\omega\dd E),
\end{equation}
where $\mathbf{P}$ is the probability measure on the probability space upon which the pair $(\omega,E)$ is built. To be more specific, we assume that both $P^{\beta,\lambda}$ and $\mathbb{P}^{\beta,\lambda}$ are probability measures on the space of c\`adl\`ag functions $D([0,\infty),\mathbb{R})$, which we will assume throughout is equipped with the usual Skorohod $J_1$-topology.

\subsection{Diffusivity/sub-diffusivity phase transition} It is known that when the density of localisation sites is suitably high, that is, when $\rho>1$, the symmetric Mott random walk undergoes homogenisation. Indeed, in this case, one has that, for any value of $\beta\geq 0$ and $\mathbf{P}$-a.e.\ realisation of $(\omega,E)$, under the quenched law,
\begin{equation}\label{homog}
\left(n^{-1}X^{\beta,0}_{n^2t}\right)_{t\geq 0}\xrightarrow[n\to\infty]{} (B_{\sigma^2t})_{t\geq 0},
\end{equation}
in distribution, where $(B_t)_{t\geq 0}$ is a standard Brownian motion, and $\sigma^2\in(0,\infty)$ is a deterministic constant \cite{CF}. Other homogenization statements for certain elliptic and parabolic equations associated with this model appear in \cite{Fnew}. On the other hand, it was also established in \cite{CF} that when $\rho\leq 1$, the limit at \eqref{homog} is valid with respect to the annealed law (and indeed in a slightly stronger sense), but with a limiting diffusion constant $\sigma^2=0$. Our principal goal is to describe the appropriate scaling for the Mott random walk in this sub-diffusive regime in the symmetric case ($\lambda=0$). We will henceforth consider the case where $\rho\leq1$. Note that we will state results and present proofs only in the case $\rho<1$ for technical convenience. For the boundary case $\rho=1$, see Remark \ref{gaps} below.
\begin{remark}

\label{rem:blocking}
Sub-diffusivity in the regime $\rho<1$ can be heuristically explained as follows. Within the interval $[-\epsilon n, \epsilon n]$ for arbitrarily small $\epsilon>0$, one finds a pair $(\omega_i,\omega_{i+1})$ with $\omega_{i+1}-\omega_i\ge (\rho^{-1}+o(1))\log n$ with high probability as $n\to\infty$, on both sides of the origin. Now by \eqref{generator}, it is reasonable to believe that the random walk has to make $n^{\rho^{-1}+o(1)}$ trials from a neighborhood of $\omega_i$ or $\omega_{i+1}$ to get over such gaps. But since the invariant measure is approximately uniform on the spatial scale $n$, the Mott random walk is able to make only $n^{1+o(1)}$ visits to the above neighborhood up to time $n^2$. Thus when $\rho<1$, the random walk cannot make enough trials to get over such gaps by time $n^2$, and hence $n^{-1}X^{\beta,0}_{n^2 t}=o(1)$. This `blocking' by large gaps in $\omega$ is the main feature of this model in the sub-diffusive regime. Taking the above argument slightly further, one might conjecture that the time scale on which the process $X^{\beta,0}$ is able to cross gaps of the magnitude described is $n^{1+1/\rho}$, and this is indeed what we see in our main result below.
\end{remark}

\subsection{The main result}

The following result characterizes the scaling for the Mott random walk when $\rho<1$. We additionally include a `weak' bias, which, although vanishing for the discrete model, impacts the limiting process that arises.

\begin{theorem}
\label{thm:main}
For every $\rho<1$ and $\beta, \lambda\geq 0$, it holds that as $n\to\infty$,
\[\mathbb{P}^{\beta,\lambda/n}\left((n^{-1}X_{n^{1+1/\rho}t})_{t\geq 0}\in\cdot\right)\]
converge weakly as probability measures on $D([0,\infty),\mathbb{R})$ to the law of the continuous process $Z^{\beta,\lambda}$ defined below.
\end{theorem}

In order to define the limiting process $Z^{\beta,\lambda}$, we introduce two objects: a standard Brownian motion $(B_t)_{t\geq 0}$ and an independent two-sided $\rho$-stable L\'evy process $(S^{\beta,0}(u))_{u\in\R}$ (i.e.\ $(S^{\beta,0}(u))_{u\geq0}$ and $(-S^{\beta,0}((-u)^-))_{u\geq0}$ are independent $\rho$-stable L\'evy processes, each started from 0) with L\'evy measure given by
\begin{equation}\label{Levy}
C_\beta \rho x^{-\rho-1}\mathbf{1}_{\{x>0\}}\dd x,
\end{equation}
where $C_\beta\in(0,\infty)$ is a constant that is defined below at \eqref{eq:cbeta}. We also define an exponentially `tilted' version $(S^{\beta,\lambda}(u))_{u\in\R}$ of the L\'evy process by setting
\begin{equation}\label{eq:stilt}
S^{\beta,\lambda}(u):=\int_0^ue^{-2\lambda v/\rho}\dd S^{\beta,0}( v),
\end{equation}
in the sense of the Stieltjes integral, and a measure $\mu^{\beta,\lambda}$ supported on the closure of its image $\overline{S^{\beta,\lambda}(\R)}\subseteq\R$ by
\begin{equation}
\label{eq:mudef}
\mu^{\beta,\lambda}\left((a,b]\right):=\mathbf{E}\left(c^{\beta,0}(\omega_0)\right)\int_{(S^{\beta,\lambda})^{-1}(a)}^{(S^{\beta,\lambda})^{-1}(b)}e^{2\lambda r/\rho}\dd r,
\end{equation}
where $(S^{\beta,\lambda})^{-1}$ denotes the right-continuous inverse of $S^{\beta,\lambda}$, i.e.\
\[(S^{\beta,\lambda})^{-1}(u):=\inf\left\{v\in\mathbb{R}:\:S^{\beta,\lambda}(v)>u\right\},\]
and we will later check that $\mathbf{E}(c^{\beta,0}(\omega_0))\in(0,\infty)$. Next, writing $(L^B_t(x))_{t\geq 0,x\in\R}$ for the local time of $(B_t)_{t\geq 0}$, we let
\begin{equation}\label{hdef}
H_t^{\beta,\lambda}:=\inf\left\{s\geq 0:\int_\R L^B_s(x)\mu^{\beta,\lambda}(\dd x)>t\right\},
\end{equation}
and define a process $Z^{\beta,\lambda}=(Z^{\beta,\lambda}_t)_{t\geq 0}$ by
\begin{equation}\label{zdef}
Z^{\beta,\lambda}_t:=\left(S^{\beta,\lambda}\right)^{-1}\left(B_{H^{\beta,\lambda}_t}\right).
\end{equation}
The process $Z^{\beta,\lambda}$ is obtained by applying a time change and a scale transformation to Brownian motion. Thus it can be regarded as the one-dimensional diffusion process with scale function $S^{\beta,\lambda}$ and speed measure $e^{2\lambda r/\rho}\dd r$, but in the generalised sense of \cite{MR984748} since the scale function is not continuous. That $Z^{\beta,\lambda}$ is continuous will be checked below, as will the fact that, conditional on $S^{\beta,\lambda}$, it is Markov when started from 0 (see Lemma \ref{zbllem}). As is already mentioned in \cite{MR984748}, however, such a generalised process may not have the strong Markov property, and indeed does not in the present case. (This is in contrast to the process $(B_{H^{\beta,\lambda}_t})_{t\geq 0}$, which is strong Markov.) We elaborate on this in Section \ref{sec:simulation} with some simulations.

\begin{remark}
\label{rem:CSvsVS}
The process $X^{\beta,\lambda}$ has unit mean exponential holding times, and so its long-time behaviour closely matches the discrete-time random walk with transition probabilities given by $c^{\beta,\lambda}(\omega_i,\omega_j)/c^{\beta,\lambda}(\omega_i)$. In the earlier works such as \cite{FGS, FSS}, it has been common to study the so-called variable-speed random walk, whose jump rate from $\omega_i$ to $\omega_j$ is given by
\begin{equation*}
\exp\left(-|\omega_i-\omega_j|-\beta U(E_i,E_j)+\lambda(\omega_j-\omega_i)\right)=c^{\beta,\lambda}(\omega_i,\omega_j)e^{-2\lambda \omega_i}.
\end{equation*}
The conclusion of Theorem~\ref{thm:main} holds for this variant, except that in the description of limiting process, the factor $\mathbf{E}(c^{\beta,0}(\omega_0))$ in~\eqref{eq:mudef} is replaced by one. In fact, the proof for this case would be simpler; see Remark~\ref{rem:VSproof} for further discussion on this point.
\end{remark}

\begin{remark}
The paper \cite{CF} contains results for the model where the spatial separation term $|\omega_i-\omega_j|$ in \eqref{cdef} is replaced by $|\omega_i-\omega_j|^\alpha$. When $\alpha<1$, for any density $\rho>0$, the quenched limit at \eqref{homog} is shown to hold with $\sigma^2>0$. On the other hand, when $\alpha>1$, sub-diffusivity is observed (specifically, the annealed limit at \eqref{homog} is trivial). It is an interesting question to determine the asymptotic behaviour of the Mott random walk in the latter case. We conjecture that the qualitative `blocking' behaviour of the model is similar to, but more extreme than, that seen in the present article, and plan to describe this precisely in a future work. To provide further context for these comments, we note that in higher dimensions, the qualitative behaviour of the symmetric version of the model does not depend on $\rho$ and $\alpha$, with quenched homogenisation occurring regardless of the particular value of these parameters \cite{CFP}.
\end{remark}

\begin{remark}
In the case of a non-vanishing bias, ballisticity/sub-ballisticity for the Mott random walk is explored in \cite{FGS}. See also the related work \cite{BS}, which identifies the appropriate scaling in the sub-ballistic phase. According to the later work \cite[Remark 1.3]{BS2}, it should be possible to identify the scaling limit.
\end{remark}

\subsection{Comment on the method}\label{sec:method}

Let us briefly comment on the method used in the proof of Theorem \ref{thm:main}. This has a twofold aim. First, it elucidates how the limiting process arises. Second, it indicates a more general theory behind our proof that is applicable to other problems in random media.

To these ends, we appeal to the well-known connection between random walks and electrical networks. In the present model, we can view $\omega$ as nodes in a resistor network, where the resistance of edge $\{\omega_i,\omega_j\}$ is given by
\begin{align*}
r^{\beta,\lambda/n}(\omega_i,\omega_j):=c^{\beta,\lambda/n}(\omega_i,\omega_j)^{-1}.
\end{align*}
Moreover, the effective resistance between disjoint sets $A,B\subseteq \omega$ is defined by
\begin{align}
\lefteqn{R^{\beta,\lambda/n}(A,B)^{-1}}\label{effres}\\
&:=\inf\left\{\frac 12\sum_{i,j} c^{\beta,\lambda/n}(\omega_i,\omega_j)\left(f(\omega_i)-f(\omega_j)\right)^2:\:f\colon \omega\to [0,1]\text{ with }f|_A\equiv 0,\:f|_B\equiv 1\right\}.\nonumber
\end{align}
In the following, we write $R^{\beta,\lambda/n}(x,y)$ in place of $R^{\beta,\lambda/n}(\{x\},\{y\})$. (Later in the article, we consider other examples of graphs equipped with symmetric conductances, and the effective resistance upon these is defined similarly.) By standard theory, the restriction of $R$ to singleton sets defines a metric on $\omega$ (see \cite[Theorem 2.64]{Barlowbook} or \cite[Theorem 1.6]{Kigdendrite}, for example). Moreover, it turns out that our model is close enough to the one-dimensional setting for the function $R^{\beta,\lambda/n}(0,x)$ to take on the role of a scale function in the theory of one-dimensional diffusion. In particular, we have that $\sign(X_k)R^{\beta,\lambda/n}(0,X_k)$, which is the random walk on the \emph{resistance space} $(\omega, R^{\beta,\lambda/n}(\cdot,\cdot))$, behaves approximately like a time-changed Brownian motion. (As usual $\sign(x):=1$ for $x>0$, $\sign(x):=-1$ for $x<0$, and $\sign(0):=0$.) As a consequence, the limiting process will be determined once we understand the scaling limit of the effective resistance and the invariant measure.

Now, to begin with the symmetric ($\lambda=0$) and infinite temperature ($\beta=0$) case, it is straightforward to observe that the nearest-neighbor resistance is heavy-tailed:
\begin{align}\label{eq:heavy}
\mathbf{P}\left(r^{0,0}(\omega_0,\omega_1)\geq u\right)=\mathbf{P}\left(\omega_1-\omega_0\geq \log u\right)=u^{-\rho}.
\end{align}
Since the collection $(r^{0,0}(\omega_i,\omega_{i+1}))_{i\in\Z}$ is i.i.d., the resistance along the nearest-neighbor path between $\omega_0$ and $\omega_{\floor{tn}}$ is therefore, after suitable normalization, well-approximated by a $\rho$-stable process. With some additional work to take into account the non-nearest neighbor edges in the model, the energy marks, and the non-zero bias, we establish in Theorem \ref{thm:approx} that the rescaled resistances
\begin{equation}\label{rrr}
\left(n^{-1/\rho}\sign(v-u)R^{\beta,\lambda/n}(\omega_{\floor{un}},\omega_{\floor{vn}})\right)_{u,v\in\mathbb{R}}
\end{equation}
converge to the increment process $(S^{\beta,\lambda}(v)-S^{\beta,\lambda}(u))_{u,v\in\mathbb{R}}$, which is a precise statement of the intuition that the jumps of the tilted L\'evy process capture the asymptotic inhomogeneity in the resistance environment. We stress that the incorporation of non-nearest neighbor edges in particular is by no means trivial. For instance, we cannot simply cut all non-nearest neighbor edges as they affect the scaling limit through the constant $C_\beta$ that appears in the L\'evy measure in \eqref{Levy}.

As for the invariant measure, which places mass $c^{\beta,\lambda/n}(\omega_i)$ at site $\omega_i$, one can readily show that on the `physical' space $(\R, |\cdot|)$, it converges under scaling to the measure $\mathbf{E}(c^{\beta,0}(\omega_0)) e^{2\lambda r/\rho}\dd r$, see Theorem \ref{thm:measure}.

Putting these conclusions together, we can determine the limiting process as follows: First, deform the space by changing the metric to $R^{\beta,\lambda/n}(\cdot,\cdot)$. On this new `resistance' space, the invariant measure is approximated by $\mu^{\beta,\lambda}$ and thus the process behaves approximately like a Brownian motion time-changed by $\mu^{\beta,\lambda}$, which is $(B_{H^{\beta,\lambda}_t})_{t\geq 0}$ in our above notation. Reverting back to `physical space' requires the reversal of the resistance scaling, and thus leads us to see that $Z^{\beta,\lambda}$ should be the limiting process.

To make these steps precise, we appeal to the recent general result of \cite{croydon2018}, which is based on the theory of resistance forms initiated and developed by Kigami (see Section~\ref{sec:mr} for more details). Roughly speaking, in a result that is particularly well-suited to `low-dimensional' settings, \cite{croydon2018} shows that if the resistance metric associated with a random walk and its invariant measure suitably converge, then so does the random walk. (See also the closely related \cite{CHK}.) Thus the question of the scaling limit of a stochastic process is reduced to a question about the convergence of metric measure spaces. Despite there being various classical results relating convergence of scale functions and speed measures to that of one-dimensional processes, such as \cite{stone}, we find this recent resistance form approach useful because the Mott walk is not a genuinely one-dimensional process.

\begin{remark}\label{gaps}
In this article, we only consider the case when $\omega$ is given by a Poisson point process of intensity $\rho\in(0,1)$ for convenience. Indeed, the same arguments would also apply to other configurations for which the distribution of gaps between sites has a suitably heavy tail. More precisely, if $(\omega_{i+1}-\omega_i)_{i\in\mathbb{Z}}$ are i.i.d.\ and $e^{\omega_{i+1}-\omega_i}$ has an infinite mean and falls into the domain of attraction of a $\rho$-stable random variable with $\rho\in(0,1]$, then a modified version of the conclusion of Theorem \ref{thm:main} will hold. The main difference would be that the statement corresponding to \eqref{rrr} for the scaling of the resistance would be in terms of $n^{-1/\rho}\ell(n)^{-1}$ for some slowly varying function $\ell$, and as a consequence, the correct scaling of the Mott random walk would be given by
\[\left(n^{-1}X^{\beta,\lambda/n}_{n^{1+1/\rho}\ell(n)t}\right)_{t\geq 0}.\]
Note that this more general statement would include the case when $\omega$ is given by a Poisson point process of intensity $\rho=1$, with $\ell(n)=\log n$.
\end{remark}

\begin{remark}\label{homogrem}
Similarly to the previous remark, if $(\omega_{i+1}-\omega_i)_{i\in\mathbb{Z}}$ are i.i.d.\ and $e^{\omega_{i+1}-\omega_i}$ has a finite mean, then one can recover the homogenisation result given at \eqref{homog} by the argument of this paper. Indeed, in this case, one can prove that
\[\left(n^{-1}\sign(v-u)R^{\beta,0}(\omega_{\floor{un}},\omega_{\floor{vn}})\right)_{u,v\in\mathbb{R}}\]
converges almost-surely to $(C(v-u))_{u,v\in\mathbb{R}}$ for some deterministic constant $C\in(0,\infty)$. It follows that, for almost-every environment, the quenched laws of $(n^{-1}X^{\beta,0}_{n^{2}t})_{t\geq 0}$ converge to that of a one-dimensional Brownian motion with a non-trivial, deterministic diffusion constant. This argument covers the case when $\omega$ is given by a Poisson point process of intensity $\rho>1$. We provide the details in the \hyperref[homogsec]{Appendix}.
\end{remark}

\subsection{An extension with random holding times}\label{sec:extension} That the process $Z^{\beta,0}$ can be regarded as a generalised one-dimensional diffusion with scale function $S^{\beta,0}$ and Lebesgue speed measure makes it a something of a dual to the Fontes-Isopi-Newman (FIN) diffusion of \cite{FIN}, which is a process in natural scale and with purely atomic speed measure, where the sizes and positions of atoms are given by the jumps of a subordinator. The latter process arises naturally as the scaling limit of the Bouchaud trap model on $\mathbb{Z}$, the simplest case of which is a symmetric continuous time random walk with spatially inhomogeneous holding times whose means obey a heavy-tailed distribution. It is straightforward to generalise the Mott random walk to include both the `blocking' described in Remark \ref{rem:blocking} and the `trapping' of the FIN diffusion. Indeed, suppose that $\tau=(\tau_i)_{i\in\mathbb{Z}}$ is a sequence of i.i.d.\ random variables, independent of $(\omega,E)$, satisfying
\begin{equation}\label{holding}
\mathbf{P}\left(\tau_i\geq t\right)=t^{-\kappa}
\end{equation}
for $t\geq 1$, where $\kappa\in(0,1)$ is a fixed parameter. Given $(\omega,E,\tau)$, consider the continuous-time Markov chain $\tilde{X}^{\beta,\lambda}=(\tilde{X}^{\beta,\lambda}_t)_{t\geq 0}$ on $\omega$ with generator given by
\[(\tilde{L}^{\beta,\lambda}f)(\omega_i):=\frac{1}{\tau_i}\sum_{j\in\Z}\frac{c^{\beta,\lambda}(\omega_i,\omega_j)}{c^{\beta,\lambda}(\omega_i)}\left(f(\omega_j)-f(\omega_i)\right).\]
We will write the quenched and annealed laws of this process as $\tilde{P}^{\beta,\lambda}$ and $\tilde{\mathbb{P}}^{\beta,\lambda}$, respectively. Note that, under its quenched law, $\tilde{X}^{\beta,\lambda}$ is simply a time change of ${X}^{\beta,\lambda}$, with holding times at a site $i$ having mean $\tau_i$, rather than 1. We have the following generalisation of Theorem \ref{thm:main}.

\begin{theorem}\label{thm:main2}
For every $\rho,\kappa\in(0,1)$ and $\beta, \lambda\geq 0$, it holds that as $n\to\infty$,
\[\tilde{\mathbb{P}}^{\beta,\lambda/n}\left((n^{-1}\tilde{X}_{n^{1/\kappa+1/\rho}t})_{t\geq 0}\in\cdot\right)\]
converge weakly as probability measures on $D([0,\infty),\mathbb{R})$ to the law of the continuous process $\tilde{Z}^{\beta,\lambda}$ defined below.
\end{theorem}

To describe the scaling limit, let $S^{\beta,\lambda}$ and $B$ be as before. Additionally, independent of these, let $S^\kappa$ denote a two-sided subordinator with L\'evy measure given by $\kappa x^{-\kappa-1}\mathbf{1}_{\{x>0\}}\dd x$, and define
\[\tilde{\mu}^{\beta,\lambda}\left((a,b]\right):=\mathbf{E}\left(c^{\beta,0}(\omega_0)^\kappa\right)\int_{(S^{\beta,\lambda})^{-1}(a)}^{(S^{\beta,\lambda})^{-1}(b)}e^{2\lambda r/\rho}\dd S^\kappa( r).\]
Next, analogously to \eqref{hdef}, suppose $\tilde{H}^{\beta,\lambda}$ is given by $B$ and $\tilde{\mu}^{\beta,\lambda}$, and, similarly to \eqref{zdef}, set
\[\tilde{Z}^{\beta,\lambda}_t:=\left(S^{\beta,\lambda}\right)^{-1}\left(B_{\tilde{H}^{\beta,\lambda}_t}\right).\]
A simulation of the limiting process will be given in Section \ref{sec:simulation}.

\begin{remark}
It should be possible, and in fact simpler, to show that (up to constant factors) the limiting process $\tilde{Z}^{0,0}$ is the scaling limit of a constant-speed version of a one-dimensional nearest-neighbor random conductance model, where the individual edge resistances $(r(i,i+1))_{i\in\mathbb{Z}}$ are i.i.d.\ and satisfy
\begin{equation}\label{rrrrr}
\mathbf{P}\left(r(i,i+1)\geq u\right)\sim u^{-\rho},\qquad \mathbf{P}\left(r(i,i+1)\leq u^{-1}\right)\sim u^{-\kappa}.
\end{equation}
Indeed, that the tail at zero of the edge resistances gives the same trapping behaviour as holding times with a tail of the form \eqref{holding} can be seen by comparing the FIN diffusion scaling limit of the random conductance model that appears in \cite{Cerny} with the original result of Fontes, Isopi and Newman \cite{FIN}. Moreover, one could readily see $\tilde{Z}^{\beta,\lambda}$ (again, with suitably modified constants) as a scaling limit by adding a tilt to the resistances, whereby $r(i,i+1)$ is replaced by $e^{-2\lambda i/n}r(i,i+1)$ in the scale $n$ model, and incorporating some version of energy marks. The description of both `blocking' and `trapping' in the scaling limit $\tilde{Z}^{\beta,\lambda}$ gives a (near-)symmetric analogue to the `walls' and `wells' seen in the non-vanishing bias case considered in \cite{BS,BS2}.

Similar to Remark \ref{rem:CSvsVS}, one might also consider the corresponding variable-speed version of the one-dimensional nearest-neighbor random conductance model with the distribution of $r(i,i+1)$ having tails at 0 as on the right-hand side at \eqref{rrrrr}. This was studied in \cite{KK}, where it was shown that $Z^{0,0}$ (with suitably modified constants in \eqref{Levy} and \eqref{eq:mudef}) was the scaling limit. The same result could be recovered using the techniques of this article.
\end{remark}

\subsection{Quenched fluctuations}

In the definition of the limit process $Z^{\beta,0}$, it is clear that the subordinator $S^{\beta,0}$ can be interpreted as the contribution of the random environment, while the Brownian motion $B$ corresponds to the random walk. One might therefore be tempted to conjecture that one can construct a coupling between $\omega$ and $S^{\beta,0}$ such that the quenched law of $(n^{-1}X_{n^{1+1/\rho}t})_{t\geq 0}$ converges weakly to $\mathbb P((Z^{\beta,0})_{t\geq 0}\in\cdot\:|\: S^{\beta,0})$, for almost all $\omega$. A moment of thought reveals that this is false. Indeed, the subordinator $S$ is obtained as the scaling limit of the effective resistance, which, as we described above, behaves like a sum of i.i.d.\ heavy-tailed random variables, and hence the convergence exhibits LIL-type fluctuations. More precisely, almost-surely there exists a (random) subsequence $(n_k(\omega))_{k\in\N}$ such that $R(\omega_0,\omega_{n_k})$ is either atypically large or small. In the case of an atypically small resistance, we expect that the random walk diffuses faster than predicted by Theorem \ref{thm:main}, and indeed this is the case.

\begin{proposition}\label{prop:quenched}
Let $\sigma_n=\inf\{t\geq 0:X(t)\in\{...,\omega_{-n-1},\omega_{-n}\}\cup\{\omega_n,\omega_{n+1},...\}\}$.  There exists a constant $M>0$ such that, $\mathbf P$-a.s.,
\begin{align*}
\limsup_{n\to\infty}P^{\beta,0}\left(\sigma_n\leq \frac{Mn^{\frac{1}{\rho}+1}}{\log\log^{\frac{1}{\rho}-1} n}\right)>0.
\end{align*}
\end{proposition}

It readily follows that there cannot be a quenched version of Theorem \ref{thm:main}. Since the proof of the subsequent result is straightforward given Proposition \ref{prop:quenched}, it is omitted.

\begin{corollary}
$\mathbf P$-a.s., the sequence $(P^{\beta,0}((n^{-1}X_{n^{1+1/\rho}t})_{t\geq 0}\in\cdot))_{n\in\N}$ is not tight.
\end{corollary}

\begin{remark}
For a sum of i.i.d.\ heavy-tailed random variables with a tail as considered here, it is known that the order of the poly-loglogarithmic fluctuation in Proposition \ref{prop:quenched} is optimal, see \cite[Theorem 1]{lipschutz}. We therefore conjecture that, almost-surely, there are no exceptional times where the random walk moves faster than in Proposition \ref{prop:quenched}, i.e. that, $\mathbf{P}$-a.s.,
\begin{align*}
\limsup_{n\to\infty}P^{\beta,0}(\sigma_n\leq f_n)=0
\end{align*}
whenever $\lim_{n\to\infty}f_n\frac{\log\log^{\frac{1}{\rho}-1} n}{n^{\frac{1}{\rho}+1}}=0$. To prove such a result, it would be helpful to establish more quantitative statements connecting the effective resistance and the approximating i.i.d.\ process than those proved in this paper (see \eqref{eq:PMprobconv} and \eqref{eq:Pprobconv}). We leave this as a problem for further research.
\end{remark}

\begin{remark}
We further conjecture that there are exceptional times where the random walk diffuses slower than expected, i.e. that, $\mathbf P$-a.s.,
\begin{align*}
\limsup_{n\to\infty}P^{\beta,0}(\sigma_n>Kn^{1+1/\rho})=1
\end{align*}
for every $K>0$. We expect that such a slowdown is caused by an atypically large resistance, and hence one needs to examine the upper deviations in the LIL for sums of heavy-tailed random variables. In the i.i.d.\ case, these are of poly-logarithmic order, rather than the poly-loglogarithmic order seen in the lower deviations (again, see \cite{lipschutz}). In the Mott random walk model, for $\sigma_n$ to be large, one would need to control the resistance on both the left- and right-hand sides of the origin simultaneously, and thus it is not immediately clear how the i.i.d.\ result transfers. As a result, we do not have a precise prediction for the correct order of the slowdown. (By contrast, if one were to consider instead the hitting time $\sigma^+_n=\inf\{t\geq 0:X(t)\in\{\omega_n,\omega_{n+1},...\}\}$, it might be reasonable to expect a quenched poly-logarithmic fluctuation of the same order as in the case of the corresponding i.i.d.\ sum.)
\end{remark}

\subsection{Conjecture on aging}

In the one-dimensional Bouchaud trap model with heavy-tailed holding time means, one sees an aging phenomenon, whereby there is a non-trivial probability that the random walker will be found in the current trap after a time interval that is of the order of the length of time for which the system has already been running \cite{BaC,FIN}. Moreover, as is discussed in \cite{BaC2}, this property is natural in the context of models with trapping more generally. For the Mott random walk considered in the present article, the same version of the aging property will not apply. However, we expect that the running maximum of the process will exhibit the following aging behaviour.

\begin{conjecture} For every $\rho<1$ and $\beta, \lambda\geq 0$, it holds that
\begin{align}\label{eq:conj}
\lim_{n\rightarrow\infty}\mathbb{P}^{\beta,\lambda/n}\left(\sup_{s\leq n^{1+1/\rho}}X_{s}=\sup_{s\leq n^{1+1/\rho}h}X_{s}\right)=\theta(h):=\mathbb{P}\left(\sup_{s\leq 1}Z^{\beta,\lambda}_{s}=\sup_{s\leq h}Z^{\beta,\lambda}_{s}\right),\:\:\: \forall h>1,
\end{align}
where $\lim_{h\rightarrow 1}\theta(h)=1$ and $\lim_{h\rightarrow \infty}\theta(h)=0$.
\end{conjecture}

Note that the weak convergence in Theorem \ref{thm:main} is with respect to the Skorohod $J_1$-topology, which is not fine enough to immediately imply \eqref{eq:conj}. Moreover, establishing the continuity of $\theta(h)$ at $h=1$ seems delicate due to the dependence between the time change $H^{\beta,\lambda}$ and the Brownian motion $B$. We expect that both parts of the conjecture can be proved with a more careful sample path analysis of the Mott random walk than we pursue here.

\subsection{Simulations}
\label{sec:simulation}
Illustrating the above discussion, in Figures \ref{sim1} and \ref{sim2} we present some simulations of the Mott random walk, which by Theorem \ref{thm:main} approximates the process $Z^{\beta,\lambda}$. (Time runs upwards in the figures.)

We highlight that, even under its quenched law, that is, conditional on the subordinator, $Z^{\beta,\lambda}$ is not a strong Markov process, due to the `blocking' resulting from jumps in the subordinator. Indeed, the left- and right-hand sides of subordinator jump locations in physical space are separated by gaps in the support of the measure $\mu^{\beta,\lambda}$ in resistance space. Since the Brownian motion $B$ accumulates local time at each side of such a gap before returning to the other side, one sees in physical space that the process $Z^{\beta,\lambda}$ is reflected from the relevant jump location for some time before it crosses, when the reflection then occurs on the other side of the site. As a consequence, at the hitting time of a subordinator jump location, the future evolution of the process will depend upon whether this location was approached from the right or from the left. Note however that, since at any fixed time the process $Z^{\beta,\lambda}$ is almost-surely not at a subordinator jump location, it will be Markov whenever it is started away from the set of such. In particular, this will be the case when started from 0, say (see Lemma \ref{zbllem}).

\begin{figure}
\begin{minipage}{\textwidth}
\begin{center}
\begin{tikzpicture}
\node[draw,very thick,anchor=south west,inner sep=.5] at (0,0) {\includegraphics[width=.45\textwidth, height=.4\textheight]{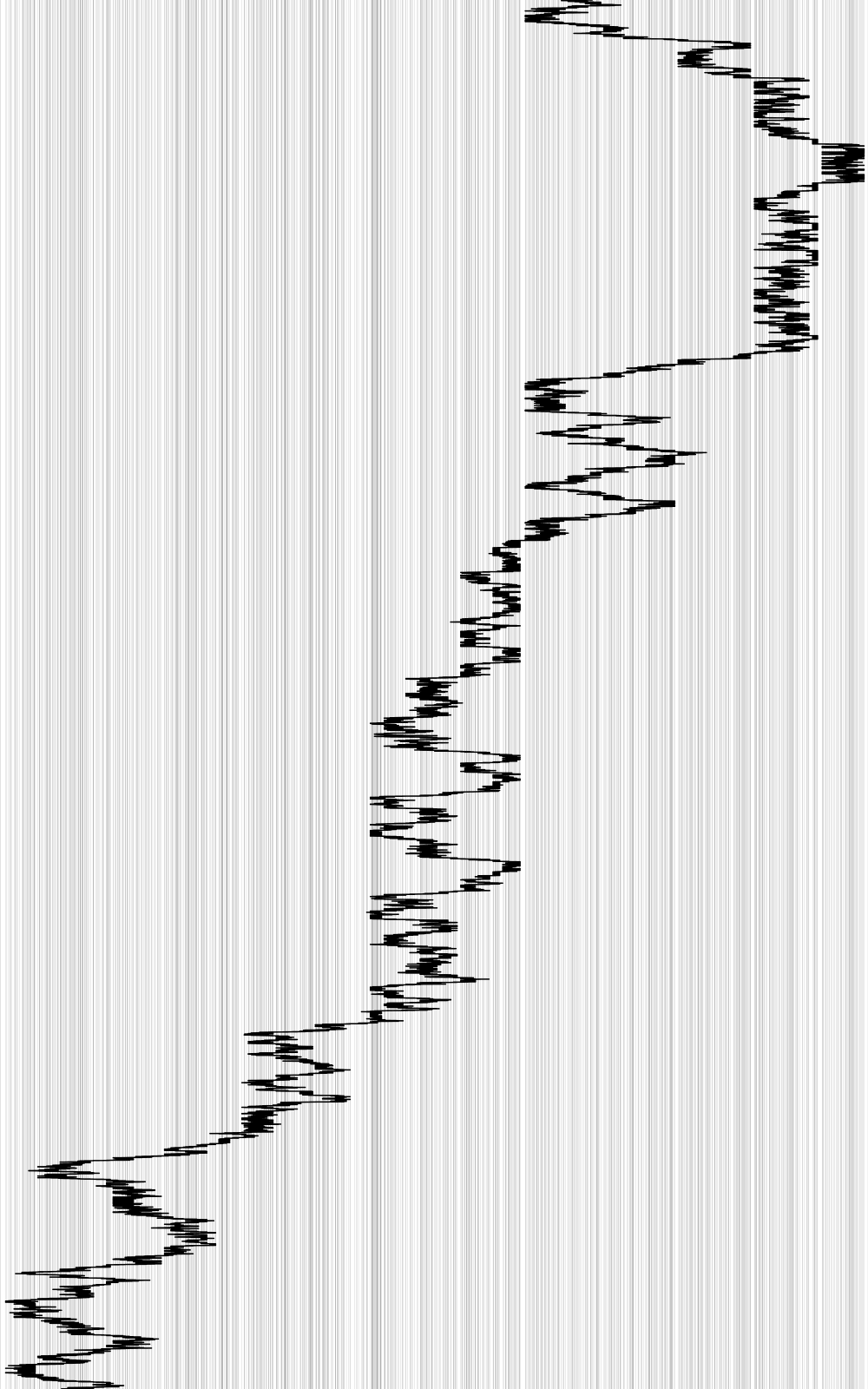}};

\node[draw,very thick,anchor=south west,inner sep=.5] at (6,0) {\includegraphics[width=.45\textwidth, height=.4\textheight]{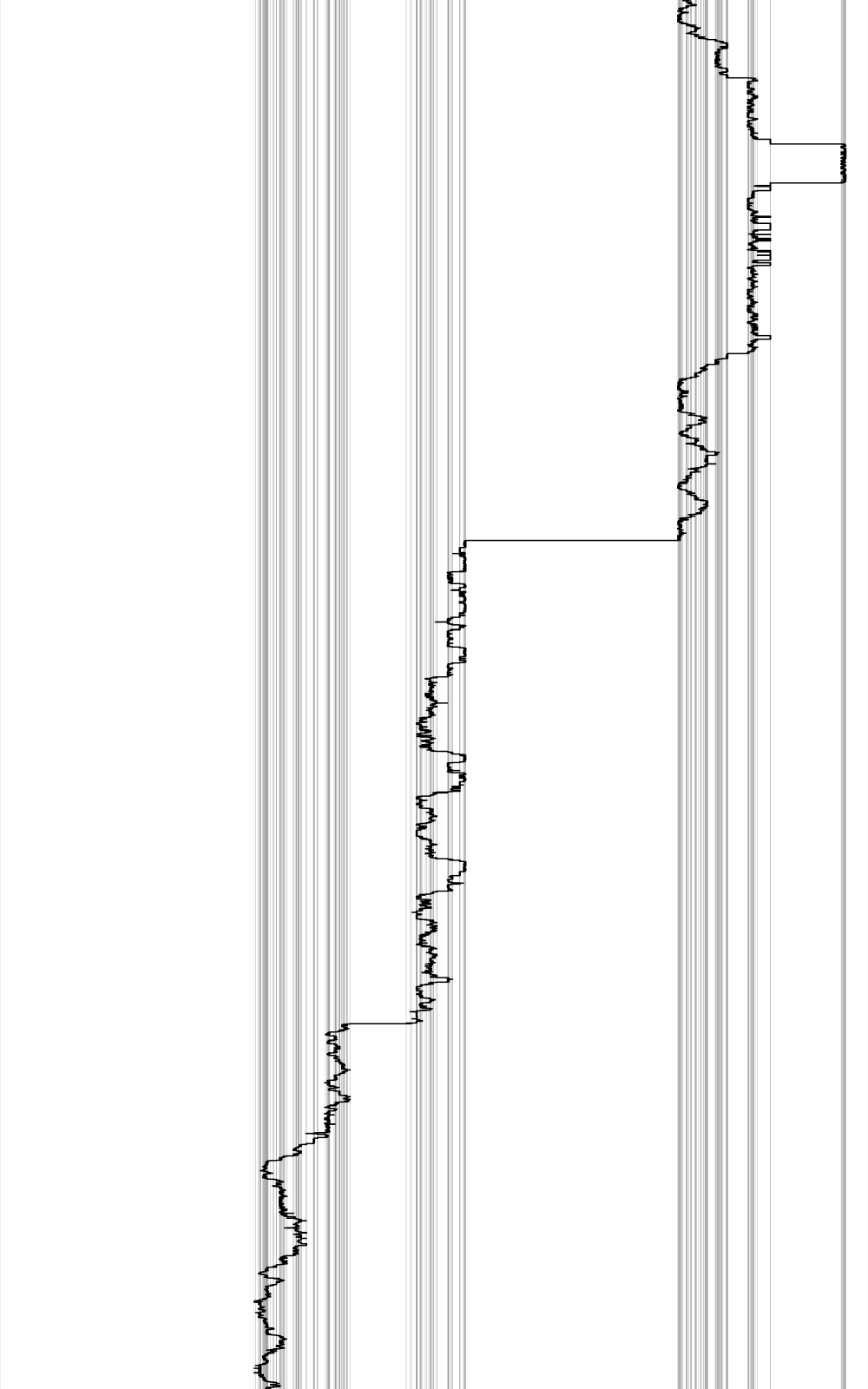}};

\node[draw,very thick,anchor=south west,inner sep=.5] at (0,-9.7) {\includegraphics[width=.45\textwidth, height=.4\textheight]{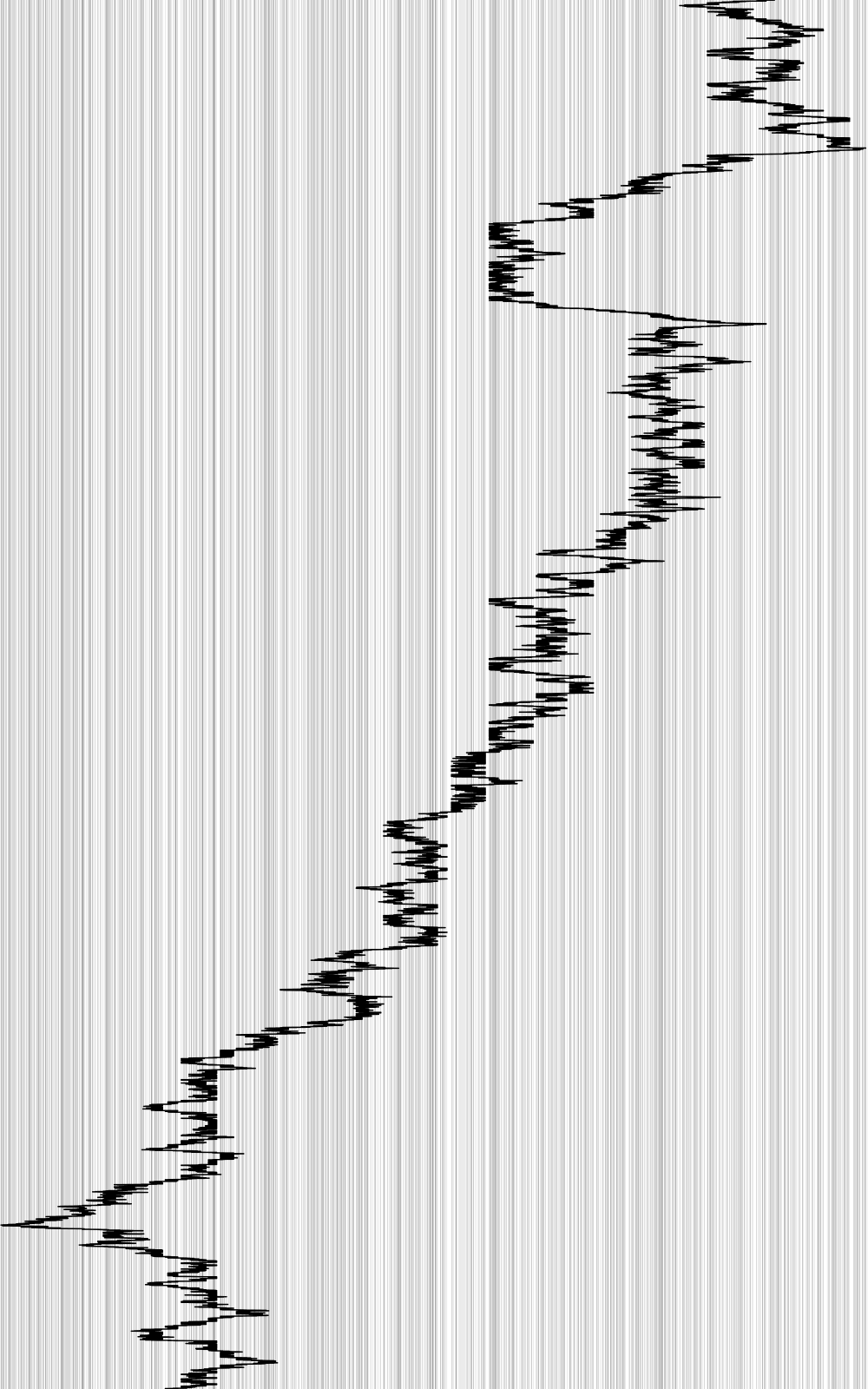}};

\node[draw,very thick,anchor=south west,inner sep=.5] at (6,-9.7) {\includegraphics[width=.45\textwidth, height=.4\textheight]{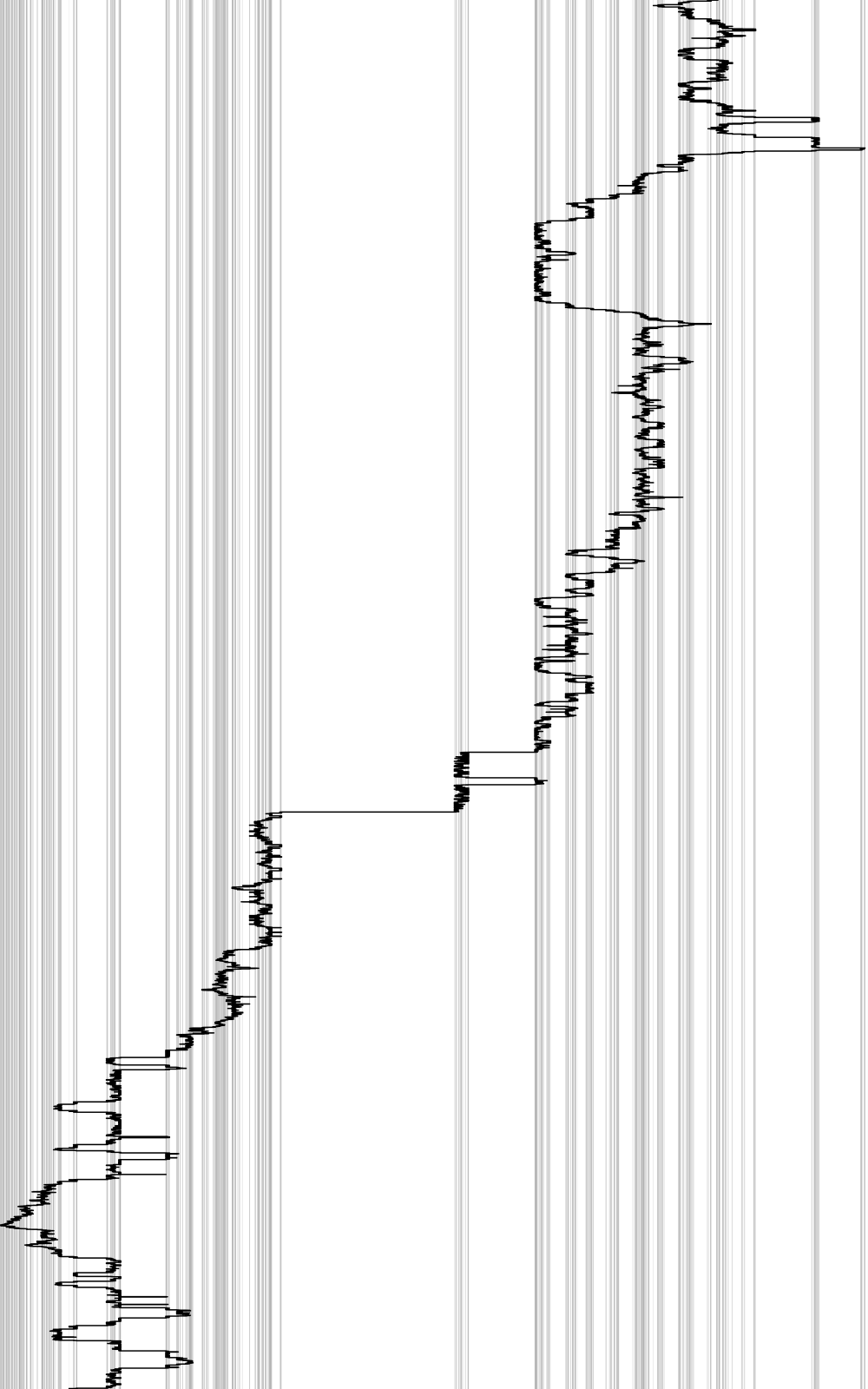}};
\end{tikzpicture}
\end{center}
\end{minipage}
\caption{Simulation of $(X_t)_{t\geq0}$ in the cases $\rho=0.7$ (top row) and $\rho=0.95$ (bottom row), for $\beta=\lambda=0$ and $3\cdot 10^6$ steps. The left column shows the process in physical space, with vertical lines indicating the environment $\{\omega_i:i\in\Z\}$. The vertical lines in the right column denote the coordinates $\{\sign(i)R^{0,0}(\omega_0,\omega_i):i\in\Z\}$ in resistance space. In resistance space, the process behaves like the trace of Brownian motion, meaning it cannot easily cross large gaps. In physical space, the gaps in the environment $\omega$ disappear, but their effect on the path is still visible.}\label{sim1}
\end{figure}

\begin{figure}
\makebox[\textwidth][c]
{
\begin{minipage}{\textwidth}
\begin{center}
\begin{tikzpicture}
\node[draw,very thick,anchor=south west,inner sep=.5] at (0,0) {\includegraphics[width=.45\textwidth, height=.25\textheight]{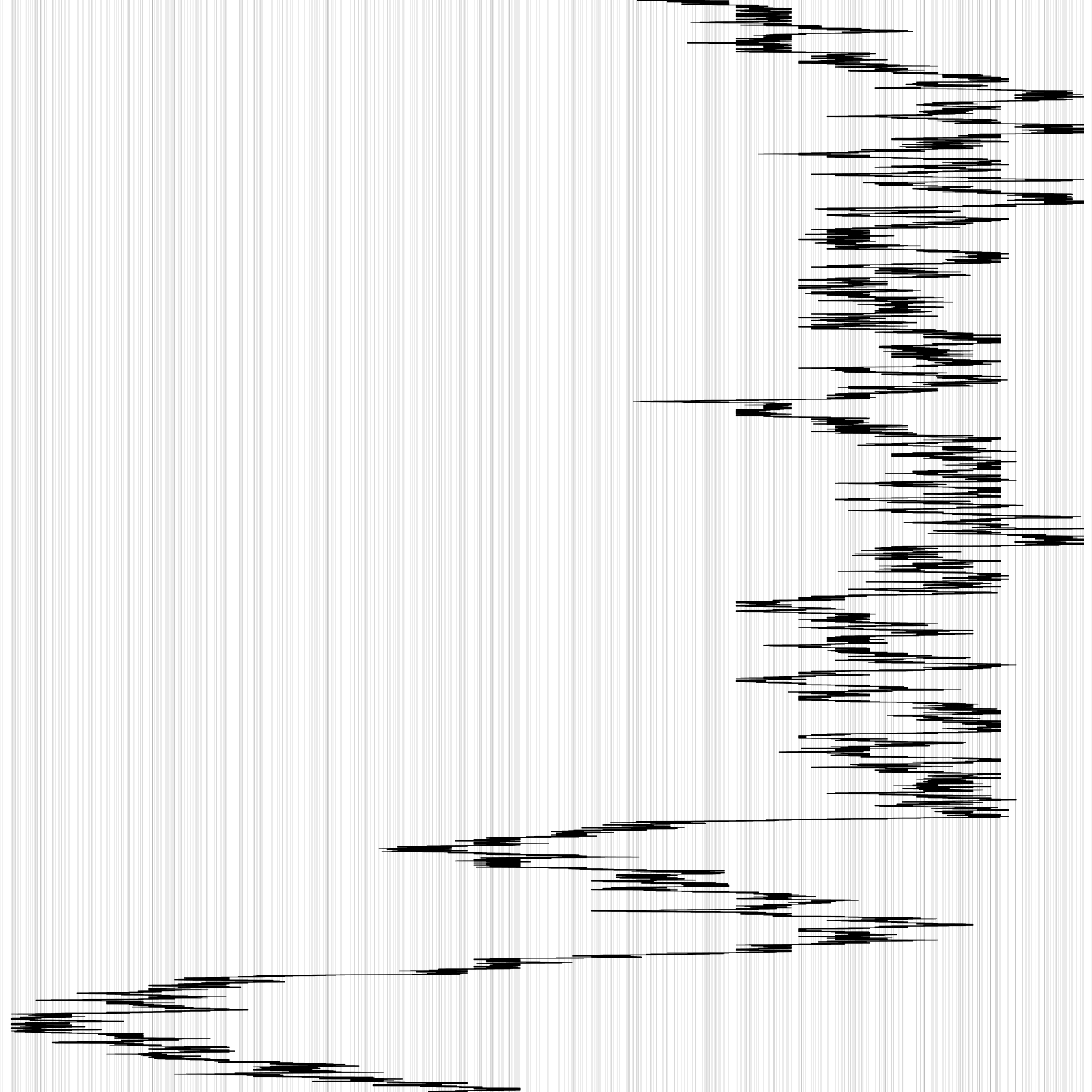}};

\node[draw,very thick,anchor=south west,inner sep=.5] at (6,0) {\includegraphics[width=.45\textwidth, height=.25\textheight]{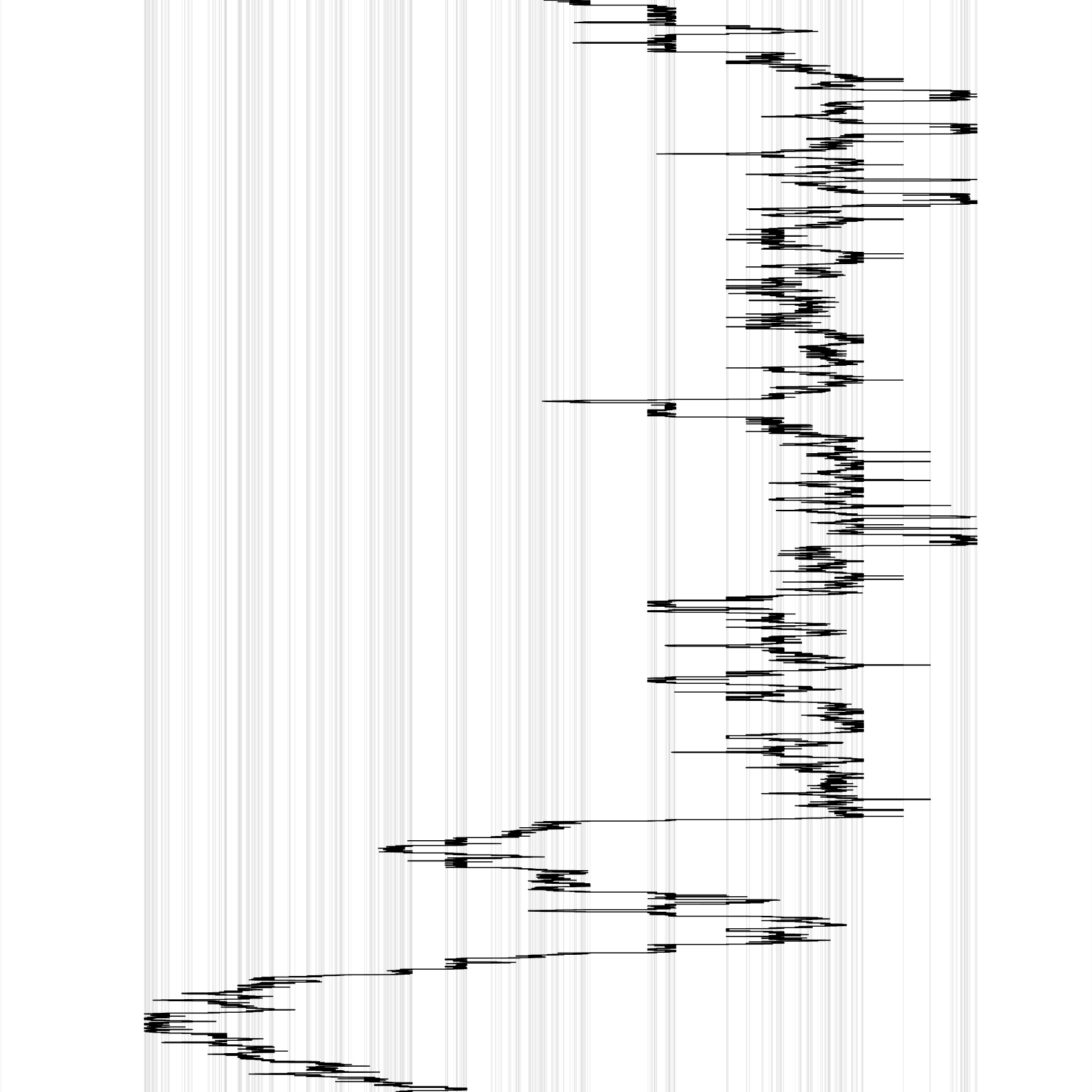}};

\node[draw,very thick,anchor=south west,inner sep=.5] at (0,-6.6) {\includegraphics[width=.45\textwidth, height=.25\textheight]{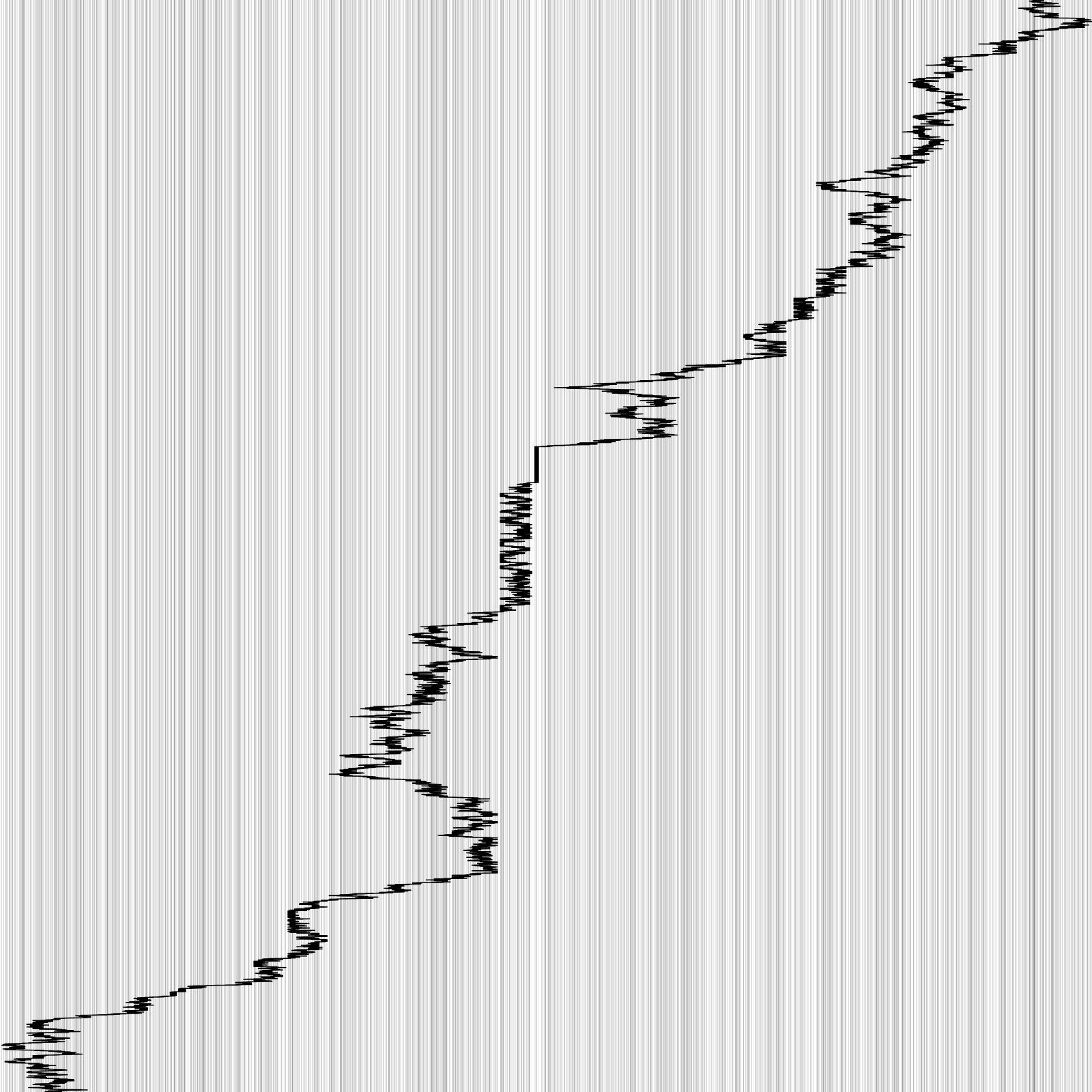}};

\node[draw,very thick,anchor=south west,inner sep=.5] at (6,-6.6) {\includegraphics[width=.45\textwidth, height=.25\textheight]{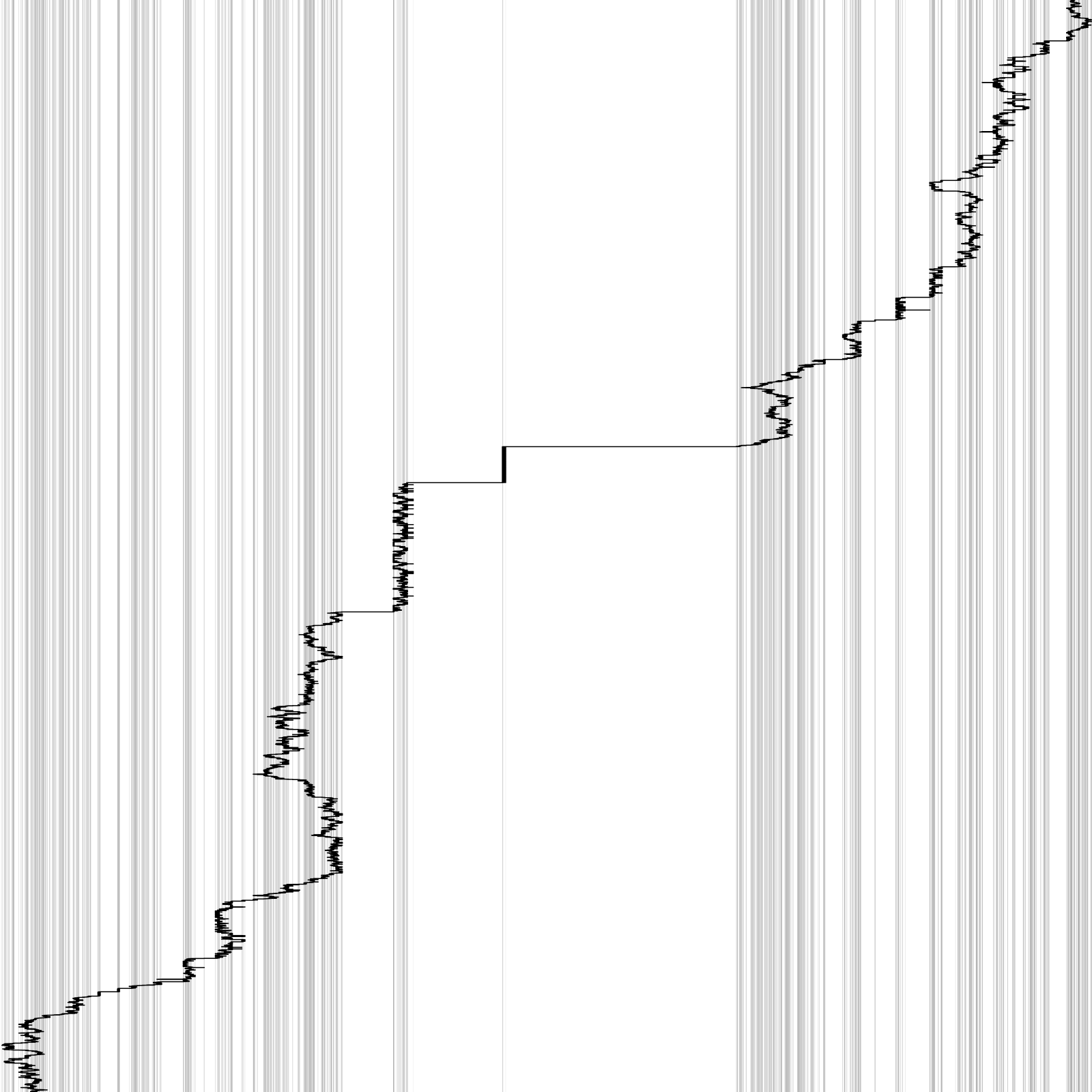}};

\node[draw,very thick,anchor=south west,inner sep=.5] at (0,-13.2) {\includegraphics[width=.45\textwidth, height=.25\textheight]{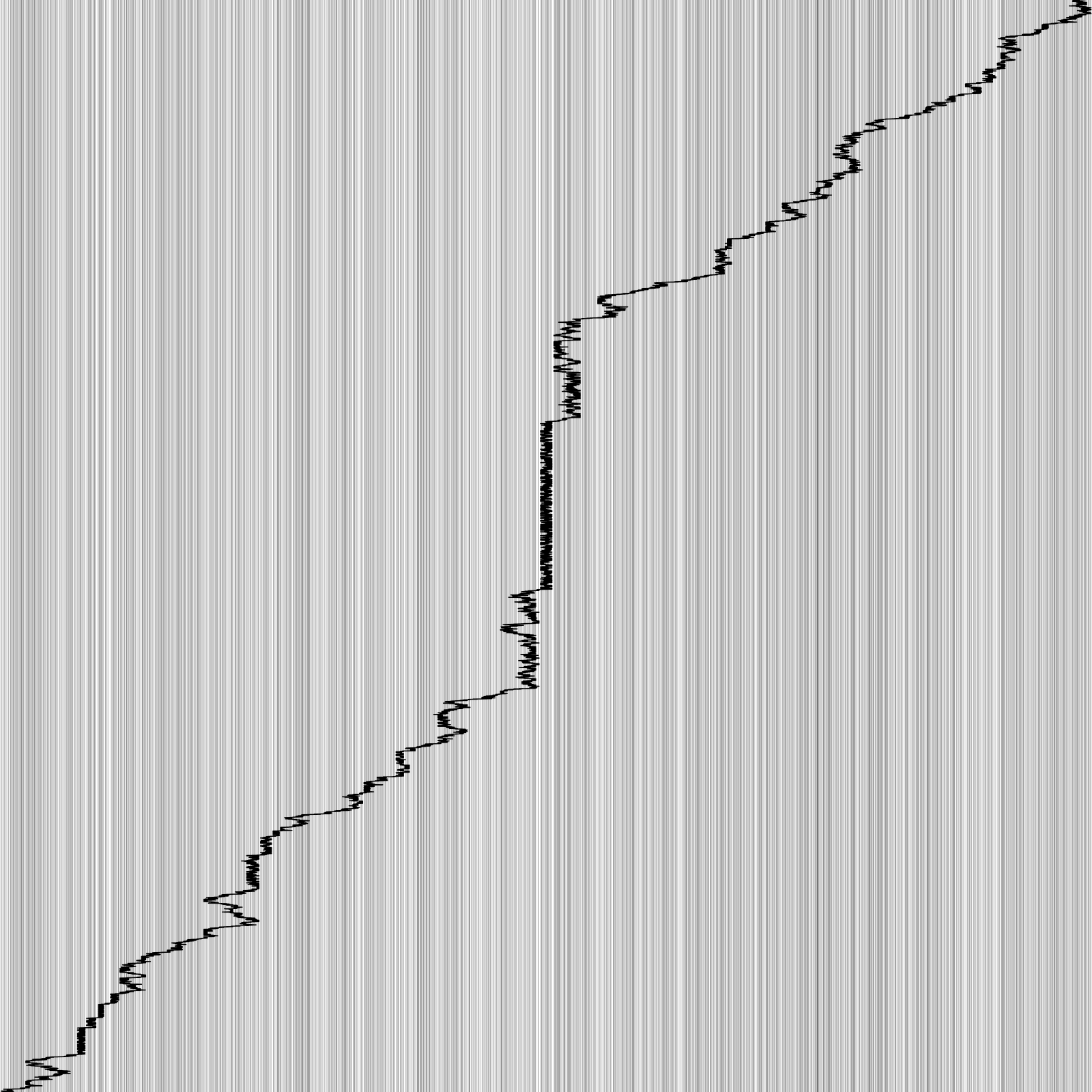}};

\node[draw,very thick,anchor=south west,inner sep=.5] at (6,-13.2) {\includegraphics[width=.45\textwidth, height=.25\textheight]{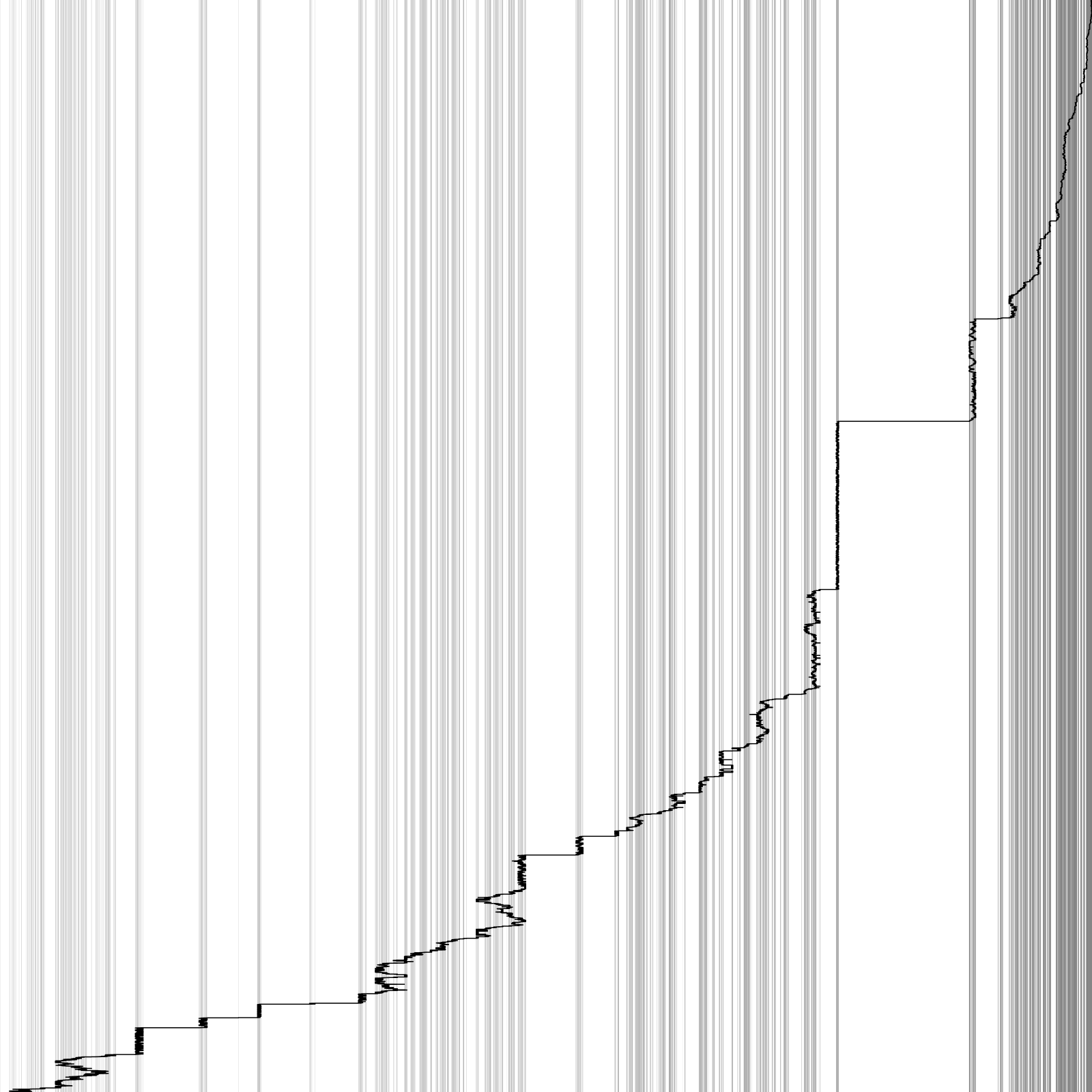}};
\end{tikzpicture}
\end{center}
\end{minipage}
}
\caption{Simulation of $(X_t)_{t\geq 0}$ in the cases $\lambda=100$ (top row), $\lambda=500$ (middle row) and $\lambda=2000$ (bottom row), with $\rho=0.9$, $\beta=0$ and $3\cdot 10^3$ steps. Note that the resistance space in bounded from the right, where vertical lines become infinitely dense. The process in resistance space still behaves like the trace of Brownian motion, but time-changed so that it slows down as it approaches the accumulation point.}\label{sim2}
\end{figure}

Figure \ref{fig:btm} shows a simulation of the random walk with random holding times from Section \ref{sec:extension}, which according to Theorem \ref{thm:main2} has $\tilde{Z}^{\beta,\lambda}$ as its scaling limit. These simulations illustrate that both the `blocking' and the `trapping' mechanisms contribute to its sub-diffusivity.

\begin{figure}
\makebox[\textwidth][c]
{

\begin{minipage}{\textwidth}
\begin{center}
\begin{tikzpicture}
\node[draw,very thick,anchor=south west,inner sep=.5] at (0,0) {\includegraphics[width=.45\textwidth, height=.23\textheight]{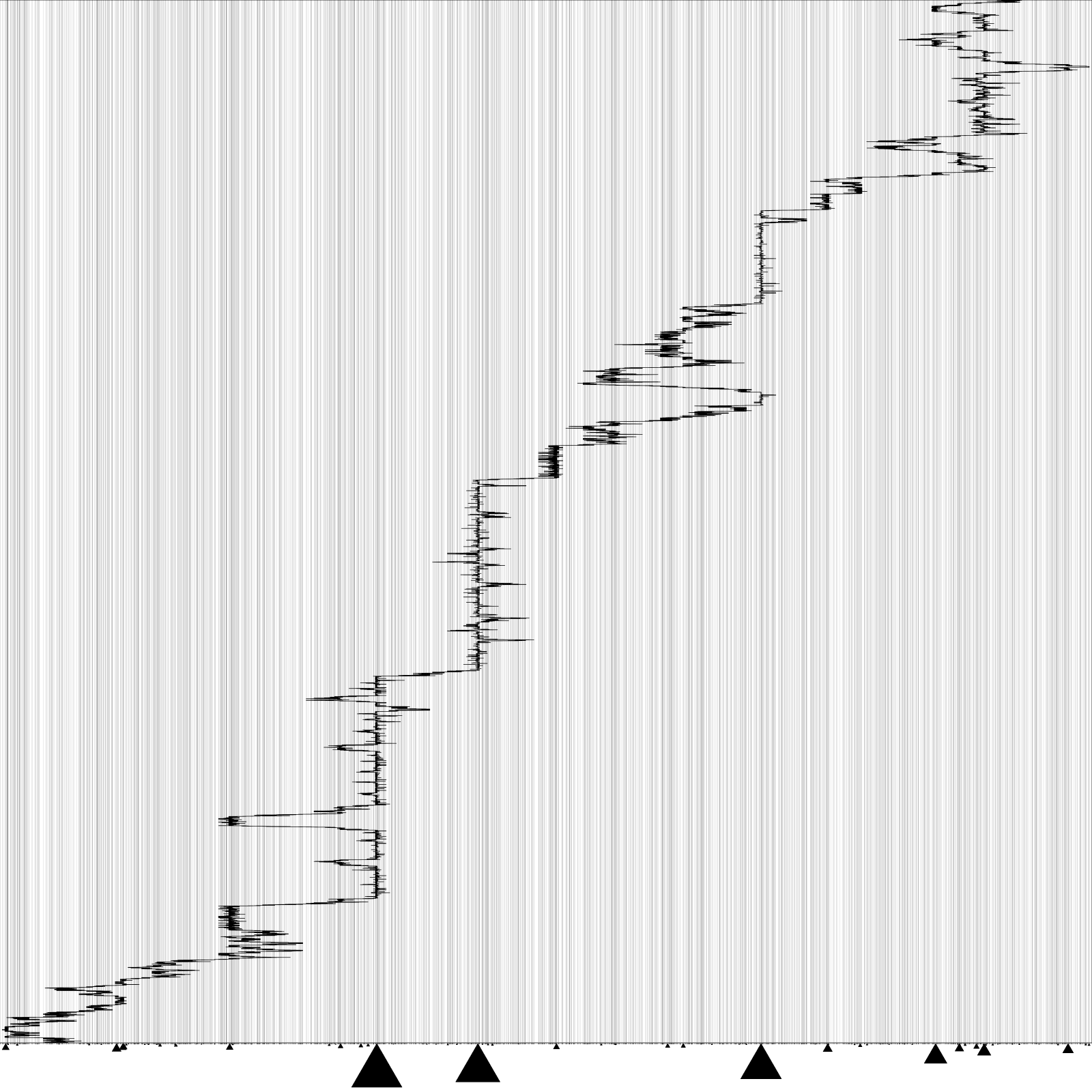}};

\node[draw,very thick,anchor=south west,inner sep=.5] at (6,0) {\includegraphics[width=.45\textwidth, height=.23\textheight]{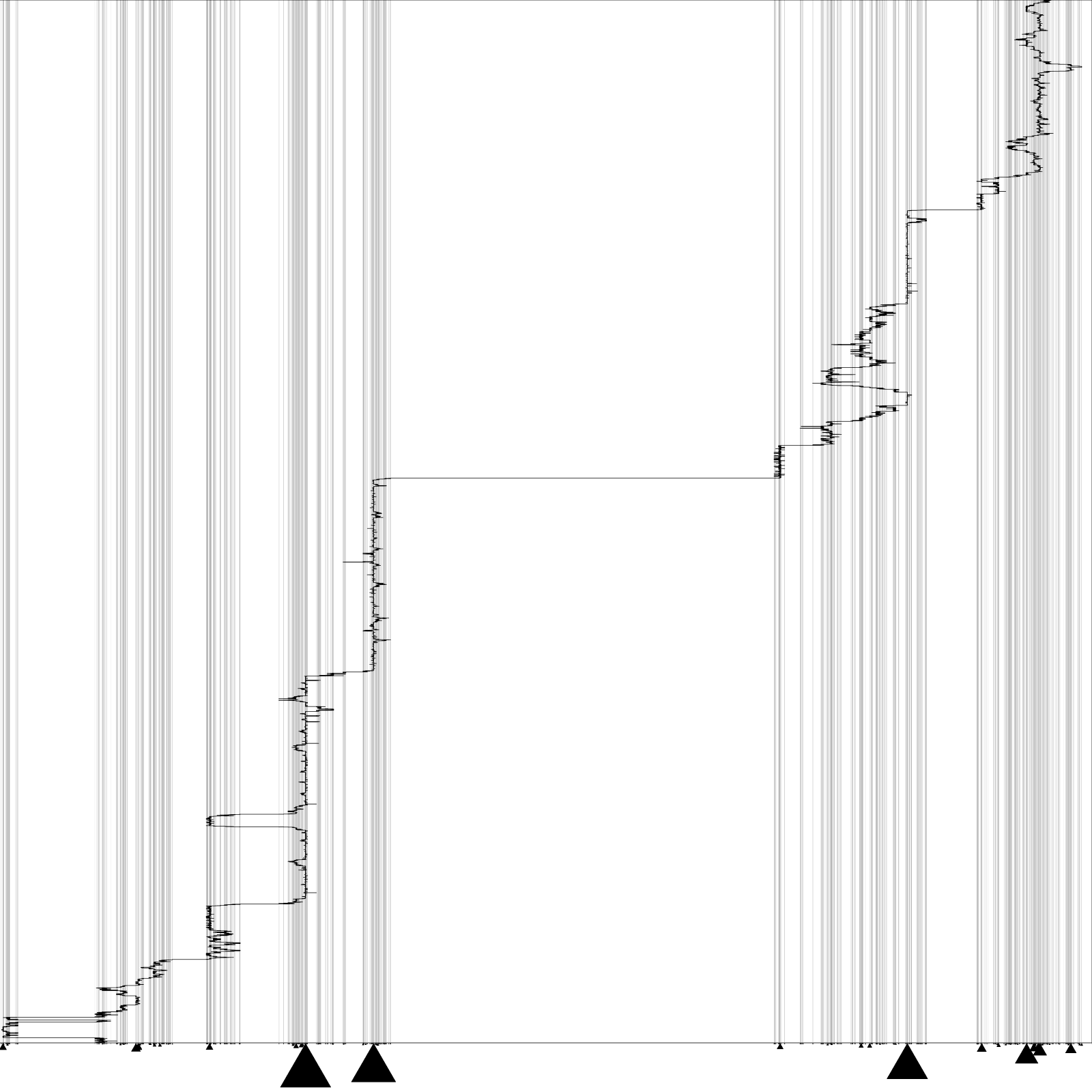}};

\node[draw,very thick,anchor=south west,inner sep=.5] at (0,-6.8) {\includegraphics[width=.45\textwidth, height=.23\textheight]{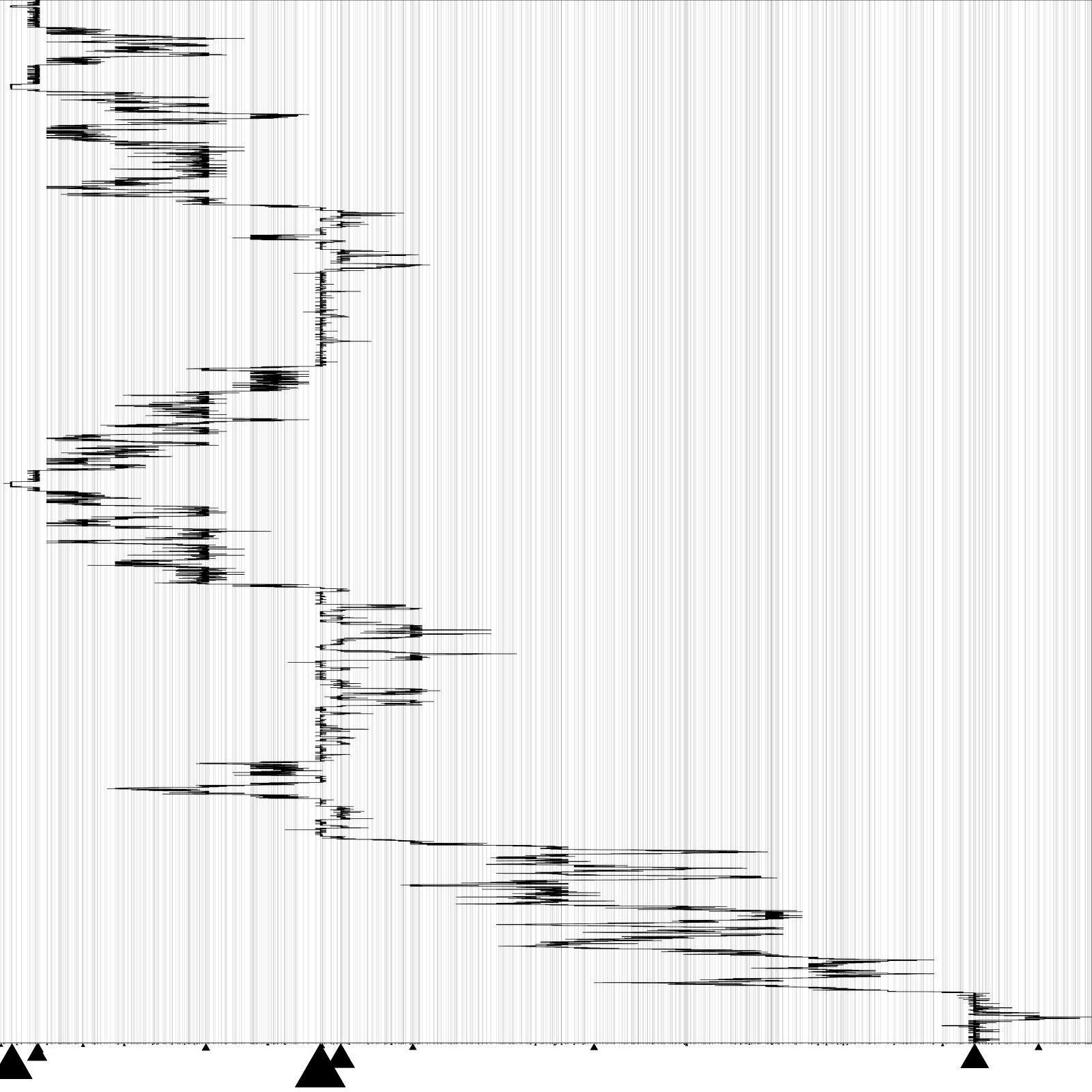}};

\node[draw,very thick,anchor=south west,inner sep=.5] at (6,-6.8) {\includegraphics[width=.45\textwidth, height=.23\textheight]{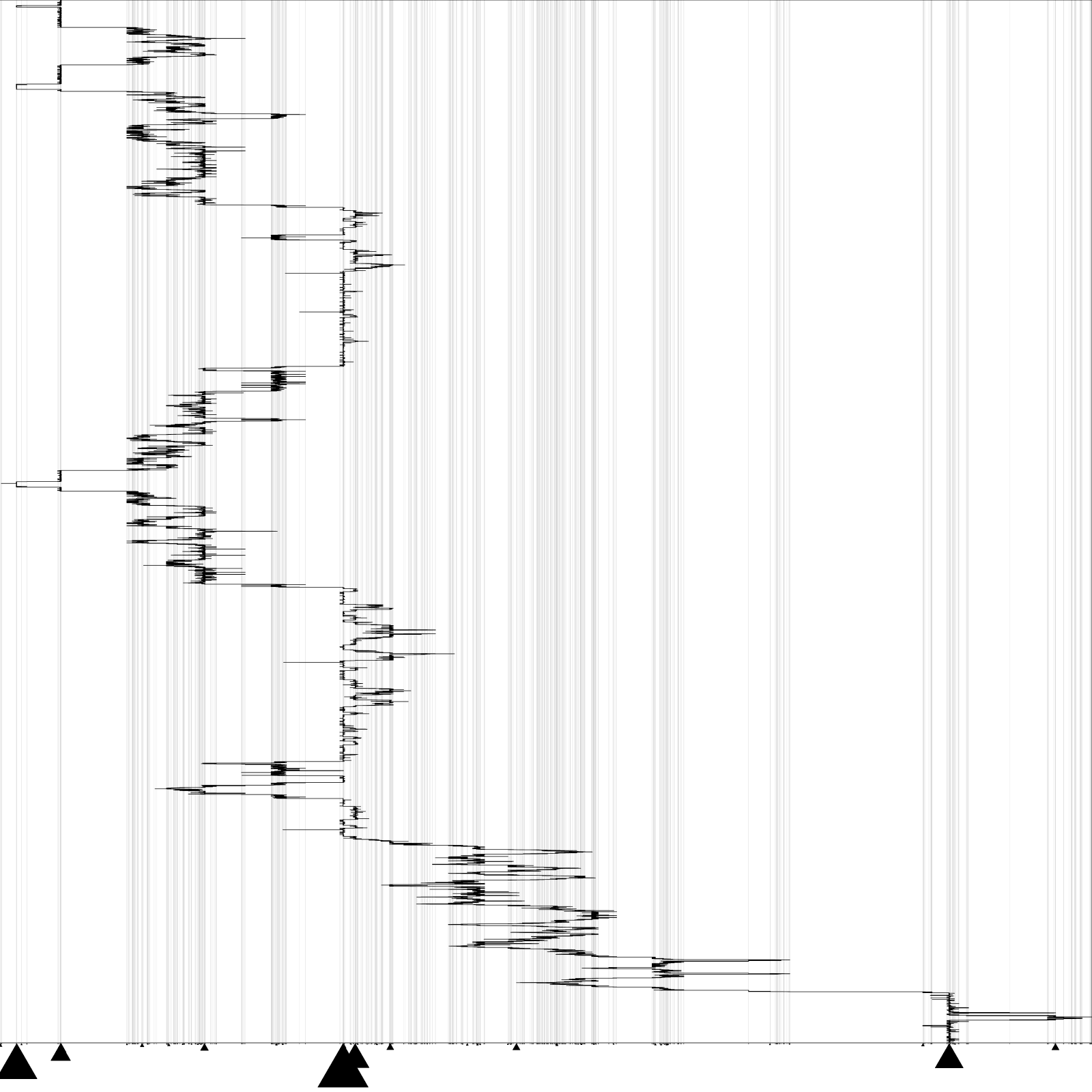}};

\node[draw,very thick,anchor=south west,inner sep=.5] at (0,-13.6) {\includegraphics[width=.45\textwidth, height=.23\textheight]{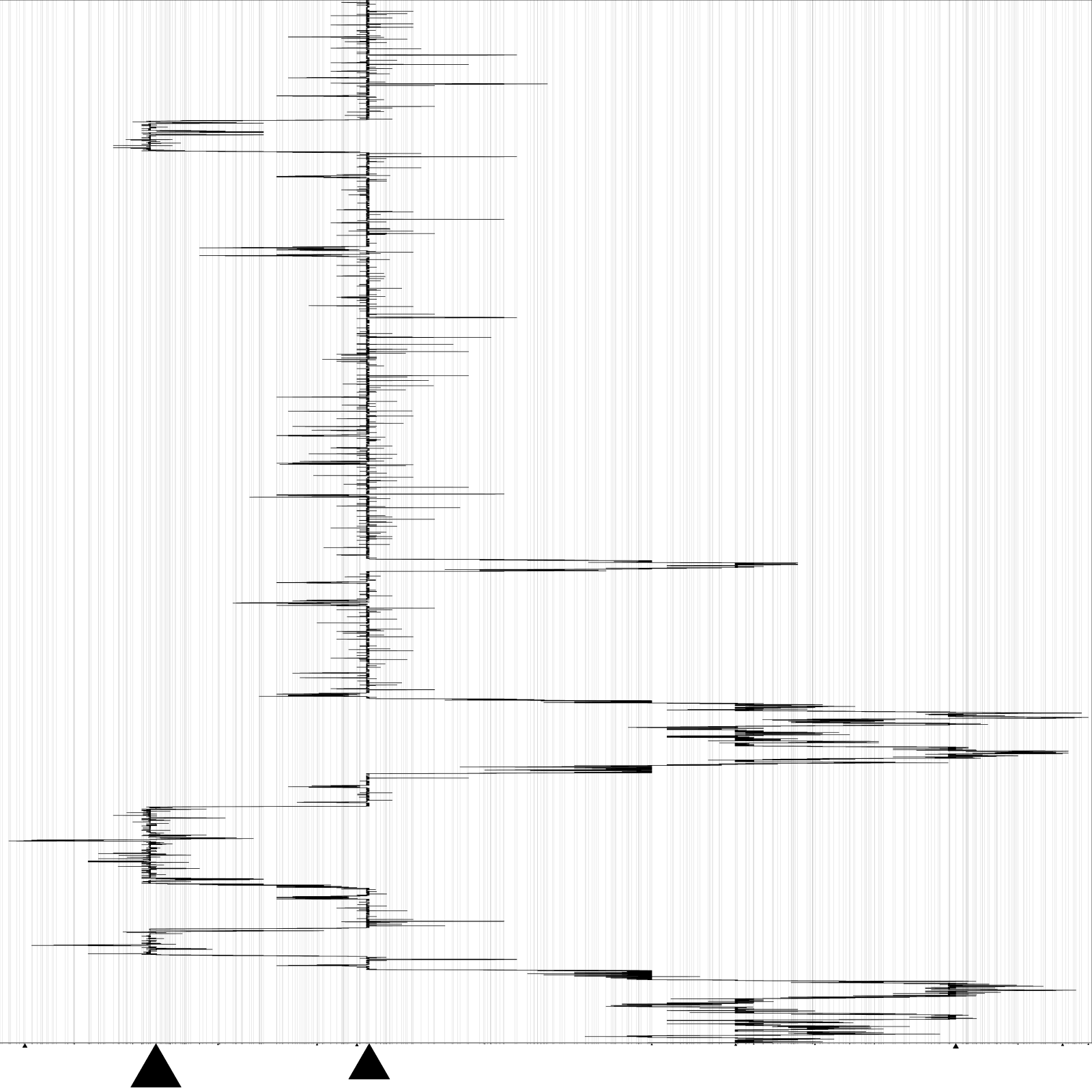}};

\node[draw,very thick,anchor=south west,inner sep=.5] at (6,-13.6) {\includegraphics[width=.45\textwidth, height=.23\textheight]{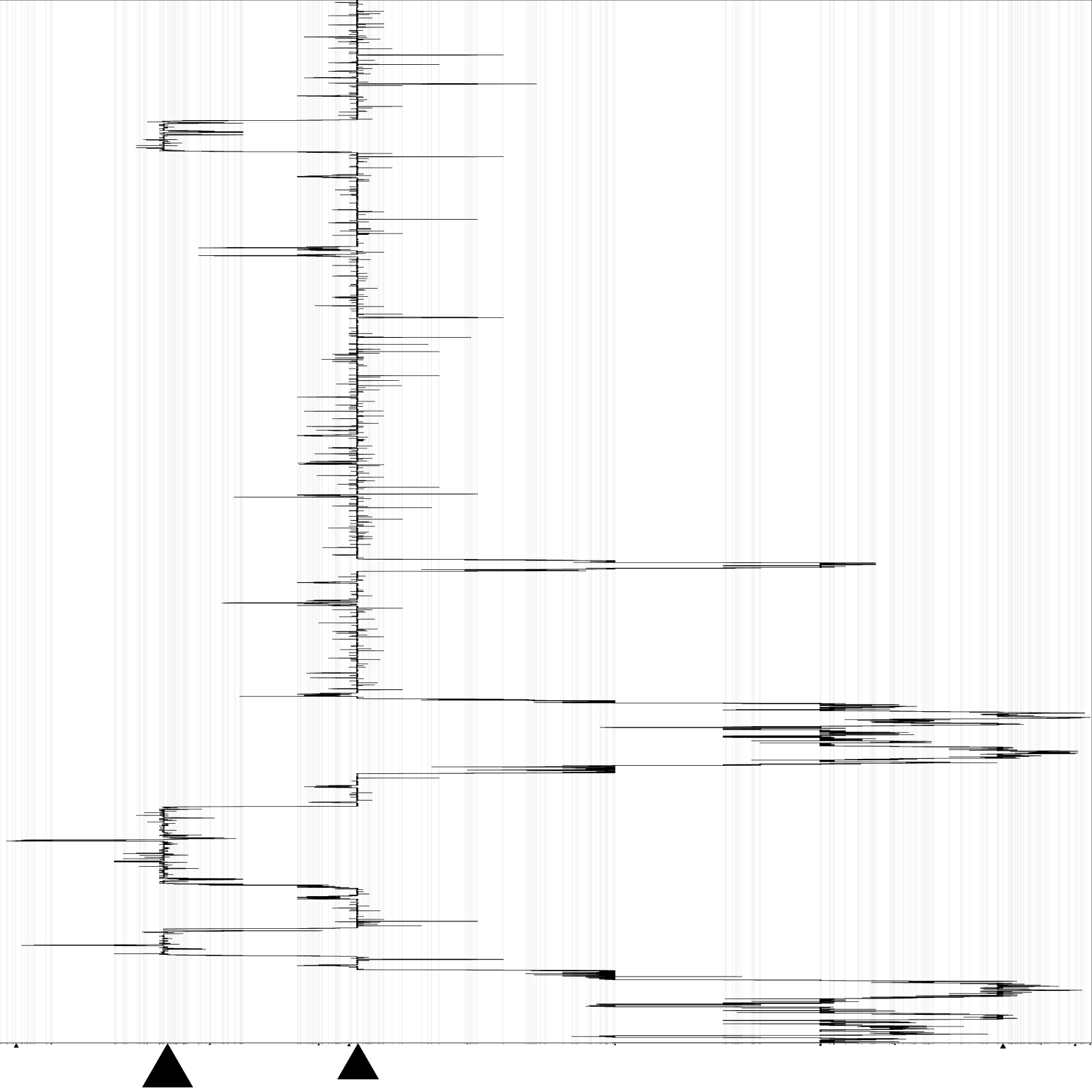}};
\end{tikzpicture}
\end{center}
\end{minipage}
}

\caption{Simulation of the process with random holding times in physical space (left column) and in resistance space (right column), with parameters $(\rho,\kappa)$ equal to $(\frac {17}{22},\frac {17}{18})$ (top row), $(\frac {17}{20},\frac {17}{20})$ (middle row) and $(\frac {17}{18},\frac {17}{22})$ (bottom row), and $\beta=\lambda=0$. The values are chosen such that $1/\rho+1/\kappa$ is constant, hence we expect the same spatial scaling for all three realizations. The size of the triangles is proportional to the holding time $\tau_i$ at the site.}
\label{fig:btm}
\end{figure}

\subsection{Outline and notational conventions} The remainder of the article is organised as follows. In Section \ref{sec:res} we establish a functional convergence statement for effective resistance, before going on in Section \ref{sec:meas} to deduce weak convergence of the invariant measure of the Mott random walk. These results are put together in Section \ref{sec:mms} to deduce a metric measure convergence result for compact versions of the spaces, and extended to the original non-compact setting in Section \ref{sec:mr}, which is where the main result of Theorem \ref{thm:main} is established. Following this, in Section \ref{sec:btm}, we explain the adaptations needed to deduce Theorem \ref{thm:main2}, and, in Section \ref{sec:quenched}, we prove that the Mott random walk exhibits quenched fluctuations, as described in Proposition \ref{prop:quenched}. Finally, in the \hyperref[homogsec]{Appendix}, we detail how our approach also applies in the homogenisation regime.

Regarding notation, throughout the article, we write $i\wedge j:=\min\{i,j\}$ and $i\vee j:=\max\{i,j\}$. We will sometimes consider sums of the form $\sum_{j=0}^{i-1}$, where $i$ can take an arbitrary value in $\mathbb{Z}$. In such cases, we suppose $\sum_{j=0}^{i-1}=0$ if $i=0$, and $\sum_{j=0}^{i-1}=-\sum_{j=i}^{-1}$ if $i\leq -1$. Moreover, we will sometimes use a continuous variable, $x$ say, where a discrete argument is required, with the understanding that it should be treated as $\lfloor x\rfloor$.

\section{Convergence of the effective resistance}\label{sec:res}

As already noted in the introduction, the collection $(r^{0,0}(\omega_i,\omega_{i+1}))_{i\in\Z}$ is i.i.d., and the marginal distribution falls into the domain of a $\rho$-stable random variable (see \eqref{eq:heavy}). Thus the rescaled partial sums, which give the effective resistances in the network that only includes nearest-neighbor resistors, readily admit a $\rho$-stable approximation. The aim of this section is to show that essentially the same holds true for effective resistances in the full model, with our main result being Theorem \ref{thm:approx} below.

The presence of non-nearest-neighbor edges has two consequences. First, it decreases the resistance between neighboring sites $\omega_i$ and $\omega_{i+1}$, because it is possible to reach $\omega_{i+1}$ from $\omega_i$ by visiting a sequence of other sites first. Second, the effective resistances $R^{\beta,0}(\omega_i,\omega_{i+1})$ and $R^{\beta,0}(\omega_j,\omega_{j+1})$ are not independent for $i\neq j$. To deal with the first difficulty, we introduce the random variables
\begin{equation}
\begin{split}
\chi^{\beta,\lambda}(i):=&\left(\sum_{j\leq i< i+1\leq k} \frac{r^{0,\lambda}(\omega_i,\omega_{i+1})}{r^{\beta,\lambda}(\omega_j,\omega_k)}\right)^{-1}\\
=&\left(\sum_{j\leq i< i+1\leq k} e^{-(1+\lambda)(\omega_i-\omega_j)-(1-\lambda)(\omega_k-\omega_{i+1})-\beta U(E_j,E_k)}\right)^{-1}.
\end{split}
\label{eq:chi}
\end{equation}
Since the increments $(\omega_{j+1}-\omega_j)_{j\in\Z}$ are i.i.d., $\chi^{\beta,\lambda}(i)$ is independent of $\omega_{i+1}-\omega_i$ for each $i$. Intuitively, $\chi^{\beta,\lambda}(i)$ is a correction that captures non-nearest-neighbor edges, in the sense that it is possible to check that
\begin{align*}
\lim_{u\to\infty}\frac{\mathbf{P}(R^{\beta,0}(\omega_0,\omega_1)\geq u)}{\mathbf{P}(r^{0,0}(\omega_0,\omega_1)\chi^{\beta,0}(0)\geq u)}=1.
\end{align*}
(Although we will not need to prove this exact statement for our argument.) It will further transpire that, since
\[\mathbf{P}(r^{0,0}(\omega_0,\omega_1)\chi^{\beta,0}(0)\geq u)\sim \mathbf{E}\left(\chi^{\beta,0}(0)^\rho\right)u^{-\rho}\]
as $u\rightarrow\infty$ (cf.\ \eqref{xxx}), the random variables $\chi^{\beta,\lambda}(i)$ only influence the scaling limit of the resistance through the constant
\begin{equation}\label{eq:cbeta}
C_\beta:=\mathbf{E}\left(\chi^{\beta,0}(0)^\rho\right).
\end{equation}
(That $C_\beta$ takes a value in $(0,\infty)$ is a straightforward consequence of the fact that $\chi^{\beta,0}(0)$ is a non-zero, bounded random variable.) The second difficulty mentioned above comes down to dealing with the correlations between the random variables in the collection $(\chi^{\beta,0}(i))_{i\in\Z}$. Here, we will show that most of the contribution towards $R^{\beta,0}(\omega_i,\omega_j)$ comes from a few edges with high resistance. Such edges are typically well-separated, and therefore, to derive the desired $\rho$-stable limit, it is enough to control the correlation between $\chi^{\beta,0}(i)$ and $\chi^{\beta,0}(j)$ for $|i-j|$ `large'. See Figure \ref{fig:strategy}. Finally, we remark that the inclusion of a non-zero $\lambda$ does not significantly affect the above discussion, merely resulting in an exponential tilting of the limiting stable process, as at \eqref{eq:stilt}.

\begin{figure}[!t]
 \includegraphics[width=\textwidth]{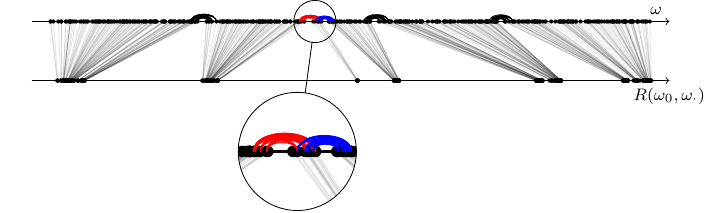}
 \caption{The circles on the upper line denote the sites $(\omega_i)_{i\in\Z}$ of the Poisson process, while on the lower line, the sites have been transformed by $\omega_i\mapsto R^{0,0}(\omega_0,\omega_i)$. The gray lines connect sites with their images. In principle, the random walk can jump between any sites $\omega_i$ and $\omega_j$, but we will see that the process is `almost nearest-neighbor', in the sense that we can disregard all edges except the nearest-neighbor edges and those that help bridge a big edge (shown above). The contribution from the edges of the second type is encoded in the random variables $(\chi(i))_{i\in\Z}$. If two big edges are close, then the bridge-edges can intersect, as shown above in red and blue. However, this will only happen with vanishingly small probability.}\label{fig:strategy}
\end{figure}

On a more technical point, we note that the result of \cite{croydon2018} assumes that the limiting process is recurrent, and hence does not directly apply to our model when $\lambda>0$. For this reason, we will approximate the effective resistance in a truncated state space. More precisely, for given natural numbers $K$ and $n$, we will consider the complete graph on the vertex set
\begin{align*}
\left\{\{...,\omega_{-Kn}\},\omega_{-Kn+1},...,\omega_{Kn-1},\{\omega_{Kn},...\}\right\},
\end{align*}
for which it is convenient to introduce the notation
\begin{align*}
\oomega_i:=
\begin{cases}
\omega_i,&\text{ if }-Kn<i<Kn,\\
\{...,\omega_{-Kn}\},&\text{ if }i={-Kn},\\
\{\omega_{Kn},...\},&\text{ if }i=Kn,
\end{cases}
\end{align*}
and let $R^{\beta,\lambda/n,Kn}$ denote the effective resistance (defined analogously to \eqref{effres}) associated with the conductances given by, for $-Kn<i,j<Kn$,
\begin{equation*}
\begin{split}
c^{\beta,\lambda/n,Kn}(\oomega_i,\oomega_j)&:=c^{\beta,\lambda/n}(\omega_i,\omega_j),\\
c^{\beta,\lambda/n,Kn}(\oomega_i,\oomega_{-Kn})&:=\textstyle \sum_{k\in\{...,-Kn\}}c^{\beta,\lambda/n}(\omega_i,\omega_k),\\
c^{\beta,\lambda/n,Kn}(\oomega_i,\oomega_{Kn})&:=\textstyle \sum_{k\in\{Kn,...\}}c^{\beta,\lambda/n}(\omega_i,\omega_k),\\
c^{\beta,\lambda/n,Kn}(\oomega_{-Kn},\oomega_{Kn})&:=\textstyle \sum_{k\in\{Kn,...\},k'\in\{...,-Kn\}}c^{\beta,\lambda/n}(\omega_{k'},\omega_k).
\end{split}
\end{equation*}
(It is straightforward to check that the sums above are almost-surely finite once $\lambda/n<1$.) In other words, the conductances $c^{\beta,\lambda,Kn}$ are obtained by collapsing all sites beyond $Kn$ and $-Kn$ into a single site each, and resolving the resulting parallel edges into a single edge using the parallel law. We will make a suitable choice for the speed measure on this graph, so that the resulting random walk can be interpreted as the original random walk reflected at $-Kn$ and $Kn$. Taking a suitable limit $n\to\infty$, this reflected random walk will converge in distribution to a stochastic process on a compact state space. Finally, we show that the limiting process does not explode in finite time so that we can obtain a limit for the process without reflection by letting $K\to\infty$.

\begin{remark}
There is a condition for non-explosion in \cite{croydon2018} in terms of resistance, but this requires the recurrence of the limiting process. In our model, we need to employ the fact that the speed measure grows rapidly in the direction of transience when $\lambda>0$. \end{remark}

\begin{theorem}\label{thm:approx}
Let $(S^{\beta,0}(u))_{u\in\R}$ denote a two-sided L\'evy process with L\'evy measure given by \eqref{Levy}. Moreover, define $(S^{\beta,\lambda}(u))_{u\in\R}$ as at \eqref{eq:stilt}. Then
\begin{align}\label{eq:convergence}
\left(n^{-1/\rho}\sign(u)R^{\beta,\lambda/n,Kn}(\overline \omega_{0},\overline \omega_{\floor{un}})\right)_{-K\leq u\leq K}\xrightarrow[n\to\infty]d \left(S^{\beta,\lambda}(u)\right)_{-K\leq u\leq K},
\end{align}
where the convergence is with respect to the Skorohod $J_1$-topology. Moreover,
\begin{eqnarray}
\lefteqn{\sup_{-Kn\leq i\leq j\leq Kn}\hspace{-8pt}n^{-1/\rho}\left|R^{\beta,\lambda/n,Kn}(\overline \omega_{i},\overline \omega_j)\vphantom{\left.\left(\sign(j)R^{\beta,\lambda/n,Kn}(\overline \omega_0,\overline \omega_j)-\sign(i)R^{\beta,\lambda/n,Kn}(\overline \omega_0,\overline \omega_i)\right)\right|}\right.}\label{eq:distortion}\\
&&\hspace{100pt}\left.\vphantom{\left|R^{\beta,\lambda/n,Kn}(\overline \omega_{i},\overline \omega_j)-\right.}-\left(\sign(j)R^{\beta,\lambda/n,Kn}(\overline \omega_0,\overline \omega_j)-\sign(i)R^{\beta,\lambda/n,Kn}(\overline \omega_0,\overline \omega_i)\right)\right|\nonumber
\end{eqnarray}
converges to 0 in $\mathbf{P}$-probability as $n\rightarrow\infty$.
\end{theorem}

The proof of this result is broken up into several steps. In Subsections \ref{sec:upper} and \ref{sec:lower}, respectively, we derive upper and lower bounds for $R^{\beta,\lambda/n,Kn}(\overline \omega_{i},\overline \omega_j)$. These bounds hold on certain likely events, the probability of which is estimated in Subsection \ref{sec:typical}. In Subsection \ref{sec:limits}, we derive a limit as at \eqref{eq:convergence} for approximations to the effective resistance based on i.i.d.\ sums. Finally, in Subsection \ref{sec:distortion}, we tie all the pieces together to complete the proof of Theorem \ref{thm:approx}. Before we proceed, we introduce some notation that will be used throughout. Firstly, we call a pair $\{i,i+1\}$ a \textbf{big edge} if it satisfies $r^{0,0}(\omega_i,\omega_{i+1})\geq n^{3/(4\rho)}$, and we write
\begin{align}
\B_n:=\left\{i\in\{-Kn,...,Kn-1\}:r^{0,0}(\omega_i,\omega_{i+1})\geq n^{3/(4\rho)}\right\}
\label{eq:big}
\end{align}
for the set of indices of big edges. Note that, for ease of notation, we identify a nearest-neighbor edge $\{i,i+1\}$ by its left vertex $i$. From \eqref{eq:heavy}, we can guess that the effective resistance is dominated by contributions from edges of nearest-neighbor resistance at least $n^{1/\rho-\eps}$.  Moreover, we let $\L_n$ denote the \textbf{long edges},
\begin{align*}
\L_n:=\left\{\{i, j\}\subset\{-Kn,...,Kn\}:|i-j|>n^{1/4}\right\},
\end{align*}
and introduce
\begin{align}\label{eq:defEn}
\mathcal{E}_n:=\sum_{\{i,j\}\in\L_n}c^{\beta,\lambda/n,Kn}(\oomega_i,\oomega_j),
\end{align}
which will be used to control the error incurred by dropping them.

\subsection{Upper bound}\label{sec:upper}

The aim of this section is to provide a convenient upper bound for $R^{\beta,\lambda/n,Kn}(\oomega_i,\oomega_j)$, see Proposition \ref{prop:upper}. For this purpose, we will approximate the correction term $\chi^{\beta,\lambda/n}(i)$ in \eqref{eq:chi} by
\begin{align}
\chiu_n(i):=\left[\sum_{j,k=0}^{a_n} \left(n^{-1/8\rho}+e^{(1+\lambda/n)(\omega_i-\omega_{i-j})+(1-\lambda/n)(\omega_{i+1+k}-\omega_{i+1})+\beta U(E_{i-j},E_{i+1+k})}\right)^{-1}\right]^{-1},
\label{eq:chiu}
\end{align}
where $a_n:=\lfloor a\log(n)\rfloor$ for some constant $a > 0$ that will be chosen to satisfy \eqref{eq:LDP} below. Note that $\chiu_n(i)$ is bounded uniformly in $n$, so by the dominated convergence theorem,
\begin{align}\label{eq:limitu}
\lim_{n\to\infty}\mathbf{E}[\chiu_n(i)]=\mathbf{E}[\chi^{\beta,0}(i)].
\end{align}
We moreover introduce the event
\begin{equation}\label{eq:defAu}
\begin{split}
\Aupper_n:=&\left\{\{-Kn,...,-Kn+a_n\}\cap \B_n=\emptyset,\{Kn-a_n-1,...,Kn\}\cap\B_n=\emptyset \right\}\\
&\cap\left\{|k-l|>2a_n\text{ for all }k, l\in\B_n\text{ with }k\neq l\right\}\\
&\cap \bigcap_{k\in\B_n}\left\{\omega_k-\omega_{k-a_n}\leq \log(n)/(2\rho),\omega_{k+1+a_n}-\omega_{k+1}\leq \log(n)/(2\rho)\right\},
\end{split}
\end{equation}
and define, for $i\leq j$,
\begin{align}\label{eq:defRu}
\Ru_n(i,j):=
&\sum_{\substack{k\in \{(i-a_n)\vee(-Kn),...,(j+a_n)\wedge(Kn-1)\}\setminus\B_n}}r^{\beta,\lambda/n}(\omega_k,\omega_{k+1})\\
&+\sum_{k\in \{i,...,j-1\}\cap \B_n}r^{0,\lambda/n}(\omega_k,\omega_{k+1})\chiu_n(k).\nonumber
\end{align}

\begin{proposition}[Upper bound]\label{prop:upper}
On $\Aupper_n$ for suitably large $n$, for all $-Kn\leq i\leq j\leq Kn$,
\begin{align}\label{eq:upper}
R^{\beta,\lambda/n,Kn}(\oomega_i,\oomega_j)\leq \Ru_n(i,j).
\end{align}
\end{proposition}

\begin{proof}
Throughout this proof, we will drop the superscripts, and simply write $R$ and $r$ for $R^{\beta,\lambda/n,Kn}$ and $r^{\beta,\lambda/n,Kn}$. For $k\in \B_n\cap[-Kn,Kn]$, define
\begin{align*}
\U^-_k&:=\{k-a_n+1,...,k\},\\
\U^+_k&:=\{k+1,...,k+a_n\},\\
\U&:=\bigcup_{k\in\B_n} \left(\U^+_k\cup\U^-_k\right).
\end{align*}
Note that, on $\Aupper_n$, the collection $\{\U^\pm_k:k\in \B_n\}$ is disjoint and does not intersect $\{\pm Kn\}$. On that event, we consider the graph with vertex set
\begin{align*}
\left(\{-Kn,...,Kn\}\setminus \U\right)\cup \bigcup_{k\in \B_n}(\U^-_k\times \U^+_k)\cup \bigcup_{k\in \B_n}(\U^+_k\times \U^-_k).
\end{align*}
That is, each vertex $i'$ in $\U^-_k$ has been replaced by $a_n$ new vertices $\{(i',j'):j'\in \U^+_k\}$, each corresponding to a vertex on the `opposite side' of $k$, and vice versa. The conductances $\widehat c$ in the new graph are defined as follows.
\begin{itemize}
 \item Outside of $\U$, we only keep the nearest-neighbor conductances. That is, for $i',j'\in \{-Kn,...,Kn\}\setminus \U$,
\begin{align*}
\widehat c(i',j'):=c(\oomega_{i'},\oomega_{j'})\1_{|i'-j'|=1}.
\end{align*}
\item For every $k\in \B_n$, the edges connecting $\omega_{k-a_n}$ to $\U^-_k$ and $\omega_{k+a_n+1}$ to $\U^+_k$ in the original graph are `split up' among the new vertices: for $i'\in\U^-_k$, $j'\in\U^+_k$,
\begin{align*}
\widehat c\left(k-a_n,(i',j')\right)&:=\frac{c(\oomega_{k-a_n},\oomega_{i'})}{a_n},\\
\widehat c\left((j',i'),k+1+a_n\right)&:=\frac{c(\oomega_{j'},\oomega_{k+1+a_n})}{a_n}.
\end{align*}
\item The edges connecting $\U^-_k$ and $\U^+_k$ in the original graph are `redistributed' among the new vertices: for $i',i''\in\U_k^-$, $j',j''\in \U_k^+$,
 \begin{align*}
\widehat c\left((i',j'),(j'',i'')\right):=c(\oomega_{i'},\oomega_{j'})\1_{j'=j'',i'=i''}.
\end{align*}
\end{itemize}
See Figure \ref{fig:deleting} for an illustration. We can recover the original conductances via the following two steps. First, we `merge' the newly created vertices. This yields parallel edges between $\U^-_k$ and $\U^+_k$, which we replace by a single edge with the same effective resistance. This results in a graph whose conductances agree with the original conductances, except that some edges are missing. We obtain the original graph by adding the missing edges.

\begin{figure}[!t]
\includegraphics[width=.8\textwidth]{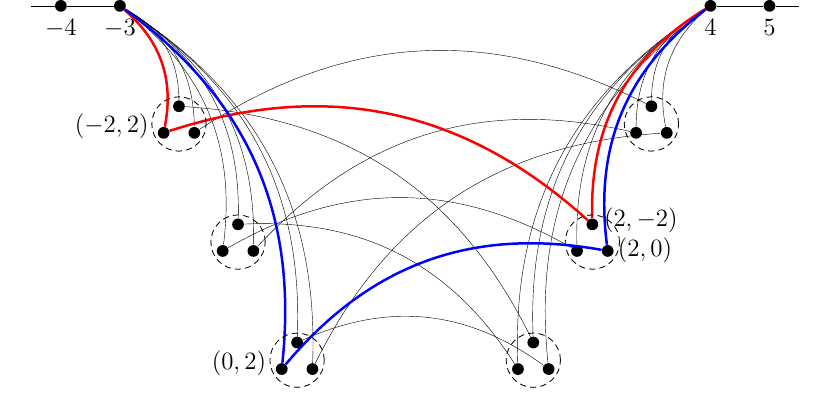}
\includegraphics[width=.8\textwidth]{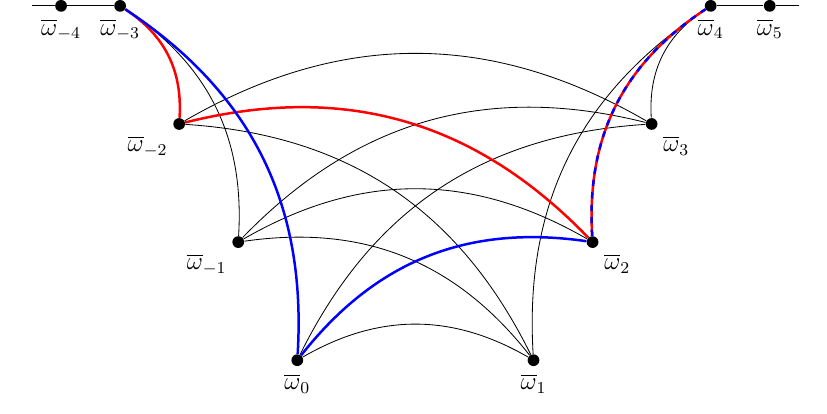}
\caption{The top diagram shows a portion of the new graph for $a_n=3$ in the case $\omega_0\in\B_n$ (shifted vertically for clarity). Note that all paths between $\omega_{-3}$  and $\omega_4$ are disjoint, so we can compute $\widehat R(\omega_{-3},\omega_4)$ by the parallel law. To recover the original graph, we first `merge' every dashed circle into a single vertex, replace the resulting parallel edges by a single edge with the appropriate resistance (red/blue edge in the bottom diagram), and then add missing edges (not shown).}\label{fig:deleting}
\label{fig:modified}
\end{figure}

Importantly, both steps described in the previous paragraph decrease the effective resistance between any two sites (as a consequence of Rayleigh's monotonicity law, see \cite[Section 1.4]{DS}, for example). Writing $\widehat R$ for the effective resistance in the new graph, we therefore have, for $i,j\notin \U$,
\begin{align}
\label{eq:R<hatR}
R(\overline \omega_i,\overline \omega_j)\leq \widehat R(i,j).
\end{align}
Since all paths between $i$ and $j$ in the new graph are disjoint (see Figure \ref{fig:modified}), we can compute $\widehat R(i,j)$ by the parallel law. Specifically, for $i,j\notin \U$ with $i<j$, we have
\begin{equation}\label{eq:Rhat}
\begin{split}
&\widehat R(i,j)=\sum_{k\in \{i,...,j-1\},\{k,k+1\}\cap\U=\emptyset}r(\oomega_k,\oomega_{k+1})\\
&+\sum_{k\in \B_n\cap\{i,...,j-1\}}\left(\sum_{\substack{i'\in\U_k^-,\\j'\in\U_k^+}} \left(r(\oomega_{k-a_n},\oomega_{i'})a_n+r(\oomega_{j'},\oomega_{k+1+a_n})a_n+r(\oomega_{i'},\oomega_{j'})\right)^{-1}\right)^{-1}.
\end{split}
\end{equation}
On $\Aupper_n$, we have for all for $j'\in\U_k^+$,
\begin{align*}
r(\oomega_{j'},\oomega_{k+1+a_n})&=e^{\beta U(E_{j'},E_{k+1+a_n})-\lambda(\omega_{j'}+\omega_{k+1+a_n})/n+(\omega_{k+1+a_n}-\omega_{j'})}\\
&\leq Ce^{-\lambda(\omega_k+\omega_{k+1})/n+(\omega_{k+1+a_n}-\omega_{k+1})}\\
&\leq C e^{-\lambda (\omega_k+\omega_{k+1})/n}n^{1/(2\rho)},
\end{align*}
where we have applied that $\omega_{k+1}\leq \omega_{j'}$ and $\omega_{k+1+a_n}\leq\omega_{k+1}+\log(n)/2\rho$, while for $i'\in\U_k^-$ we similarly have
\begin{align*}
r(\oomega_{k-a_n},\oomega_{i'})&=e^{\beta U(E_{k-a_n},E_{i'})-\lambda(\omega_{k-a_n}+\omega_{i'})/n+(\omega_{i'}-\omega_{k-a_n})}\\
&=e^{\beta U(E_{k-a_n},E_{i'})-\lambda(\omega_{k+1}+\omega_k)/n+\lambda(2\omega_k-\omega_{k-a_n}-\omega_{i'})/n+\lambda(\omega_{k+1}-\omega_k)/n+(\omega_{i'}-\omega_{k-a_n})}\\
&\leq e^{\beta U(E_{k-a_n},E_{i'})-\lambda(\omega_{k+1}+\omega_k)/n+\lambda(\omega_{k+1}-\omega_k)/n+(1+2\lambda/n)(\omega_{k}-\omega_{k-a_n})}\\
&\leq Ce^{-\lambda(\omega_{k+1}+\omega_k)/n+\lambda(\omega_{k+1}-\omega_k)/n}n^{(1+2\lambda/n)/(2\rho)}.
\end{align*}
Moreover, recalling~\eqref{eq:big}, $k\in\B_n$ implies that
\begin{align*}
r^{0,\lambda/n,Kn}(\oomega_k,\oomega_{k+1})\geq e^{-\lambda(\omega_k+\omega_{k+1})/n+\lambda(\omega_{k+1}-\omega_k)/n}n^{3(1-\lambda/n)/(4\rho)},
\end{align*}
so that for all $n$ large enough,
\begin{align*}
\frac{r(\oomega_{k-a_n},\oomega_{i'})a_n}{r^{0,\lambda/n,Kn}(\oomega_k,\oomega_{k+1})}+\frac{r(\oomega_{j'},\oomega_{k+a_n+1})a_n}{r^{0,\lambda/n,Kn}(\oomega_k,\oomega_{k+1})}\leq n^{-1/(8\rho)}.
\end{align*}
On the other hand, we have
\begin{align*}
\frac{r(\oomega_{i'},\oomega_{j'})}{r^{0,\lambda/n,Kn}(\oomega_k,\oomega_{k+1})}=e^{(1+\lambda/n)(\oomega_k-\oomega_{i'})+(1-\lambda/n)(\oomega_{j'}-\oomega_{k+1})+\beta U(E_{i'},E_{j'})}.
\end{align*}
Observe that these are exactly the terms appearing in the definition of $\chiu_n$ in~\eqref{eq:chiu}. From \eqref{eq:Rhat}, we therefore have
\begin{align}
\label{eq:ijnotinU}
\widehat R(i,j)\leq \sum_{k\in \{i,...,j-1\},\{k,k+1\}\cap\U=\emptyset}r(\oomega_k,\oomega_{k+1})+\sum_{k\in \B_n\cap\{i,...,j-1\}}r^{0,\lambda/n,Kn}(\oomega_k,\oomega_{k+1})\chiu_n(k),
\end{align}
which establishes \eqref{eq:upper} in the case $i,j\notin \U$. Note that in the above argument, we had only to modify the graph around $k\in \B_n\cap\{i,\ldots,j\}$. If $i\in\U_k^+$ for some $k\in\B_n$, we can obtain the desired bound by the same argument but not treating the edge $\{k,k+1\}$ as a big edge, i.e. performing the same construction with $\B_n$ replaced by $\B_n\setminus\{k\}$. Similarly if $j\in\U_k^-$ for some $k\in\B_n$.

Finally, in the case where $i\in\U_k^-$ for some $k\in\B_n$ and $j\not\in\mathcal{U}$, we first use the triangle inequality for the effective resistance to deduce that
\[
R(\oomega_i,\oomega_j)\leq R(\oomega_i,\oomega_{k-a_n})+R(\oomega_{k-a_n},\oomega_j).
\]
For the first term on the right-hand side, we use $R(\oomega_i,\oomega_{k-a_n})\leq \sum_{l=k-a_n}^{i-1}r^{\beta,\lambda/n}(\omega_l,\omega_{l+1})$. For the second term, since we know $k-a_n\notin \mathcal{U}$ on $\Aupper_n$, we can use \eqref{eq:R<hatR} and \eqref{eq:ijnotinU} to get
\begin{align*}
&R(\oomega_{k-a_n},\oomega_j) \\
&\quad \leq \sum_{l\in \{k-a_n,...,j-1\},\{l,l+1\}\cap\U=\emptyset}r(\oomega_l,\oomega_{l+1})+\sum_{l\in \B_n\cap\{k-a_n,...,j-1\}}r^{0,\lambda/n,Kn}(\oomega_l,\oomega_{l+1})\chiu_n(l).
\end{align*}
Using that $\{k-a_n+1,\ldots,i-1\}\subseteq \U$, we may restrict the sum in the first term to $l\in \{i,...,j-1\}$ with $\{l,l+1\}\cap\U=\emptyset$. Then combining the last three bounds, we find~\eqref{eq:upper}. The case $i\not\in\mathcal{U}$, $j\in\U_k^+$ can be handled symmetrically. The cases $i\in \U_k^-$ and $j\in \U_{k'}^+$ for some $k,k'\in\B_n$ with $k\le k'$ are also dealt with in a similar way, starting from the bound
\begin{equation*}
R(\oomega_i,\oomega_j)\leq R(\oomega_i,\oomega_{k-a_n})+R(\oomega_{k-a_n},\oomega_{k'+a_n+1})+R(\oomega_{k'+a_n+1},\oomega_j).
\end{equation*}
\end{proof}

\subsection{Lower bound}\label{sec:lower}

We now proceed to deduce a lower bound for the resistance. For this purpose, we now approximate the correction term $\chi^{\beta,\lambda/n}(i)$ in \eqref{eq:chi} by
\begin{align*}
\chil_n(i):=&\left(\sum_{\substack{j\leq i, k\ge i+1,\\
k-j\leq b_n}} e^{-(1+\lambda/n)(\omega_i-\omega_{j})-(1-\lambda/n)(\omega_{k}-\omega_{i+1})-\beta U(E_{j},E_{k})}\right)^{-1},
\end{align*}
where $b_n:=n^{1/4}$. By the dominated convergence theorem,
\begin{align}\label{eq:limitl}
\lim_{n\to\infty}\mathbf{E}[\chil_n(i)]=\mathbf{E}[\chi^{\beta,0}(i)].
\end{align}
Let
\begin{equation}\label{eq:defAl}
\begin{split}
\Alower_n:=&\left\{\{-Kn,...,-Kn+b_n\}\cap\B_n=\emptyset,\{Kn-b_n-1,...,Kn\}\cap\B_n=\emptyset\right\}\\
&\cap \left\{|i-j|>2b_n\text{ for all distinct $i,j$}\in\B_n\right\},
\end{split}
\end{equation}
and define, for $i\leq j$,
\begin{align}\label{eq:defRl}
\Rl_n(i,j):=\sum_{k\in\{i,...,j-1\}\cap\B_n}r^{0,\lambda/n}(\omega_k,\omega_{k+1})\chil_n(k).
\end{align}

\begin{proposition}[Lower bound]\label{prop:lower}
On $\Alower_n$, for all $-Kn\leq i\leq j\leq Kn$,
\begin{align*}
R^{\beta,\lambda/n,Kn}(\oomega_i,\oomega_j)^{-1}\leq \left(\Rl_n(i,j)\right)^{-1}+\mathcal{E}_n,
\end{align*}
where $\mathcal{E}_n$ is defined as at \eqref{eq:defEn}, and the right hand side is interpreted as infinity if the sum in \eqref{eq:defRl} is empty.
\end{proposition}

\begin{proof}
We obtain a lower bound on $R^{\beta,\lambda/n,Kn}$ in two steps: first, we consider a modified graph
\begin{align*}
G^{\L_n^c}:=\big(\{\oomega_{-Kn},...,\oomega_{Kn}\},\{\oomega_i,\oomega_j\}_{\{i,j\}\in\L_n^c,\:i\neq j}\big)
\end{align*}
obtained by removing the long edges $\L_n$. We then further modify $G^{\L_n^c}$ into a graph
\begin{align*}
G^{\L_n^c,\textup{coll.}}:=\big(\{\oomega_k\}_{k\in\B_n\cup\{Kn\}},\{\oomega_{k},\oomega_{l}\}_{k,l\in\B_n\cup\{Kn\}\colon k=\prev(l)}\big)
\end{align*}
by collapsing all sites between consecutive big edges into a single site, where we write $\prev(k)$ for the previous element of $\B_n$ (if it exists), i.e.\
\begin{align*}
\prev(k):=\max\{\B_n\cap\{-Kn,...,k-1\}\},
\end{align*}
where we use the convention $\max\emptyset\coloneqq-\infty$. We now carry out this program. For the first step, we define conductances on $G^{\L_n^c}$ by
\begin{align*}
c^{\L_n^c}(\oomega_{i'},\oomega_{j'}):=c^{\beta,\lambda/n,Kn}(\oomega_{i'},\oomega_{j'})\1_{\{i',j'\}\notin\L_n}
\end{align*}
and note that the resulting effective resistance is given by
\begin{align*}
R^{\L_n^c}(\oomega_{i'},\oomega_{j'})^{-1}&:=\inf_{\substack{f\colon \oomega\to[0,1]\\f(\oomega_{i'})=0,f(\oomega_{j'})=1}}\left\{\frac 12\sum_{\{k,l\}\notin\L_n}c^{\beta,\lambda/n,Kn}(\oomega_k,\oomega_l)(f(\oomega_k)-f(\oomega_l))^2\right\}.
\end{align*}
Using the definition of effective resistance, we get
\begin{equation}\label{eq:that}
\begin{split}
R^{\beta,\lambda/n,Kn}(\oomega_i,\oomega_j)^{-1}&=\inf_{\substack{f\colon \oomega\to[0,1]\\f(\oomega_i)=0,f(\oomega_j)=1}}\left\{\frac 12\sum_{k,l}c^{\beta,\lambda/n,Kn}(\oomega_k,\oomega_l)(f(\oomega_k)-f(\oomega_l))^2\right\}\\
&\leq R^{\L_n^c}(\oomega_i,\oomega_j)^{-1}+\mathcal{E}_n,
\end{split}
\end{equation}
where $\mathcal E_n$ is defined at \eqref{eq:defEn}. For the second step, let $I_k:=\{\oomega_{-Kn\vee(\prev(k)+1)},...,\oomega_k\}$ denote the set of vertices between $k\in\B_n\cup\{Kn\}$ and its previous element. To define the conductances on $G^{\L_n^c,\textup{coll.}}$, we formally define conductances on $G^{\L_n^c}$ by
\begin{align*}
\widetilde c(\oomega_{i'},\oomega_{j'}):=\begin{cases}\infty,&\text{if }i',j'\in I_k\text{ for some }k\in\B_n\cup\{Kn\},\\
c^{\L_n^c}(\oomega_{i'},\oomega_{j'}),&\text{otherwise},
\end{cases}
\end{align*}
and then identify the set $I_k$, $k\in\B_n\cup\{Kn\}$, with the element $\oomega_k$. Note that in this way, the index $-Kn$ is always identified with the smallest element of $\B_n\cup\{Kn\}$. The conductances $c^{\L_n^c,\textup{coll.}}$ can now be computed from $\widetilde c$ and the parallel law. More precisely, if $k,l\in\B_n\cup\{Kn\}$ with $k=\prev(l)$, then, on $\Alower_n$,
\begin{equation}\label{eq:resolved}
\begin{split}
c^{\L_n^c,\textup{coll.}}(\oomega_k,\oomega_l)&=\sum_{k'\in I_k,l'\in I_l} c^{\L_n^c}(\oomega_{k'},\oomega_{l'})\\
&= \sum_{\substack{k'\in I_k,l'\in I_l\\ k'-l'\leq b_n}} c^{\beta,\lambda/n,Kn}(\oomega_{k'},\oomega_{l'})\\
&= \left(r^{0,\lambda/n}(\omega_k,\omega_{k+1})\chil_n(k)\right)^{-1}.
\end{split}
\end{equation}
Next we claim that the network defined by $c^{\L_n^c,\textup{coll.}}$ is nearest-neighbor in the sense that
\begin{align}\label{eq:nearest-neighbor}
c^{\L_n^c,\textup{coll.}}(\oomega_k,\oomega_l)=0\text{ unless }k=\prev(l)\text{ or }l=\prev(k).
\end{align}
Indeed, by definition, $c^{\L_n^c}(\oomega_{i'},\oomega_{j'})=0$ whenever $|i'-j'|> n^{1/4}$ and, on $\Alower_n$, there are no two big edges within distance $2b_n$ of each other or $b_n$ to the boundary $\{-Kn,Kn\}$. In particular, if $i'\leq k<l<j'$ with $k,l\in\B_n\cup\{Kn\}$, we have $c^{\L_n^c}(\omega_{i'},\omega_{j'})=0$. In other words, there are no edges in $\L_n^c$ skipping over more than one big edge, as the red edges in Figure \ref{fig:crossing} do.

\begin{figure}[!t]
\includegraphics[width=\textwidth]{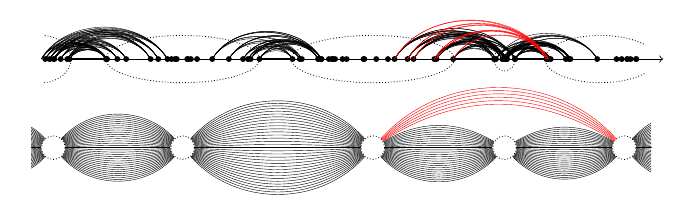}
\caption{We first modify the graph by removing long edges -- in the top diagram we have displayed a portion of the resulting graph $G^{\L_n^c}$. The second graph $G^{\L_n^c,\textup{coll.}}$ is obtained by collapsing all sites between big edges, $\{\omega_k,\omega_{k+1}\}$ for $k\in\B_n$ (shown as thick lines in both diagrams), into single sites (shown as dotted lines in both diagrams). If $\B_n$ is well-separated, then $G^{\L_n^c,\textup{coll.}}$ has only nearest-neighbor edges (unlike the above configuration).}\label{fig:crossing}
\end{figure}

We can now conclude: Let $k,l\in\B_n\cup\{Kn\}$ be such that $i\in I_k$ and $j\in I_l$. There is nothing to prove if $k=l$, so we assume $k<l$. We have
\begin{align*}
R^{\L_n^c}(\oomega_i,\oomega_j)&\geq R^{\L_n^c,\textup{coll.}}(\oomega_k,\oomega_l)\\
&=\sum_{\substack{k'\in \B_n\cup\{Kn\}\\k<k'\leq l}}\left(c^{\L_n^c,\textup{coll.}}(\oomega_{\prev(k')},\oomega_{k'})\right)^{-1}\\
&=\Rl(\oomega_i,\oomega_j).
\end{align*}
The inequality is due to the definition of $c^{\L_n^c,\textup{coll.}}$ and Rayleigh's monotonicity principle. The first equality is the series law for the effective resistance, which we can use due to \eqref{eq:nearest-neighbor}. The final equality is \eqref{eq:resolved}, and together with \eqref{eq:that} the proof is finished.
\end{proof}

\subsection{Unlikely configurations}\label{sec:typical}

We next show that the events $\Alower_n$ and $\Aupper_{n}$ described in the previous two subsections occur with high probability, and give a tail estimate for the quantity $\mathcal{E}_n$. Recall that $a_n=\floor{ a\log(n)}$ and $b_n=n^{1/4}$.

\begin{lemma}\label{lem:aux}
Recall $\Aupper_{n}$ and $\Alower_n$ from \eqref{eq:defAu} and \eqref{eq:defAl}, respectively. It holds that
\begin{align*}
\lim_{n\to\infty}\mathbf{P}(\Aupper_n)=\lim_{n\to\infty}\mathbf{P}(\Alower_n)=1.
\end{align*}
\end{lemma}

\begin{proof}
We start by showing that $\lim_{n\to\infty} \mathbf{P}(\Alower_n^c)=0$. Recalling~\eqref{eq:heavy} and~\eqref{eq:big}, the union bound gives that
\begin{align}
\label{eq:nobigedge}
 \mathbf{P}\left(\{-Kn,...,-Kn+b_n\}\cap\B_n\neq\emptyset\right)\leq b_n \mathbf{P}\left(0\in\B_n\right)=b_nn^{-3/4}\xrightarrow[n\to\infty]{} 0.
\end{align}
Similarly, note that
\begin{align*}
\{|i-j|>2b_n\text{ for all distinct $i, j$}\in\B_n\}^c\subseteq \textstyle \bigcup_{i=-Kn}^{Kn}\left(A(i)\cap\{i\in\B_n\}\right),
\end{align*}
where
\begin{align*}
A(i):=\left\{\exists j\in \{i-2b_n,...,i+2b_n\}\setminus\{i\}\text{ such that }j\in\mathcal{B}_n\right\}.
\end{align*}
Since $A(i)$ is independent of $\{i\in\B_n\}$ and
\begin{equation}
\label{eq:anotherbigedge}
 \mathbf{P}(A(i))\le 4b_n n^{-3/4} \le 4n^{-1/2}
\end{equation}
as in~\eqref{eq:nobigedge}, by using the union bound we find that
\begin{align*}
\mathbf{P}\left(\{|i-j|>2b_n\text{ for all distinct $i, j$}\in\B_n\}^c\right)&\leq 2Kn \mathbf{P}(A(0))\mathbf{P}(0\in\B_n)\\
&\leq 8Kn^{-1/4}.
\end{align*}
This establishes that $\lim_{n\to\infty}\mathbf{P}(\Alower_n^c)=0$.

Next, we prove that $\lim_{n\to\infty}\mathbf{P}(\Aupper_n^{c})=0$. Since $a_n\leq b_n$, the result of the previous paragraph takes care of the events in the first two lines of~\eqref{eq:defAu}. Thus it suffices to show that
\begin{align*}
\mathbf{P}\left(\bigcup_{k\in\{-Kn,...,Kn\}}\left(A'(k) \cap\left\{k\in \B_n\right\}\right)\right)\xrightarrow[n\to\infty]{}0,
\end{align*}
where
\begin{equation*}
 A'(k):=\left\{\max\{\omega_k-\omega_{k-a_n},\omega_{k+1+a_n}-\omega_{k+1}\}>\log(n)/(2\rho)\right\}.
\end{equation*}
Note that the events $A'(k)$ and $\{k\in\B_n\}$ are independent, and that
\begin{align*}
\mathbf{P} \left(A'(k)\right)&\le 2\mathbf P\left(\omega_k-\omega_{k-a_n}>\log(n)/(2\rho)\right)\notag\\
\quad &\leq 2 e^{-\log(n) aI(1/(2\rho a))}\notag
\end{align*}
where $I$ is the large deviation rate function of the exponential distribution with rate $\rho$. Since $\lim_{a\to 0}aI(1/(2a\rho))=1/2$, we can find $a>0$ such that
\begin{equation}
\mathbf P\left(A'(k)\right) \leq 2 n^{-3/8}.
\label{eq:LDP}
\end{equation}
Using $\mathbf{P}(k\in \B_n)=n^{-3/4}$ and the union bound over $k$, we can complete the proof.
\end{proof}

\begin{lemma}\label{lem:aux2}
Recall the definition of $\mathcal{E}_n$ from \eqref{eq:defEn}. For every $\beta$, $\lambda>0$ and $K$, there exist $c_1,c_2,c_3>0$ such that
\begin{align*}
\mathbf{P}\left(\mathcal{E}_n\geq c_1n^2 e^{-c_2 n^{1/4}} \right)\leq e^{-c_3n^{1/4}}.
\end{align*}
\end{lemma}

\begin{proof}
Let
\begin{align*}
A_n:=&\left\{\omega_{Kn}\leq (K+1)n/\rho\right\}\cap\bigcap_{i\in\{-Kn,...,Kn-n^{1/4}\}}\left\{|\omega_i-\omega_{i+n^{1/4}}|\geq n^{1/4}/(2\rho)\right\}\\
&\cap\bigcap_{j\geq n^{1/4}}\left\{j/(2\rho)\leq |\omega_{-Kn+n^{1/4}}-\omega_{-Kn+n^{1/4}-j}|\leq 2j/\rho \right\}\\
&\cap\bigcap_{j\geq n^{1/4}}\left\{j/(2\rho)\leq |\omega_{Kn-n^{1/4}+j}-\omega_{Kn-n^{1/4}}|\leq 2j/\rho \right\}.
\end{align*}
Towards obtaining a lower bound for the probability of the events in the last two lines, we note that the union bound and standard large deviation estimates (e.g.\ \cite[Theorem 2.2.3 and Exercise 2.2.23(c)]{DZ98}) imply, for some $c>0$ and $n$ large enough,
\begin{align*}
\mathbf P\left(\exists j\geq n^{1/4}:\omega_j\leq j/(2\rho)\text{ or }\omega_j\geq 2j/\rho\right)\leq \sum_{j\geq n^{1/4}} e^{-cj}\leq ce^{-cn^{1/4}}.
\end{align*}
Together with the translation invariance of the model, and similar estimates for the remaining events, we consequently find that there exists a constant $c_3>0$ such that $\mathbf{P}(A_n^c)\leq e^{-c_3n^{1/4}}$. Moreover, on $A_n$, for every $i\in\{-Kn,...,Kn-n^{1/4}\}$ and $n\geq 4\lambda$,
\begin{align*}
c^{\beta,\lambda/n,Kn}(\oomega_i,\oomega_{Kn})&=\sum_{j\geq 0}c^{\beta,\lambda/n}(\omega_i,\omega_{Kn+j})\\
&\leq \sum_{j\geq 0}e^{2(K+1)\lambda/\rho+2(j+n^{1/4})\lambda/(n\rho)-(j+n^{1/4})/(2\rho)}\\
&\leq Ce^{-cn^{1/4}}.
\end{align*}
A similar argument shows that, for any $i\in\{-Kn+n^{1/4},...,Kn\}$,
\begin{align*}
c^{\beta,\lambda,Kn}(\oomega_i,\oomega_{-Kn})\leq Ce^{-cn^{1/4}}.
\end{align*}
And, for $i,j\in\{-Kn+1,...,Kn-1\}$ with $|i-j|\geq n^{1/4}$,
\begin{align*}
c^{\beta,\lambda/n,Kn}(\oomega_i,\oomega_j)\leq e^{2(K+1)\lambda}e^{-cn^{1/4}}.
\end{align*}
This shows that, on $A_n$, it holds that $\mathcal{E}_n\leq c_1n^2e^{-c_2 n^{1/4}}$.
\end{proof}

\subsection{Convergence of an auxiliary process}\label{sec:limits}

In our next result, we establish a scaling limit for two auxiliary processes that are sums of independent random variables and capture the behaviour of $R^{\beta,\lambda/n,Kn}(\oomega_i,\oomega_j)$. To this end, we introduce i.i.d.\ copies $((\omega^{(i)},E^{(i)}))_{i\in\Z}$ of $(\omega,E)$, independent of $(\omega,E)$, and define
\begin{align*}
\lefteqn{\xiu_n(i):=}\\
&\left(\sum_{j,k=0}^{a_n} \left(n^{-1/(8\rho)}+e^{(1+\lambda/n)(\omega_i^{(i)}-\omega_{i-j}^{(i)})+(1-\lambda/n)(\omega_{i+1+k}^{(i)}-\omega_{i+1}^{(i)})+\beta U(E_{i-j}^{(i)},E_{i+1+k}^{(i)})}\right)^{-1}\right)^{-1},
\end{align*}
\begin{align*}
\xil_n(i)&:=\left(\sum_{\substack{j\leq i, k\geq i+1\\
k-j\leq b_n}} e^{-(1+\lambda/n)(\omega_i^{(i)}-\omega_{j}^{(i)})-(1-\lambda/n)(\omega_{k}^{(i)}-\omega_{i+1}^{(i)})-\beta U(E^{(i)}_{j},E^{(i)}_{k})}\right)^{-1}.
\end{align*}
Then $\xil_n=(\xil_n(i))_{i\in\mathbb{Z}}$ and $\xiu_n=(\xiu_n(i))_{i\in\mathbb{Z}}$ are i.i.d.\ sequences with the same marginals as $\chil_n$ and $\chiu_n$. We then have the following result.

\begin{proposition}\label{prop:limits}
For $i\in\{-Kn,...,Kn\}$, define
\begin{align*}
\Pu_n(i)&:=\sum_{k=0}^{i-1} e^{-2 \lambda k/ (\rho n)}r^{0,0}(\omega_k,\omega_{k+1})\xiu_n(k)\mathbf{1}_{\B_n}(k),\\
\Pl_n(i)&:=\sum_{k=0}^{i-1} e^{-2 \lambda k/ (\rho n)}r^{0,0}(\omega_k,\omega_{k+1})\xil_n(k)\mathbf{1}_{\B_n}(k).
\end{align*}
Then $(n^{-1/\rho} \Pu(\floor{un}))_{u\in[-K,K]}$ and $(n^{-1/\rho} \Pl(\floor{un}))_{u\in[-K,K]}$ both converge in distribution to $(S^{\beta,\lambda}(u))_{u\in[-K,K]}$. Moreover,
\begin{align}\label{eq:dist_u}
\sup_{i\in\{-Kn,...Kn\}}n^{-1/\rho}\left|\Pu_n(i)-\Pl_n(i)\right|\xrightarrow[n\to\infty]{\mathbf{P}}0.
\end{align}
\end{proposition}

\begin{proof}
In light of \eqref{eq:limitu} and \eqref{eq:limitl}, we also define
\begin{align*}
\xi(i):=\lim_{n\to\infty}\xiu_n(i)=\lim_{n\to\infty}\xil_n(i)=\left(\sum_{j,k=0}^{\infty} e^{-(\omega_i^{(i)}-\omega_{i-j}^{(i)})-(\omega_{i+1+k}^{(i)}-\omega_{i+1}^{(i)})-\beta U(E_{i-j}^{(i)},E_{i+1+k}^{(i)})}\right)^{-1}
\end{align*}
and
\begin{align*}
V_n^\lambda(i):=\sum_{j=0}^{i-1}  e^{-2 \lambda j/ (\rho n)}r^{0,0}(\omega_j,\omega_{j+1})\xi(j).
\end{align*}
Note that $V_n^0(i)$ is a sum of i.i.d.\ random variables whose law does not depend on $n$, and whose increments satisfy
\begin{align}
\mathbf{P}(r^{0,0}(\omega_0,\omega_{1})\xi(0)\geq t)&=\mathbf{E}\left[\mathbf{P}(r^{0,0}(\omega_{0},\omega_{1})\geq t/\xi(0)|\xi(0))\right]\nonumber\\
&=\mathbf{E}\left[\mathbf{P}(\omega_1-\omega_0\geq \log( t/\xi(0))|\xi(0))\right]\nonumber\\
&=\mathbf{E}[e^{-\rho\log(t/\xi(0))}]\nonumber\\
&=t^{-\rho}\mathbf{E}[(\chi^{\beta,0}(0))^\rho],\label{xxx}
\end{align}
where we have applied the fact that $\xi(0)\isDistr \chi^{\beta,0}(0)$. It readily follows that the rescaled process $(n^{-1/\rho}V_n^0(\floor{un}))_{u\in[-K,K]}$ converges to $(S^{\beta,0}(u))_{u\in[-K,K]}$ in distribution with respect to the Skorohod $J_1$-topology. Combining this convergence and \cite[Theorem 3.1]{kasahara} yields that
\begin{align*}
\left(n^{-1/\rho}V_n^\lambda(\floor{un})\right)_{u\in[-K,K]}\xrightarrow[n\to\infty]d(S^{\beta,\lambda}(u))_{u\in[-K,K]}.
\end{align*}
To complete the proof, it will thus be sufficient to show that
\begin{align}
\sup_{i\in\{-Kn,...,Kn\}}n^{-1/\rho}\left|\Pu_{n}(i)-V^{\lambda}_n(i)\right|&\xrightarrow[n\to\infty]{\mathbf{P}}0,\label{eq:PMprobconv}\\
\sup_{i\in\{-Kn,...,Kn\}}n^{-1/\rho}\left|\Pl_{n}(i)-V^{\lambda}_n(i)\right|&\xrightarrow[n\to\infty]{\mathbf{P}}0\label{eq:Pprobconv}.
\end{align}
We only check \eqref{eq:PMprobconv}, since the proof of \eqref{eq:Pprobconv} is similar. By definition, we have that
\begin{equation}
  \label{eq:U-V}
 \begin{split}
  &n^{-1/\rho} \sup_{i\in\{-Kn,...,Kn\}}\left|\Pu_{n}(i)-V^{\lambda}_n(i)\right|\\
  &\quad\leq e^{2 \lambda K/ \rho }n^{-1/\rho}\sum_{i={-Kn}}^{Kn-1}r^{0,0}(\omega_i,\omega_{i+1})|\xiu_n(i)-\xi(i)|\\
  &\qquad + e^{2 \lambda K/ \rho }n^{-1/\rho}\sum_{i\in\{-Kn,\dots,Kn-1\}\setminus \B_n}r^{0,0}(\omega_i,\omega_{i+1})|\xi(i)|
 \end{split}
\end{equation}
For the first term on the right-hand side, we note that \eqref{eq:heavy} implies that the sequence $n^{-1/\rho}\sum_{i=-Kn}^{Kn}r^{0,0}(\omega_i,\omega_{i+1})$ is tight. Thus for the event
\begin{align*}
  A_{2,n,L}:=\left\{n^{-1/\rho}\textstyle  \sum_{i=-Kn}^{Kn}r^{0,0}(\omega_i,\omega_{i+1})\leq L\right\},
\end{align*}
we have $\lim_{L\to \infty}\limsup_{n\to\infty} \mathbf{P}(A_{2,n,L}^c)=0$. On the other hand, since $\xi$ and $\xiu_n$ are bounded, the dominated convergence theorem yields
\begin{eqnarray*}
\lefteqn{\limsup_{n\rightarrow\infty}\mathbf{E}\left(  \1_{A_{2,n,L}}n^{-1/\rho}\sum_{i={-Kn}}^{Kn-1}r^{0,0}(\omega_i,\omega_{i+1})|\xiu_n(i)-\xi(i)|\right)}\\
&=&\limsup_{n\rightarrow\infty}\mathbf{E}\left(  \1_{A_{2,n,L}}n^{-1/\rho}\sum_{i={-Kn}}^{Kn-1}r^{0,0}(\omega_i,\omega_{i+1})\mathbf{E}\left(|\xiu_n(i)-\xi(i)|\:\vline\:\omega\right)\right)\\
&=&\limsup_{n\rightarrow\infty}\mathbf{E}\left(  \1_{A_{2,n,L}}n^{-1/\rho}\sum_{i={-Kn}}^{Kn-1}r^{0,0}(\omega_i,\omega_{i+1})\right)\mathbf{E}\left(|\xiu_n(0)-\xi(0)|\right)\\
&\leq &L\limsup_{n\rightarrow\infty}\mathbf{E}\left(|\xiu_n(0)-\xi(0)|\right)\\
&=&L\mathbf{E}\left(\lim_{n\rightarrow\infty}|\xiu_n(0)-\xi(0)|\right)\\
&=&0
\end{eqnarray*}
for each $L>0$. Putting the latter two observations together, we readily conclude that
 \begin{align*}
 n^{-1/\rho}\sum_{i={-Kn}}^{Kn-1}r^{0,0}(\omega_i,\omega_{i+1})|\xiu_n(i)-\xi(i)| \xrightarrow[\mathbf{P}]{n\to\infty} 0,
 \end{align*}
as desired. For the second term on the right-hand side of \eqref{eq:U-V}, we use the boundedness of $\xi$ and the well-known fact
 \begin{align}\label{eq:wkf}
  n^{-1/\rho}\sum_{i\in\{-Kn,\dots,Kn-1\}\setminus \B_n}r^{0,0}(\omega_i,\omega_{i+1})\xrightarrow[\mathbf{P}]{n\to\infty} 0,
 \end{align}
see, e.g., \cite[(3.7.4)]{durrett}. This completes the proof of \eqref{eq:PMprobconv}.
\end{proof}

\subsection{Proof of Theorem \ref{thm:approx}}\label{sec:distortion}
In this subsection we provide the proof of Theorem \ref{thm:approx} after presenting a preliminary technical result in Lemma \ref{lem:lower'}. Towards this end, we recall the definitions of $\Ru_n$ and $\Rl_n$ from \eqref{eq:defRu} and \eqref{eq:defRl}, respectively. We will prove that $\Ru_n$ and $\Rl_n$ can be approximated by the following, defined for $-Kn\le i\leq j\le Kn$:
\begin{align*}
  \Ru'_n(i,j):=\sum_{k\in \{i,...,j-1\}\cap \B_n}e^{-2\lambda k/(\rho n)}r^{0,0}(\omega_k,\omega_{k+1})\chiu_n(k),\\
  \Rl'_n(i,j):=\sum_{k\in \{i,...,j-1\}\cap \B_n}e^{-2\lambda k/(\rho n)}r^{0,0}(\omega_k,\omega_{k+1})\chil_n(k).
\end{align*}
\begin{lemma}
  \label{lem:lower'}
  \begin{align}
    \label{eq:lower'}
    &n^{-1/\rho}\sup_{-Kn\leq i\leq j\leq Kn}\left|\Rl_n(i,j)-\Rl_n'(i,j)\right|\xrightarrow[n\to\infty]{\mathbf{P}} 0,\\
    \label{eq:upper'}
    &n^{-1/\rho}\sup_{-Kn\leq i\leq j\leq Kn}\left|\Ru_n(i,j)-\Ru_n'(i,j)\right|\xrightarrow[n\to\infty]{\mathbf{P}} 0.
  \end{align}
\end{lemma}
\begin{proof}
  Let us define
    \begin{align*}
    A_{3,n}:=\left\{i/\rho-n^{3/4}\leq \omega_i\leq i/\rho+n^{3/4}\text{ for all }i=-Kn,...,Kn\right\},
    \end{align*}
  on which we have
    \begin{align}
      \label{eq:A_3}
    \sup_{k\in\{-Kn,...,Kn\}}\left|e^{-\lambda(\omega_k+\omega_{k+1})/n}-e^{-2\lambda k/(\rho n)}\right|\leq Cn^{-1/4}.
    \end{align}
  Standard moderate deviation estimates (see, e.g., \cite[Theorem 3.7.1]{DZ98}) and the union bound show that $\mathbf{P}(A_{3,n})\to 1$ as $n\to\infty$.
  Combining \eqref{eq:A_3} with
    \begin{align*}
      |\Rl_n(i,j)-\Rl'_n(i,j)|\le \sum_{k\in\{-Kn,...,Kn\}}\left|e^{-\lambda(\omega_k+\omega_{k+1})/n}-e^{-2\lambda k/(\rho n)}\right|r^{0,0}(\omega_k,\omega_{k+1})|\chil_n(k)|,
    \end{align*}
  the boundedness of $\chil_n$ and the tightness of $n^{-1/\rho}\sum_{k\in\{-Kn,...,Kn\}}r^{0,0}(\omega_k,\omega_{k+1})$ as in the proof of Proposition \ref{prop:limits}, we can verify \eqref{eq:lower'}.

  To show \eqref{eq:upper'}, note first that
  \begin{align*}
  &n^{-1/\rho}\sup_{-Kn\leq i\leq j \leq Kn}\left|\Ru_n(i,j)-\Ru'_n(i,j)\right|\\
  &\quad\leq
  Cn^{-1/\rho}\sum_{k\in\{-Kn,...,Kn\}}\left|e^{-\lambda(\omega_k+\omega_{k+1})/n}-e^{-2\lambda k/(\rho n)}\right|r^{0,0}(\omega_k,\omega_{k+1})\\
  &\quad\quad+Cn^{-1/\rho} \sum_{k\in\{-Kn,\dots,Kn-1\}\setminus\B_n}r^{0,\lambda/n}(\omega_k,\omega_{k+1})
  \end{align*}
  since $\chiu_n$ and $U$ are bounded. The first term here can be dealt with as in the first part of the proof. To handle the second term, we start by noting that, on $A_{3,n}$, we can apply the bound \eqref{eq:A_3} to deduce that $\sup_{k\in\{-Kn,...,Kn\}}e^{-\lambda(\omega_k+\omega_{k+1})/n}$ is bounded. Therefore, by taking $C$ larger, $r^{0,\lambda/n}$ can be replaced by $r^{0,0}$ on $A_{3,n}$. Hence we obtain from \eqref{eq:wkf} that the sum in question converges to 0 in $\mathbf{P}$-probability.
\end{proof}

\begin{proof}[Proof of Theorem \ref{thm:approx}]
To apply the convergence results for $\Pu_n$ and $\Pl_n$ from Proposition \ref{prop:limits}, we need to justify that $\chil_n$ and $\chiu_n$ in the definitions of $\Rl'_n$ and $\Ru'_n$ can be replaced by $\xil_n$ and $\xiu_n$ without changing the law too much.

To this end, we first restrict ourselves to an event $\{\B_n=B\}$, where $B$ is a possible realisation of $\mathcal{B}_n$ on the event $\Alower_n$, as defined at \eqref{eq:defAl}. In particular, $B$ is a subset of $\{-Kn,\dots,Kn\}$ whose elements are at least $2b_n+1$ away from each other (and also suitably far from the boundaries $-Kn$ and $Kn$). Note that conditionally on $\{\B_n=B\}$, the random variables $(\omega_{i+1}-\omega_{i})_{i\in \{-Kn,...,Kn\}\setminus B}$ are independent and each has a parameter $\rho$ exponential distribution conditioned to be smaller than $\log n^{3/4\rho}$. Moreover, recalling that both $\chil_n(i)$ and $\chiu_n(i)$ only depend on $(\omega_{i+k+1}-\omega_{i+k})_{k\in \{-b_n,...,b_n\}\setminus\{0\}}$ and the energy marks $(E_{i+k})_{k\in \{-b_n,...,b_n\}}$, we find that the pairs $(\chil_n(i),\chiu_n(i))$, $i\in B$, are conditionally independent on $\{\B_n=B\}$. It follows that the distribution of $((\chil_n(i),\chiu_n(i)))_{i\in B}$ conditioned on $\{\B_n=B\}$ is the same as that of $((\xil_n(i),\xiu_n(i)))_{i\in B}$ conditioned on $\{\B_n=B\}\cap \bigcap_{i\in B}\Alower_n(i)$, where
\begin{align*}
\Alower_n(i):=\left\{\omega^{(i)}\colon \omega^{(i)}_{i+j+1}-\omega^{(i)}_{i+j}<\log n^{3/(4\rho)}\text{ all }j\in\{-b_n,...,b_n\}\setminus\{0\}\right\}.
\end{align*}
Similarly to the explanation after \eqref{eq:chi}, we additionally have that $(\chil_n(i),\chiu_n(i))_{i\in B}$ are independent of $(\omega_{i+1}-\omega_{i})_{i\in B}$, and hence also of $(r^{0,0}(\omega_i,\omega_{i+1}))_{i\in B}$. From these considerations, we obtain that, for each $B$ of the relevant form,
\begin{align*}
\lefteqn{\mathbf P\left(\bigl\{\bigl(\Rl'_n(i,j),\Ru'_n(i,j)\bigr)_{-Kn\leq i\leq j\leq Kn}\in\cdot \bigr\}\cap
\{\B_n=B\}\cap \textstyle{\bigcap_{i\in B}}\Alower_n(i)\right)}\\
&=\mathbf P\left(\bigl\{\bigl(\Pl_n(j)-\Pl_n(i),\Pu_n(j)-\Pu_n(i)\bigr)_{-Kn\leq i\leq j\leq Kn}\in\cdot \bigr\}\cap \{\B_n=B\}\cap\textstyle{\bigcap_{i\in B}}\Alower_n(i)\right).
\end{align*}
Summing over the $B$ that satisfy the conditions of $\Alower_n$, we conclude that
\begin{equation}\label{eq:equalinlaw}
\begin{split}
&\mathbf P\left(\bigl\{\bigl(\Rl'_n(i,j),\Ru'_n(i,j)\bigr)_{-Kn\leq i\leq j\leq Kn}\in\cdot \bigr\}\cap\Alower_n\cap\textstyle{\bigcap_{i\in \B_n}}\Alower_n(i)\right)\\
&=\mathbf P\left(\bigl\{\bigl(\Pl_n(j)-\Pl_n(i),\Pu_n(j)-\Pu_n(i)\bigr)_{-Kn\leq i\leq j\leq Kn}\in\cdot \bigr\} \cap \Alower_n\cap\textstyle{\bigcap_{i\in \B_n}}\Alower_n(i)\right).
\end{split}
\end{equation}
We now prove that the probability of $\Alower_n\cap\textstyle{\bigcap_{i\in \B_n}}\Alower_n(i)$ tends to one as $n\to\infty$, so that the effect of restriction onto this event is negligible. From Lemma \ref{lem:aux}, we know that $\lim_{n\to\infty}\mathbf P\left(\Alower_n\right)=1$. In addition, by arguing as at \eqref{eq:anotherbigedge}, we can check that $\mathbf{P}(\Alower_n(i)) \ge 1-2n^{-1/2}$ for each $i$. Combining this with the readily checked fact that $\mathbf{P}(|\B_n| \le n^{1/4+\varepsilon})\to 1$ for any $\varepsilon>0$ and the union bound, we obtain
\begin{align}
\label{eq:large}
\lim_{n\to\infty}\mathbf P\left(\bigcap_{i\in \B_n}\Alower_n(i)\right)=1.
\end{align}
Therefore we can use Proposition \ref{prop:limits} and \eqref{eq:lower'} to conclude that
\begin{align}\label{eq:tight}
\left(n^{-1/\rho}\sign(u)\Rl_n(0,\floor{un})\right)_{u\in[-K,K]}\xrightarrow[n\to\infty]d (S^{\beta,\lambda}(u))_{u\in[-K,K]}
\end{align}
in the Skorohod $J_1$-topology, where we define $\Rl_n(0,i):=\Rl_n(i,0)$ for $i<0$.

Next, we show that
\begin{align}\label{eq:next}
n^{-1/\rho}\sup_{-Kn\leq i\leq j\leq Kn} \left|\Rl_n(i,j)-R^{\beta,\lambda/n,Kn}(\oomega_i,\oomega_j)\right|\xrightarrow[n\to\infty]{\mathbf{P}}0.
\end{align}
Recall that the uniform topology is stronger than the $J_1$-topology, so that \eqref{eq:tight} and \eqref{eq:next} imply \eqref{eq:convergence}. Moreover, since $\Rl_n$ satisfies, for $i\leq j$,
\begin{align*}
\Rl_n(i,j)=\sign(j)\Rl_n(0,j)-\sign(i)\Rl_n(0,i),
\end{align*}
we also obtain \eqref{eq:distortion} from \eqref{eq:next}. To prove \eqref{eq:next}, observe that by Propositions \ref{prop:upper} and \ref{prop:lower}, on $\Aupper_n\cap\Alower_n$, for every $-Kn\leq i\leq j\leq Kn$,
\begin{align}
&\left|R^{\beta,\lambda/n,Kn}(\oomega_i,\oomega_j)-\Rl_n(i,j)\right|\nonumber\\
&\leq \max\left\{\Rl_n(i,j)-\left(\Rl_n(i,j)^{-1}+\mathcal{E}_n\right)^{-1},|\Ru_n(i,j)-\Rl_n(i,j)|\right\}.\label{eq:here}
\end{align}
Towards estimating the first term, let
\begin{align*}
A_{4,n}:=\left\{\mathcal{E}_n\leq c_1n^2e^{-c_2n^{1/4}},\: \Rl_n(-Kn,Kn)\leq n^{1/2+1/\rho}\right\}.
\end{align*}
We know from Lemma \ref{lem:aux2} and \eqref{eq:tight} that
\begin{equation}
\mathbf{P}(A_{4,n})\xrightarrow[n\to\infty]{}1.
\label{eq:A_4}
\end{equation}
Observe that on $A_{4,n}$, for $n$ large enough, we have $\mathcal{E}_n\leq \Rl_n(i,j)^{-1}/n$ for all $-Kn\leq i\leq j\leq Kn$, and therefore
\begin{align*}
\Rl_n(i,j)-\left(\Rl_n(i,j)^{-1}+\mathcal{E}_n\right)^{-1}\leq \Rl_n(i,j)/n\leq n^{1/\rho-1/2}.
\end{align*}
Thus, by \eqref{eq:A_4}, for any $\eps>0$,
\begin{align*}
&\mathbf{P}\left(n^{-1/\rho}\sup_{-Kn\leq i\leq j\leq Kn}\left(\Rl_n(i,j)-(\Rl_n(i,j)^{-1}+\mathcal{E}_n)^{-1}\right)>\eps\right)\le\mathbf{P}(A_{4,n}^c)\xrightarrow[n\to\infty]{}0.
\end{align*}
Recalling \eqref{eq:here} and Lemma \ref{lem:aux}, to complete the proof it will thus suffice to show that
\begin{align*}
  n^{-1/\rho}\sup_{-Kn\leq i\leq j\leq Kn}\left|\Ru_n(i,j)-\Rl_n(i,j)\right|\xrightarrow[n\to\infty]{\mathbf{P}}0.
\end{align*}
By the triangle inequality, we have
\begin{align*}
&n^{-1/\rho}\sup_{-Kn\leq i\leq j\leq Kn} |\Ru_n(i,j)-\Rl_n(i,j)|
\leq n^{-1/\rho}\sup_{-Kn\leq i\leq j \leq Kn} |\Ru_n(i,j)-\Ru'_n(i,j)|\\
&+n^{-1/\rho}\sup_{-Kn\leq i\leq j\leq Kn} |\Rl_n(i,j)-\Rl'_n(i,j)|+n^{-1/\rho}\sup_{-Kn\leq i\leq j\leq Kn} |\Ru'_n(i,j)-\Rl_n'(i,j)|.
\end{align*}
The first two terms on the right-hand side tend to zero in $\mathbf{P}$-probability as $n\to\infty$ by \eqref{eq:lower'} and \eqref{eq:upper'}. For the final term, using \eqref{eq:equalinlaw} and \eqref{eq:large}, together with Lemma \ref{lem:aux} and \eqref{eq:dist_u}, we have
\begin{align*}
&\mathbf{P}\left(n^{-1/\rho}\sup_{-Kn\leq i\leq j\leq Kn}|\Ru'_n(i,j)-\Rl'_n(i,j)|>\eps\right)\\
&\quad\leq \mathbf{P}(\Alower_n^c)+\mathbf P\left(\bigcup_{i\in \B_n}\Alower_n(i)^c\right)+\mathbf{P}\left(2n^{-1/\rho}\sup_{i\in\{-Kn,...,Kn\}}|\Pu_n(i)-\Pl_n(i)|>\eps\right)\\
&\quad\xrightarrow[n\to\infty]{} 0.
\end{align*}
This completes the proof of Theorem \ref{thm:approx}.
\end{proof}

\section{Convergence of the invariant measure}\label{sec:meas}
Next, we derive a scaling limit for the speed measure associated with our reflected process. In particular, with
\begin{align*}
c^{\beta,\lambda/n,Kn}(\oomega_i)&:=\sum_{j\in\{-Kn,...,Kn\}}c^{\beta,\lambda/n,Kn}(\oomega_i,\oomega_j),
\end{align*}
we let $\mu^{\beta,\lambda/n,Kn}$ denote the measure on $\{\oomega_{-Kn},...,\oomega_{Kn}\}$ given by
\begin{align*}
\mu^{\beta,\lambda/n,Kn}(\{\oomega_i\}):=c^{\beta,\lambda/n,Kn}(\oomega_i).
\end{align*}
For this measure, we are able to prove the following.

\begin{theorem}\label{thm:measure}
$\mathbf{P}$-a.s., for every $-K\leq a<b\leq K$,
\begin{align}\label{mmm}
\lim_{n\to\infty}\frac 1n\mu^{\beta,\lambda/n,Kn}\left(\left\{\oomega_{an},...,\oomega_{bn}\right\}\right)=\mathbf{E}\left(c^{\beta,0}(\omega_0)\right)\int_a^b e^{2\lambda r/\rho}\dd r,
\end{align}
where $c^{\beta,0}(\omega_0)$ is defined below \eqref{generator} and satisfies $\mathbf{E}(c^{\beta,0}(\omega_0))\in(0,\infty)$.
\end{theorem}

\begin{proof} It is clear that $\mathbf{E}(c^{\beta,0}(\omega_0))>0$. Moreover,
\[\mathbf{E}\left(c^{\beta,0}(\omega_0)\right)=\sum_{j\in\mathbb{Z}\setminus\{0\}}\mathbf{E}\left(c^{\beta,0}(\omega_0,\omega_j)\right)\leq C\sum_{j=1}^\infty \mathbf{E}\left(e^{-\omega_j}\right)=C\sum_{j=1}^\infty \mathbf{E}\left(e^{-\omega_1}\right)^j<\infty.\]

Since the limiting measure is continuous, it will suffice to prove the limit at \eqref{mmm} for fixed $-K\leq a<b\leq K$. This essentially follows from an ergodic theorem but we need a little more work to take into account the fact that the bias $\lambda/n$ and the truncation at $\{-Kn,Kn\}$ makes $c^{\beta,\lambda/n,Kn}(\oomega_i)$ a non-stationary sequence. Let
\begin{align*}
A_{5,n}:=&\left\{\omega_{an}-\omega_{an-n^{1/3}}\geq n^{1/4},\omega_{bn+n^{1/3}}-\omega_{bn}\geq n^{1/4}\right\}\\
&\cap\left\{\omega_{-Kn+n^{1/3}}-\omega_{-Kn}\geq n^{1/4},\omega_{Kn}-\omega_{Kn-n^{1/3}}\geq n^{1/4}\right\}\\
&\cap\bigcap_{j\in\{-Kn,...,Kn\}}\left\{ \omega_{j}\in j/\rho +[-n^{2/3},n^{2/3}]\right\}\\
&\cap\bigcap_{j\geq n^{1/3}}\left\{\omega_{Kn+j}-\omega_{Kn}\geq j/(2\rho),\omega_{-Kn}-\omega_{-Kn-j}\geq j/(2\rho)\right\}.
\end{align*}
One can check that $\sum_n\mathbf{P}(A_{5,n}^c)<\infty$ by standard moderate deviations estimates, so that $A_{5,n}$ holds almost surely for $n$ large enough by the Borel-Cantelli lemma. Moreover, on $A_{5,n}$, for every $i,j\in\{-Kn+1,...,Kn-1\}$,
\begin{align}
c^{\beta,\lambda/n,Kn}(\oomega_i,\oomega_{Kn})&
\leq Ce^{-|\omega_{Kn}-\omega_i|+2\lambda \omega_{Kn}/n}\sum_{j\geq 0}e^{-(1-\lambda/n)|\omega_{Kn+j}-\omega_{Kn}|}\label{eq:boundary1}\\
&\leq Ce^{-|\omega_{Kn}-\omega_i|}n^{1/3},\nonumber
\end{align}
\begin{align}
c^{\beta,\lambda/n,Kn}(\oomega_{-Kn},\oomega_i)&
\leq Ce^{-|\omega_i-\omega_{-Kn}|}\sum_{j\geq 0}e^{-|\omega_{-Kn}-\omega_{-Kn-j}|}
\leq Ce^{-|\omega_i-\omega_{-Kn}|}n^{1/3},\nonumber\\
c^{\beta,\lambda/n,Kn}(\oomega_i,\oomega_j)&\leq Ce^{-|\omega_i-\omega_j|}\label{eq:middle}.
\end{align}
Assume that $-K<a<b<K$. Then, on $A_{5,n}$,
\begin{align*}
&\mu^{\beta,\lambda/n,Kn}\left(\{\oomega_{an},...,\oomega_{bn}\}\right)\\
=&\sum_{i=an}^{bn}\left(c^{\beta,\lambda/n,Kn}(\oomega_i,\oomega_{Kn})+c^{\beta,\lambda/n,Kn}(\oomega_i,\oomega_{-Kn})+\sum_{j=-Kn+1}^{Kn-1}c^{\beta,\lambda/n}(\omega_i,\omega_j)\right)\\
\leq &Cn^{2}e^{-n^{1/4}}+\sum_{i=an}^{bn}\sum_{j=an-n^{1/2}}^{bn+n^{1/2}}c^{\beta,\lambda/n}(\omega_i,\omega_j)\\
\leq& Cn^{2}e^{-n^{1/4}}+e^{2\lambda (b+n^{-1/3})/\rho}\sum_{i=an}^{bn}\sum_{j=an-n^{1/2}}^{bn+n^{1/2}}c^{\beta,0}(\omega_i,\omega_j)\\
\leq& C'n^{2}e^{-n^{1/4}}+e^{2\lambda (b+n^{-1/3})/\rho}\sum_{i=an}^{bn}c^{\beta,0}(\omega_i),
\end{align*}
where we have applied \eqref{eq:boundary1}--\eqref{eq:middle} to deduce the first two inequalities. We observe that $(c^{\beta,0}(\omega_i))_{i\in\Z}$ is stationary and ergodic, so that almost surely,
\begin{align}\label{eq:ergodic_upper}
\limsup_{n\to\infty}\frac 1n\mu^{\beta,\lambda/n,Kn}(\{\oomega_{an},...,\oomega_{bn}\})\leq e^{2\lambda b/\rho}(b-a)\mathbf{E}\left(c^{\beta,0}(\omega_0)\right).
\end{align}
A similar estimate shows
\begin{align*}
\liminf_{n\to\infty}\frac 1n\mu^{\beta,\lambda/n,Kn}\left(\{\oomega_{an},...,\oomega_{bn}\}\right)\geq e^{2\lambda a/\rho}(b-a)\mathbf{E}\left(c^{\beta,0}(\omega_0)\right).
\end{align*}
To obtain the claim, take $m\in\N$ and subdivide $\{\oomega_{an},...,\oomega_{bn}\}$ into $m$ sub-intervals of (approximately) equal length. Applying the previous bounds to each of them shows that almost surely,
\begin{align*}
\frac 1m\sum_{k=0}^{m-1}e^{2\lambda (a+(b-a)k/m)/\rho}\mathbf{E}\left(c^{\beta,0}(\omega_0)\right)
&\leq \liminf_{n\to\infty}\frac 1n\mu^{\beta,\lambda/n,Kn}(\oomega_{an},...,\oomega_{bn})\\
& \leq \limsup_{n\to\infty}\frac 1n\mu^{\beta,\lambda/n,Kn}(\oomega_{an},...,\oomega_{bn})\\
&\leq \frac 1m\sum_{k=1}^{m}e^{2\lambda (a+(b-a)k/m)/\rho}\mathbf{E}\left(c^{\beta,0}(\omega_0)\right).
\end{align*}
Now take $m\to\infty$ to see that both sides converge to $\mathbf{E}\left(c^{\beta,0}(\omega_0)\right)\int_a^be^{2\lambda r/\rho}\dd r$, which completes the proof in the case $-K<a$, $b<K$.

We next explain how to obtain \eqref{eq:ergodic_upper} when $-K<a<b=K$. In this case, we have that
\begin{align*}
&\mu^{\beta,\lambda/n,Kn}(\{\oomega_{an},...,\oomega_{Kn-n^{1/3}}\})\\
=&\sum_{i=an}^{Kn-n^{1/3}}\left(c^{\beta,\lambda/n,Kn}(\oomega_i,\oomega_{Kn})+c^{\beta,\lambda/n,Kn}(\oomega_i,\oomega_{-Kn})+\sum_{j=-Kn+1}^{Kn-1}c^{\beta,\lambda/n}(\omega_i,\omega_j)\right)\\
\leq &Cn^{2}e^{-n^{1/4}}+\sum_{i=an}^{Kn-n^{1/3}}\sum_{j=an-n^{1/2}}^{Kn}c^{\beta,\lambda/n}(\omega_i,\omega_j)\\
\leq& C'n^{2}e^{-n^{1/4}}+e^{2\lambda (K+n^{-1/3})/\rho}\sum_{i=an}^{Kn-n^{1/3}}c^{\beta,0}(\omega_i),
\end{align*}
where we used \eqref{eq:boundary1}--\eqref{eq:middle} as before. Moreover, from \eqref{eq:boundary1},
\begin{align*}
\mu^{\beta,\lambda/n,Kn}\left(\{\oomega_{Kn-n^{1/3}},...,\oomega_{Kn}\}\right)\leq Cn^{2/3}.
\end{align*}
Putting these estimates and the ergodic theorem for $(c^{\beta,0}(\omega_i))_{i\in\Z}$ together gives the desired result. The argument for the remaining case $-K=a$ is similar.
\end{proof}
\begin{remark}
\label{rem:VSproof}
To prove the scaling limit result for the variable-speed random walk mentioned in Remark~\ref{rem:CSvsVS}, we only have to modify the argument in this section. Indeed, it is the random walk with the same conductance $c^{\beta,\lambda/n}$, but a different speed measure, with the mass of $\omega_i$ given by $e^{\lambda\omega_i/n}$. The scaling limit of this speed measure can be understood by an argument that is similar to, but simpler than, that used to deduce Theorem \ref{thm:measure}.
\end{remark}

\section{Compact metric measure space convergence}\label{sec:mms}

The goal of this section is to give a metric measure space convergence statement that combines the results of the previous two sections.  The main conclusion is stated below as Theorem \ref{thm:metric}.

To present this, we let $\F_c$ denote the set of elements $(M,d,\mu,\rho,\Phi)$, where:
\begin{itemize}
 \item $(M,d)$ is a compact metric space;
 \item $\mu$ is a locally finite Borel regular measure on $M$;
 \item $\rho$ is a distinguished point in $M$;
 \item $\Phi:M\to\R$ is continuous.
\end{itemize}
In the current context, we will generally think of $(M,d)$ as the deformed `resistance space' mentioned in Section \ref{sec:method}, $\mu$ the invariant measure of the process, $\rho$ its initial position, and the function $\Phi$ an embedding that reverts the process back to `physical space'.

To define a notion of convergence on $\F_c$, we recall the spatial Gromov-Hausdorff-Prohorov topology of \cite{croydon2018}, which builds on the classical notion of the Gromov-Hausdorff topology (see \cite{bbi} for introductory material in this direction). Specifically, we introduce a metric $\Delta$ on $\F_c$ by defining $\Delta((M_1,d_1,\mu_1,\rho_1,\Phi_1),(M_2,d_2,\mu_2,\rho_2,\Phi_2))$ as
\begin{align*}
\inf_{\substack{\psi_1,\psi_2,\\(M,d),\mathcal C}}\left\{d_{\text{P}}(\mu_1\circ\psi_1^{-1},\mu_2\circ\psi_2^{-1})+\sup_{x_1,x_2\in\mathcal C}\left(d(\psi_1(x_1),\psi_2(x_2))+|\Phi_1(x_1)-\Phi_2(x_2)|\right)\right\},
\end{align*}
where the infimum is taken over metric spaces $(M,d)$, isometries $\psi_1:(M_1,d_1)\to(M,d)$ and $\psi_2:(M_2,d_2)\to(M,d)$, correspondences $\mathcal C\subseteq M_1\times M_2$ (i.e.\ subsets of $M_1\times M_2$ whose projections onto both $M_1$ and $M_2$ are surjective), and $d_{\text{P}}$ denotes the Prohorov metric on probability measures on $(M,d)$, as defined by
\begin{align*}
d_{\text{P}}(\nu_1,\nu_2):=\inf\left\{\eps>0:\nu_1(A)\leq \nu_2(A^\eps)+\eps\text{ for all }A\in\mathcal B(M)\right\},
\end{align*}
where $A^\eps:=\{x\in M:d(x,A)<\eps\}$ is the $\eps$-neighborhood of $A$, and $\mathcal B(M)$ is the Borel $\sigma$-algebra associated with $(M,d)$. As is noted in \cite[Section 7]{croydon2018}, it is possible to check that $(\F_c,\Delta)$ is a separable metric space, and it is with respect to this framework that the distributional convergence of the following result is stated. (Actually, in \cite{croydon2018}, the topology was presented for `resistance metric' spaces and the measures assumed to have full support. These restrictions are in fact met by all the spaces we consider in this section, but since they are not needed in the present discussion or in the proof of the separability of the space $(\F_c,\Delta)$, we omit them. Moreover, in \cite{croydon2018}, non-compact spaces were also considered, and the suitably extended topology called the spatial Gromov-Hausdorff-vague topology, but we do not need this generality here.) In this section, we write $S$ in place of $S^{\beta,\lambda}$ for simplicity.

\begin{theorem}\label{thm:metric}
Consider the spaces $\mathcal X_n\coloneqq\{\oomega_{-Kn},...,\oomega_{Kn}\}$ and $\mathcal X:=\overline{S([-K,K])}$ equipped with metrics $d_n\colon\mathcal X_n\times\mathcal X_n\to[0,\infty)$ and $d\colon\mathcal X\times\mathcal X\to[0,\infty)$, respectively, where
\begin{align*}
d_n(\oomega_i,\oomega_j)&:= n^{-1/\rho}R^{\beta,\lambda/n,Kn}(\oomega_i,\oomega_j)
\end{align*}
and $d$ is the restriction of the Euclidean metric to $\mathcal X$,
measures $\mu_n$ and $\mu$ given by
\begin{align*}
\mu_n\left(\{\oomega_i,...,\oomega_j\}\right)&:=n^{-1}\mu^{\beta,\lambda/n,Kn}\left(\left\{\oomega_i,...,\oomega_j\right\}\right),\\
\mu\left((l,r]\right)&:=\mathbf{E}\left(c^{\beta,0}(\omega_0)\right)\int_{S^{-1}(l)}^{S^{-1}(r)} e^{2\lambda x/\rho}\dd x,
\end{align*}
and embeddings $\Phi_n:\mathcal X_n\to\R$ and $\Phi:\mathcal X\to\R$ determined by
\begin{align*}
\Phi_n(\oomega_i)&:= n^{-1}\omega_i,\\
\Phi(u)&:=S^{-1}(u),
\end{align*}
with $S^{-1}$ denoting the right-continuous inverse of $S$. It is then the case that
\begin{align}\label{eq:distconv}
\left(\mathcal X_n,d_n,\mu_n,\oomega_0,\Phi_n\right)\xrightarrow[n\to\infty]d\left(\mathcal X,d,\mu,0,\Phi\right)
\end{align}
in the space $(\F_c,\Delta)$.
\end{theorem}

One convenient choice for the common metric space $(M,d)$ in the definition of $\Delta$ is to take the disjoint union of $M_1$ and $M_2$. We recall how the corresponding isometries are constructed. Let $(M_1,d_1)$ and $(M_2,d_2)$ be two metric spaces and let $\pi:M_1\to M_2$ be surjective. The distortion of $\pi$ is defined as follows:
\begin{align*}
\dis(\pi):=\sup_{x,x'\in M_1}|d_1(x,x')-d_2(\pi(x),\pi(x'))|.
\end{align*}
We then have the following (see the proof of \cite[Theorem 7.3.25]{bbi}).

\begin{lemma}\label{thm:dist}
Let $M:=(\{1\}\times M_1)\cup(\{2\}\times M_2)$ be the disjoint union of $M_1$ and $M_2$, and define
\begin{align*}
d((i,z),(j,z')):=
\begin{cases}
d_1(z,z'),&\text{ if }i=j=1,\\
d_2(z,z'),&\text{ if }i=j=2,\\
\frac 12\dis(\pi)+\inf_{x\in M_1}\{d_1(z,x)+d_2(\pi(x),z')\}&\text{ if }i=1,j=2,\\
\frac 12\dis(\pi)+\inf_{x\in M_1}\{d_1(z',x)+d_2(\pi(x),z)\}&\text{ if }i=2,j=1.
\end{cases}
\end{align*}
This function $d\colon M\to[0,\infty)$ is a metric, and moreover the canonical embeddings $M_1\to \{1\}\times M_1\subset M$ and $M_2\to \{2\}\times M_2\subset M$ are isometries.
\end{lemma}

\begin{proof}[Proof of Theorem \ref{thm:metric}]
We introduce a map $\pi_n:\mathcal X_n\to\R$ by setting
\begin{align*}
\pi_n(\oomega_i):= \sign(i)n^{-1/\rho}R^{\beta,\lambda/n,Kn}(\oomega_0,\oomega_i),
\end{align*}
and an element $(\mathcal X_n',d_n',\mu_n',0,\Phi_n')$ of $\F_c$ by defining
\begin{align*}
\mathcal X_n'&:=\pi_n(\mathcal X_n),\\
d_n'(\pi_n(\oomega_i),\pi_n(\oomega_j))&:=|\pi_n(\oomega_i)-\pi_n(\oomega_j)|,\\
\mu_n'&:=\mu_n\circ \pi_n^{-1},\\
\Phi_n'(\pi_n(\oomega_i))&:=\inf_{j:\:\pi_n(\oomega_j)=\pi_n(\oomega_i)}\Phi_n(\oomega_j).
\end{align*}
We will prove \eqref{eq:distconv} by showing
\begin{align}
\Delta\left((\mathcal X_n,d_n,\mu_n,\oomega_0,\Phi_n),(\mathcal X_n',d_n',\mu_n',0,\Phi_n')\right)&\xrightarrow[n\to\infty]{\mathbf{P}}0\label{eq:distconv1},\\
(\mathcal X_n',d_n',\mu_n',0,\Phi_n')&\xrightarrow[n\to\infty]d(\mathcal X,d,\mu,0,\Phi).\label{eq:distconv2}
\end{align}

For \eqref{eq:distconv1}, we choose the metric space $M:=(\{1\}\times\mathcal X_n)\cup(\{2\}\times\mathcal X_n')$ and metric $d$ as described in Lemma \ref{thm:dist}, where $\pi_n$ plays the role of $\pi$. Let
\begin{align*}
\mathcal C:=\left\{((1,\oomega_i),(2,\pi_n(\oomega_i))):i\in\{-Kn,...,Kn\}\right\}
\end{align*}
denote the associated correspondence between $\{1\}\times\mathcal X_n$ and $\{2\}\times\mathcal X_n'$, and observe that
\begin{align}\label{eq:close}
d((1,\oomega_i),(2,\pi_n(\oomega_i))=\frac 12\dis(\pi_n)
\end{align}
for every $i\in\{-Kn,...,Kn\}$. Writing $\psi_1$ and $\psi_2$ for the isometric embeddings of $(\mathcal{X}_n,d_n)$ and  $(\mathcal{X}'_n,d'_n)$ into $(M,d)$, it readily follows from \eqref{eq:close} that, for every $\eps> \tfrac12\dis(\pi_n)$ and $A\subseteq M$: if $\pi_n(\oomega_i)\in \psi_2^{-1}(A)$, then $\oomega_i\in \psi_1^{-1}(A^\varepsilon)$. That is, $\pi_n^{-1}\circ\psi_2^{-1}(A)\subseteq  \psi_1^{-1}(A^\varepsilon)$, which implies in turn that
\begin{align*}
\mu_n\circ\psi_1^{-1}(A^\eps)\geq \mu_n'\circ\psi_2^{-1}(A),
\end{align*}
and therefore
\begin{align*}
d_{\text{P}}(\mu_n\circ \psi_1^{-1},\mu_n'\circ\psi_2^{-1})\leq \frac 12\dis(\pi_n).
\end{align*}
It follows that
\begin{align*}
\Delta\left((\mathcal X_n,d_n,\mu_n,\oomega_0,\Phi_n),(\mathcal X_n',d_n',\mu_n',0,\Phi_n')\right)\leq \dis(\pi_n) +\varepsilon_n,
\end{align*}
where
\begin{align*}
\varepsilon_n&:=\sup_{i\in\{-Kn,...,Kn\}}|\Phi_n(\oomega_i)-\Phi_n'(\pi_n(\oomega_i))|.
\end{align*}
Now, since the limiting process in \eqref{eq:convergence} is strictly increasing, it holds that
\begin{align*}
\varepsilon_n&\leq \sup_{i\in\{-Kn,...,Kn\}}\sup_{j:\:\pi_n(\oomega_j)=\pi_n(\oomega_i)}|\Phi_n(\oomega_i)-\Phi_n(\oomega_j)|\xrightarrow[n\to\infty]{\mathbf{P}}0.
\end{align*}
Consequently, to complete the proof of \eqref{eq:distconv1}, it suffices to note that
\begin{align*}
\lefteqn{\dis(\pi_n)}\\
&=\max_{i,j\in\{-Kn,...,Kn\}}|d_n(\oomega_i,\oomega_j)-d_n'(\pi_n(\oomega_i),\pi_n(\oomega_j))|\\
&=n^{-1/\rho}\max_{{\textstyle\mathstrut}i,j\in\{-Kn,...,Kn\}}\left|R^{\beta,\lambda/n,Kn}(\oomega_i,\oomega_j)
\vphantom{\left.|\sign(j)R^{\beta,\lambda/n,Kn}(\oomega_0,\oomega_j)-\sign(i)R^{\beta,\lambda/n,Kn}(\oomega_0,\oomega_i)|\right|}\right.\\
&\hspace{100pt}\left.\vphantom{R^{\beta,\lambda/n,Kn}(\oomega_i,\oomega_j)}-|\sign(j)R^{\beta,\lambda/n,Kn}(\oomega_0,\oomega_j)-\sign(i)R^{\beta,\lambda/n,Kn}(\oomega_0,\oomega_i)|\right|\\
&\xrightarrow[n\to\infty]{\mathbf{P}}0,
\end{align*}
where the final line follows from \eqref{eq:distortion}.

It remains to prove \eqref{eq:distconv2}. We observe that $\mathcal X_n'$ and $\mathcal X$ are subsets of $\R$ equipped with the Euclidean distance, so we can choose $(M,d)=(\R,|\cdot|)$. The injections $\psi_1:\mathcal X_n'\to M$ and $\psi_2:\mathcal X\to M$ are isometries. To conclude, it suffices to construct a sequence
\begin{align*}
(\widetilde{\mathcal X},|\cdot|,\widetilde\mu,0,\widetilde\Phi),\:(\widetilde{\mathcal X}'_1,|\cdot|,\widetilde\mu'_1,0,\widetilde\Phi'_1),\:(\widetilde{\mathcal X}'_2,|\cdot|,\widetilde\mu'_2,0,\widetilde\Phi'_2),...
\end{align*}
satisfying, for all $n\in\N$,
\begin{align}
(\mathcal X_n',|\cdot|,\mu_n',0,\Phi_n')&\isDistr (\widetilde{\mathcal X}_n',|\cdot|,\widetilde\mu_n',0,\widetilde \Phi_n')\label{eq:equally1}\\
(\mathcal X,|\cdot|,\mu,0,\Phi)&\isDistr (\widetilde{\mathcal X},|\cdot|,\widetilde\mu,0,\widetilde \Phi)\label{eq:equally2}.
\end{align}
such that, almost surely,
\begin{align}
d_{\text{P}}(\widetilde\mu_n',\widetilde\mu)&\xrightarrow[n\to\infty]{}0\label{eq:weak_conv}\\
\inf_{\substack{\mathcal C_n\subseteq\widetilde{\mathcal X}'_n\times\widetilde{\mathcal X}\\\text{correspondence}}}\sup_{(x,y)\in\mathcal C_n}\left(|x-y|+|\widetilde\Phi_n'(x)-\widetilde\Phi(y)|\right)&\xrightarrow[n\to\infty]{}0.\label{eq:conv_hausdorff}
\end{align}
To construct this coupling, let us write
\begin{align*}
S_n&:=\left(\pi_n(\oomega_{\floor{un}})\right)_{u\in[-K,K]},
\end{align*}
which we interpret as a random variable on the space $D([-K,K],\mathbb{R})$ of c\'adl\'ag functions equipped with the $J_1$-topology. Recall that this topology is generated by the metric
\begin{align*}
d_{J_1}(f,g):=\inf_{\lambda\in\Lambda}\sup_{u\in[-K,K]}|\lambda(u)-u|+\sup_{u\in[-K,K]}|f\circ\lambda(u)-g(u)|,
\end{align*}
where $\Lambda$ is the set of strictly increasing, surjective functions $\lambda\colon[-K,K]\to[-K,K]$ such that $\lambda$ and $\lambda^{-1}$ are continuous. Moreover, define a measure $\nu_n$ on $\{i/n:i=-Kn,...,Kn\}$ by setting
\begin{align*}
\nu_n(\{i/n\}):=\mu_n(\{\oomega_{i}\}),
\end{align*}
which we interpret as a random variable in the space $\mathcal M([-K,K])$ of finite measures on $[-K,K]$, equipped with the topology of weak convergence. Observe that for every $u\in[-K,K]$,
\begin{align}\label{eq:product}
\mu_n'(\{S_n(u)\})=\nu_n\circ S_n^{-1}(\{S_n(u)\}),
\end{align}
where $S_n^{-1}$ is the usual set inverse, i.e.\ $S_n^{-1}(A):=\{v:\:S_n(v)\in A \}$. Finally, recall from Theorems \ref{thm:approx} and \ref{thm:measure} that $S_n\to S:=S^{\beta,\lambda}$ in distribution and $\nu_n\to \nu$ almost surely, where $\nu\in\mathcal M([-K,K])$ is defined by
\begin{align*}
\nu(\dd r)=\mathbf{E}\left(c^{\beta,0}(\omega_0)\right)e^{2\lambda r/\rho}\1_{r\in[-K,K]}\dd r.
\end{align*}
Since the latter limit is deterministic, we can conclude that $(S_n,\nu_n)\xrightarrow d (S,\nu)$ jointly in $D([-K,K],\mathbb{R})\times\mathcal M[-K,K]$. This space is separable, so by the Skorohod embedding theorem, there exists a probability space supporting $(\widetilde S,\widetilde\nu),(\widetilde S_1,\widetilde\nu_1),(\widetilde S_2,\widetilde\nu_2),...$ such that
\begin{align}
(S_n,\nu_n)&\isDistr (\widetilde S_n,\widetilde\nu_n),\label{eq:coupling1}\\
(S,\nu)&\isDistr (\widetilde S,\widetilde\nu),\label{eq:coupling2}
\end{align}
and such that almost surely
\begin{align}\label{eq:coupling3}
d_{J_1}(\widetilde S_n,\widetilde S)+d_{\text{P}}(\widetilde \nu_n,\widetilde\nu)\xrightarrow[n\to\infty]{}0.
\end{align}
To construct the desired coupling, let $(\widetilde {\mathcal X},|\cdot|,\widetilde \mu,0,\widetilde \Phi)$ be defined as $(\mathcal X,|\cdot|,\mu,0,\Phi)$ with $S$ replaced by $\widetilde S$, and set
\begin{align*}
\widetilde{\mathcal X}_n'&:=\widetilde S_n([-K,K]),\\
\widetilde \mu_n'(\{S_n(u)\})&:=\widetilde \nu_n\circ S_n^{-1}(\{\widetilde S_n(u)\}),\\
\widetilde \Phi_n'(\widetilde S_n(u))&:=\inf \widetilde S_n^{-1}(\{\widetilde S_n(u)\}).
\end{align*}
The coupling properties \eqref{eq:equally1} and \eqref{eq:equally2} now follow from \eqref{eq:coupling1} and \eqref{eq:coupling2}, together with \eqref{eq:product}. Next, to verify \eqref{eq:conv_hausdorff}, we construct a suitable correspondence $\mathcal C_n$. Let $\lambda_n\in\Lambda$ be such that
\begin{align}\label{eq:deflambda}
\sup_{u\in[-K,K]}|\lambda_n(u)-u|+\sup_{u\in[-K,K]}|\widetilde S(u)-\widetilde S_n\circ\lambda_n(u)|\leq d_{J_1}(\widetilde S_n,\widetilde S)+n^{-1},
\end{align}
and define $\mathcal C_n:=\mathcal C_n^1\cup\mathcal C_n^2$, where
\begin{align*}
\mathcal C^1&:=\left\{\left(\widetilde S_n(\lambda_n(u)),\widetilde S(u)\right):u\in[-K,K]\right\}\\
\mathcal C_n^2&:=\left\{\left(\widetilde S_n(\lambda_n(u)^-),\widetilde S(u^-)\right):u\in(-K,K]\right\}.
\end{align*}
Note that, since $\widetilde S$ is almost-surely strictly increasing, $\widetilde S^{-1}(\widetilde S(u))=\widetilde S^{-1}(\widetilde S(u^-))=u$ for all $u\in[-K,K]$, where we write $\widetilde S^{-1}$ for the (right-)continuous inverse of $\widetilde S$. We therefore have
\begin{eqnarray*}
\lefteqn{\sup_{(x_1,x_2)\in\mathcal C_n^1}\left|\widetilde\Phi_n(x_1)-\widetilde\Phi(x_2)\right|}\\
&=&\sup_{u\in[-K,K]}\left|\inf \widetilde S_n^{-1}(\{\widetilde S_n(\lambda_n(u))\})-\widetilde S^{-1}(\widetilde S(u))\right|\\
&\leq&\sup_{\substack{u,v\in[-K,K]:\\\tilde{S}_n(\lambda_n(v))=\tilde{S}_n(\lambda_n(u))}}|\lambda_n(v)-u|\\
&\leq & \sup_{\substack{u,v\in[-K,K]:\\|\tilde{S}(v)-\tilde{S}(u)|\leq 2d_{J_1}(\tilde{S}_n,\tilde{S})+2n^{-1}}}|v-u|
+ d_{J_1}(\tilde{S}_n,\tilde{S})+n^{-1}.
\end{eqnarray*}
Similarly, since we also have that $\sup_{u\in(-K,K]}|\widetilde S(u^-)-\widetilde S_n\circ\lambda_n(u^-)|\leq d_{J_1}(\widetilde S_n,\widetilde S)+n^{-1}$,
\begin{eqnarray*}
\sup_{(x_1,x_2)\in\mathcal C_n^2}|\widetilde\Phi_n(x_1)-\widetilde\Phi(x_2)|&\leq&\sup_{\substack{u\in(-K,K],\: v\in[-K,K]:\\\tilde{S}_n(\lambda_n(v))=\tilde{S}_n(\lambda_n(u^-))}}|\lambda_n(v)-u|\\
&\leq & \sup_{\substack{u\in(-K,K],\:v\in[-K,K]:\\|\tilde{S}(v)-\tilde{S}(u^-)|\leq 2d_{J_1}(\tilde{S}_n,\tilde{S})+2n^{-1}}}|v-u|
+ d_{J_1}(\tilde{S}_n,\tilde{S})+n^{-1}.
\end{eqnarray*}
Next, by \eqref{eq:deflambda},
\begin{align*}
\sup_{(x,y)\in\mathcal C_n^1}|x-y|
&=\sup_{u\in[-K,K]}|\widetilde S(u)-\widetilde S_n(\lambda_n(u))|\leq d_{J_1}(\widetilde S_n,\widetilde S)+n^{-1}.
\end{align*}
Finally, since $\lambda_n$ is continuous and strictly increasing, we also have
\begin{align*}
\sup_{(x,y)\in\mathcal C_n^2}|x-y|&=\sup_{u\in(-K,K]}|\widetilde S(u^-)-\widetilde S_n(\lambda_n(u)^-)|\\
&=\sup_{u\in(-K,K]}\lim_{\eps\downarrow 0}|\widetilde S(u-\eps)-\widetilde S_n(\lambda_n(u-\eps))|\\
&\leq d_{J_1}(\widetilde S_n,\widetilde S)+n^{-1}.
\end{align*}
Combining these inequalities with \eqref{eq:coupling3}, and again appealing to the fact that $\tilde{S}$ is almost-surely strictly increasing, we obtain \eqref{eq:conv_hausdorff}, as desired. For \eqref{eq:weak_conv}, take $v\in\mathbb{R}$ and observe
\begin{align*}
\widetilde\mu_n'((-\infty,v])&=\widetilde \nu_n\left(u\in[-K,K]:\:\widetilde S_n(u)\leq v\right)\\
&\rightarrow \widetilde\nu\left(u\in[-K,K]:\:\widetilde S(u)\leq v\right)\\
&=\widetilde \mu((-\infty,v]),
\end{align*}
where we have applied \eqref{eq:coupling1}, \eqref{eq:coupling2}, the fact that $\widetilde S$ is almost surely strictly increasing, and the continuity of the limit measure to deduce the convergence statement. This shows that $\widetilde\mu_n\to \widetilde \mu$ weakly, and hence establishes \eqref{eq:weak_conv}.
\end{proof}

\section{Proof of Theorem \ref{thm:main}}\label{sec:mr}

Before establishing our main result, Theorem \ref{thm:main}, we first give the corresponding result for a random walk on $\{\oomega_{-Kn},...,\oomega_{Kn}\}$, for which a scaling limit readily follows from what we have already proved in conjunction with known results for resistance forms.

To enable us to continue in this direction, let us briefly review the resistance form theory to which we will appeal (see \cite{Cint} for an extended version of this introduction). Resistance forms were introduced in the study of analysis on fractals, where Kigami also formulated the idea of a resistance metric on a general space, see \cite{Kigaof,Kigres} for background. In particular, for $\mathcal{X}$ a set, a function $R:\mathcal{X}\times \mathcal{X}\rightarrow \mathbb{R}$ is a resistance metric on $\mathcal{X}$ if, for every finite $V \subseteq \mathcal{X}$, one can find a weighted (i.e.\ equipped with $(0,\infty)$-valued conductances) connected, simple graph with vertex set $V$ for which $R|_{V\times V}$ is the associated effective resistance. (That $R$ is indeed a metric readily follows from the commute time identity for finite graphs \cite{CRRST,Tet}.) Moreover, Kigami showed that naturally associated with a resistance metric space $(\mathcal{X},R)$, there exists a so-called `resistance form', that is, a quadratic form $(\mathcal{E},\mathcal{F})$ on $\mathcal{X}$ that satisfies certain properties and is characterised by the relation
\[R(x,y)^{-1}=\inf\left\{\mathcal{E}(f,f):\:f\in\mathcal{F},\:f(x)=0,\:f(y)=1\right\},\qquad \forall x,y\in\mathcal{X},\:x\neq y.\]
(Cf.\ \eqref{effres}.) Importantly, from the point of view of probability theory, if a resistance metric space $(\mathcal{X},R)$ is compact, then the corresponding resistance form $(\mathcal{E},\mathcal{F})$ is actually a regular Dirichlet form on $L^2(\mathcal{X},\mu)$ for any finite Borel measure $\mu$ of full support (see \cite[Corollary 6.4 and Theorem 9.4]{Kigres}), and so in turn associated with a Hunt process $((X_t)_{t\geq 0},(P_x)_{x\in\mathcal{X}})$. (NB. The locally compact case is also considered in \cite{Kigres}, but we will not need that in this section.) Now, related to the discussion of the previous section, for metric measure spaces, a natural topology is given by the Gromov-Hausdorff-Prohorov distance, and in \cite{croydon2018} it was shown that the laws of the stochastic processes associated with resistance metric measure spaces are, in a certain sense, continuous with respect to this topology. More precisely, it was shown that if $(\mathcal{X}_n,R_n,\mu_n,\rho_n,\Phi_n )\xrightarrow{d}(\mathcal{X},R,\mu,\rho,\Phi )$ in $\F_c$ and also $(\mathcal{X}_n,R_n)$ and $(\mathcal{X},R)$ are resistance metric spaces, and $\mu_n$ and $\mu$ have full support, then
\[\Phi_n(X^n)\xrightarrow{d}\Phi(X)\]
in $D(\mathbb{R}_+,\mathbb{R})$ with respect to the annealed law, where $X^n$ and $X$,  the Markov processes associated with the relevant spaces, are started from $\rho_n$ and $\rho$, respectively. (This is a simplified version of \cite[Theorem 7.2]{croydon2018}.) We note that the latter result builds on the work \cite{ALW,CHK}, with \cite{ALW} covering the case of tree-like metric spaces, and \cite{CHK} giving a similar result under a uniform volume doubling assumption.

Returning to the setting of the present article, we recall the space $(\mathcal{X}_n,d_n,\mu_n,{\oomega}_0,\Phi_n)$, as defined in the statement of Theorem \ref{thm:metric}. Now, the effective resistance metric on a finite weighted graph is a resistance metric in the sense of Kigami, and it is an elementary exercise to check that the $\Phi_n$-embedding $(Z_n(t))_{t\geq 0}$ of the Markov process associated with $(\mathcal{X}_n,d_n,\mu_n)$ satisfies
\[\left(Z_n(t)\right)_{t\geq 0}\buildrel{d}\over{=}\left(n^{-1}X^{\beta,\lambda/n,Kn}_{n^{1+1/\rho}t}\right)_{t\geq 0},\]
where $X^{\beta,\lambda/n,Kn}$ is the random walk on $\{\oomega_{-Kn},...,\oomega_{Kn}\}$ with generator
\begin{align}\label{eq:truncated_generator}
(L^{\beta,\lambda/n,Kn}f)(\oomega_i)&:=\sum_{j\in\{-Kn,...,Kn\}}\frac{c^{\beta,\lambda/n,Kn}(\oomega_i,\oomega_j)}{c^{\beta,\lambda/n,Kn}(\oomega_i)}(f(\oomega_j)-f(\oomega_i)).
\end{align}
We denote by $\mathbb{P}^{\beta,\lambda/n,Kn}$ the annealed law of $X^{\beta,\lambda/n,Kn}$ (defined similarly to \eqref{annealed}). As for the limiting space $(\mathcal{X},d,\mu,0,\Phi)$ from Theorem \ref{thm:metric}, we have from \cite[Section 16]{Kigres} (and the trace properties for resistance and Dirichlet forms of \cite[Section 8]{Kigres} and \cite[Theorem 6.2.1]{FOT}, respectively) that $(\mathcal{X},d)$ is a resistance metric space and the Markov process corresponding to $(\mathcal{X},d,\mu)$ is Brownian motion time-changed according to $\mu=\mu^{\beta,\lambda,K}$, i.e.\ the process $(B_{H^{\beta,\lambda,K}_t})_{t\geq 0}$, where
\begin{align*}
\mu^{\beta,\lambda,K}\left([a,b]\right)&:= \mathbf{E}\left(c^{\beta,0}(\omega_0)\right)\int_{(S^{\beta,\lambda})^{-1}(a\vee -K)}^{(S^{\beta,\lambda})^{-1}(b\wedge K)}e^{2\lambda r/\rho}\dd r
\end{align*}
and
\begin{align*}
H_t^{\beta,\lambda,K}&:=\inf\left\{s\geq 0:\int_\R L^B_s(x)\mu^{\beta,\lambda,K}(\dd x)>t\right\}.
\end{align*}
Thus the $\Phi$-embedded version of this process is given by
\begin{align*}
Z_t^{\beta,\lambda,K}&:=(S^{\beta,\lambda})^{-1}\left(B_{H^{\beta,\lambda,K}_t}\right),
\end{align*}
and we will write $\mathbb{P}^{\beta,\lambda,K}$ for the law of $Z^{\beta,\lambda,K}$. Note that since $\mu^{\beta,\lambda,K}$ is supported on $\overline {S^{\beta,\lambda}([-K,K])}$, the process $Z^{\beta,\lambda,K}$ takes values in $[-K,K]$ and is reflected at the boundary $\{\pm K\}$. In view of these preparations, the following result is now straightforward to prove.

\begin{proposition}\label{prop:fixedK}
For every $\rho<1$ and $\beta,\lambda,K\geq 0$, it holds that as $n\to\infty$,
\[\mathbb{P}^{\beta,\lambda/n,Kn}\left((n^{-1}X_{n^{1+1/\rho}t})_{t\geq 0}\in\cdot\right)\]
converge weakly as probability measures on $D([0,\infty),\mathbb{R})$ to the law of $Z^{\beta,\lambda,K}$.
\end{proposition}

\begin{proof}
Recall the elements $(\mathcal X_n,d_n,\mu_n,\oomega_0,\Phi_n)$ and $(\mathcal X,d,\mu,0,\Phi)$ from the statement of Theorem \ref{thm:metric}, and let $Z_n$ and $Z$ denote the corresponding processes on $\mathcal X_n$ and $\mathcal X$, as introduced above. Since $\Phi_n(Z_n)$ has annealed law $\mathbb{P}^{\beta,\lambda/n,Kn}$ and the annealed law of $\Phi(Z)$ matches the law of $Z^{\beta,\lambda,K}$, the conclusion follows from Theorem \ref{thm:metric} and \cite[Theorem 7.2]{croydon2018}.
\end{proof}

Our proof of Theorem \ref{thm:main} will be based on three lemmas. The first two lemmas are estimates for excess times under $\mathbb{P}^{\beta,\lambda/n,Kn}$ and under $\mathbb{P}^{\beta,\lambda,K}$. (Recall that we write $\mathbb{P}^{\beta,\lambda,K}$ for the law of $Z^{\beta,\lambda,K}$ and $\mathbb{P}^{\beta,\lambda}$ for the law of $Z^{\beta,\lambda}$.) In particular, in what follows, we let $\tau_a(Z)$ denote the excess time of $a$ for some process $Z$, that is
\begin{align*}
\tau_a(Z):=
\begin{cases}
\inf\{s\geq 0:\:Z_s\geq a\},&\text{ if }a>0,\\
\inf\{s\geq 0:\:Z_s\leq a\},&\text{ if }a<0.
\end{cases}
\end{align*}
Even though the excess time is not continuous as a function on $D([0,\infty),\mathbb{R})$, it is possible to deduce the following bound. We highlight that the proofs of this and the subsequent two lemmas are postponed until the end of the section.

\begin{lemma}\label{lem:limsup}
For every $t>0$,
\begin{equation*}
 \begin{split}
&\limsup_{n\to\infty}\mathbb{P}^{\beta,\lambda/n,Kn}\left(\tau_{Kn}(X)\wedge \tau_{-Kn}(X)\leq tn^{1/\rho}\right)\\
&\qquad \qquad \leq \mathbb{P}^{\beta,\lambda,K}\left(\tau_{K-1}(Z)\wedge\tau_{-K+1}(Z)\leq t+1\right).
 \end{split}
\end{equation*}
\end{lemma}

We are going to derive Theorem \ref{thm:main} from Proposition \ref{prop:fixedK}. It is clear that the laws of $Z^{\beta,\lambda,K}$ and $Z^{\beta,\lambda}$ agree until the first hitting time of $\{-K,K\}$ by $Z^{\beta,\lambda,K}$, and so it remains to show that, as $K\to\infty$, the hitting times of $\{-K,K\}$ by $Z^{\beta,\lambda,K}$ diverge in probability. In other words, the sequence $(\mathbb{P}^{\beta,\lambda,K})_{K\geq 0}$ is tight in $D([0,\infty),\mathbb{R})$. Note that this is not at all obvious: if $\lambda>0$, then the image $S^{\beta,\lambda}(\R)$ is bounded from above, and the Brownian motion $B$ will hit $S^{\beta,\lambda}(\infty):=\lim_{t\to\infty}S^{\beta,\lambda}(t)$ at some finite time $\zeta$. That is, the process $((S^{\beta,\lambda})^{-1}(B_t))_{t\geq 0}$ without a time change diverges in finite time. Note, however, that the mass at $x$ under the speed-measure $\mu^{\beta,\lambda}$ grows exponentially as $x\uparrow S^{\beta,\lambda}(\infty)$, so that the time-changed Brownian motion $B_{H_t}$ is slowed down as it approaches $S^{\beta,\lambda}(\infty)$. The next lemma shows that the explosion time $\zeta$ is `delayed until time $\infty$'.

\begin{lemma}\label{lem:nonexpl}
For any $u>0$,
\begin{align*}
\limsup_{K\to\infty}\mathbb{P}^{\beta,\lambda}\left(\tau_{K}(Z)\wedge\tau_{-K}(Z)\leq u\right)=0.
\end{align*}
\end{lemma}

As a final ingredient for the proof of Theorem \ref{thm:main}, we give some basic properties of $Z^{\beta,\lambda}$.

\begin{lemma}\label{zbllem}
The process $Z^{\beta,\lambda}$ is continuous. Moreover, conditional on $S^{\beta,\lambda}$, it is Markov.
\end{lemma}

Before proving Lemmas \ref{lem:limsup}, \ref{lem:nonexpl} and \ref{zbllem}, we give the proof of the main theorem.

\begin{proof}[Proof of Theorem \ref{thm:main}]
Writing $Z_n:=(n^{-1}X_{n^{1+1/\rho}t})_{t\geq 0}$, for the convergence part of the result, it is enough to show that for all $f:D([0,T],\mathbb{R})\to\R$ bounded and continuous, where $T>0$,
\begin{align*}
\lim_{n\to\infty}\mathbb{E}^{\beta,\lambda/n}[f(Z_n)]=\mathbb{E}^{\beta,\lambda}[f(Z)].
\end{align*}
Clearly, for any $K\in\N$,
\begin{align*}
\mathbb{E}^{\beta,\lambda/n,Kn}\left[f(Z_n)\1_{\tau_{Kn}(Z_n)\wedge \tau_{-Kn}(Z_n)>Tn^{1/\rho}}\right]&=\mathbb{E}^{\beta,\lambda/n}\left[f(Z_n)\1_{\tau_{Kn}(Z_n)\wedge \tau_{-Kn}(Z_n)>Tn^{1/\rho}}\right],\\
\mathbb{E}^{\beta,\lambda,K}\left[f(Z)\1_{\tau_{K}(Z)\wedge \tau_{-K}(Z)>T}\right]&=\mathbb{E}^{\beta,\lambda/n}\left[f(Z)\1_{\tau_{K}(Z)\wedge \tau_{-K}(Z)>T}\right],
\end{align*}
and therefore
\begin{align*}
\left|\mathbb{E}^{\beta,\lambda/n,Kn}[f(Z_n)]-\mathbb{E}^{\beta,\lambda/n}[f(Z_n)]\right|&\leq \|f\|_\infty \mathbb{P}^{\beta,\lambda/n}\left(\tau_{Kn}(Z_n)\wedge\tau_{-Kn}(Z_n)\leq Tn^{1/\rho}\right),\\
\left|\mathbb{E}^{\beta,\lambda,K}[f(Z)]-\mathbb{E}^{\beta,\lambda}[f(Z)]\right|&\leq \|f\|_\infty \mathbb{P}^{\beta,\lambda}\left(\tau_{K}(Z)\wedge\tau_{-K}(Z)\leq T\right).
\end{align*}
Consequently, using Proposition \ref{prop:fixedK} and Lemma \ref{lem:limsup}, for any $K>0$,
\begin{align*}
\limsup_{n\to\infty}\left|\mathbb{E}^{\beta,\lambda/n}[f(Z_n)]-\mathbb{E}^{\beta,\lambda}[f(Z)]\right| \leq 2\|f\|_\infty \mathbb{P}^{\beta,\lambda}\left(\tau_{K-1}(Z)\wedge\tau_{-K+1}(Z)\leq T+1\right),
\end{align*}
and so we can conclude the proof by taking $K\to\infty$ and applying Lemma \ref{lem:nonexpl}. Finally, the claim that $Z^{\beta,\lambda}$ is continuous is covered by Lemma \ref{zbllem}.
\end{proof}

\begin{proof}[Proof of Lemma \ref{lem:limsup}]
We define $f:D([0,\infty),\R)\to[0,1]$ by
\begin{align*}
f(x):=\int_0^1 \1\{|x(u)|\geq K-s\text{ for some }u\in[0,t+s]\}\dd s.
\end{align*}
Note that $f$ is a continuous function satisfying
\begin{align*}
f(x)
\begin{cases}
=0,&\text{ if }\tau_{K-1}(x)\wedge \tau_{-K+1}(x)>t+1,\\[0mm]
=1,&\text{ if }\tau_{K}(x)\wedge \tau_{-K}(x)\leq t.
\end{cases}
\end{align*}
Hence, by Proposition \ref{prop:fixedK},
\begin{align*}
\lefteqn{\limsup_{n\to\infty}\mathbb{P}^{\beta,\lambda/n,Kn}\left(\tau_{Kn}(X)\wedge \tau_{-Kn}(X)\leq tn^{1/\rho}\right)}\\
&\leq \lim_{n\to\infty}\mathbb{E}^{\beta,\lambda/n,Kn}\left[f(n^{-1}X_{n^{1+1/\rho}\cdot})\right]\\
&= \mathbb{E}^{\beta,\lambda,K}\left[f(Z)\right]\\
&\leq \mathbb{P}^{\beta,\lambda,K}\left(\tau_{K-1}(Z)\wedge \tau_{-K+1}(Z)\leq t+1\right).
\end{align*}
\end{proof}

\begin{proof}[Proof of Lemma \ref{lem:nonexpl}]

We start with the case $\lambda=0$. Note that, by symmetry, it is enough to show that for any $\eps>0$ we can find $K_0$ such that for all $K\geq K_0$,
\begin{align}\label{eq:todo}
\mathbb{P}^{\beta,0}(\tau_K(Z)\leq u)<4\eps.
\end{align}
To check this, it is enough to consider the largest `gap' in $S^{\beta,0}([0,K])$ and the time $B$ has to spend on its left side before crossing it for the first time. Note that, as it attempts to cross the gap, $B$ has to spend some time on the left side. As $K$ increases, the length of the gap increases and $B$ will spend more time of on the left hand side before crossing. We will thus show that the time spent on the left becomes arbitrarily large.

First, choose $\delta>0$ such that
\begin{align}\label{eq:eps0}
\mathbf{P}\left(S^{\beta,0}(1)\leq \delta\right)<\eps.
\end{align}
Let
\begin{align*}
M_K&:=\sup_{t\in[1,K]}\left(S^{\beta,0}(t)-S^{\beta,0}(t^-)\right),\\
T_K&:=\mathop{\text{argmax}}_{t\in[1,K]}\left(S^{\beta,0}(t)-S^{\beta,0}(t^-)\right).
\end{align*}
Recall that the number of jumps of $S^{\beta,0}$ in $[1,K]$ of size at least $m$ has a Poisson distribution with mean $C_\beta m^{-\rho}(K-1)$. For any $m$, we can therefore find $K_0$ such that for all $K\geq K_0$,
\begin{align*}
\mathbf{P}(M_K\leq m)\leq \eps.
\end{align*}
Moreover, we claim that we can choose $\eta>0$ so that, for large enough $K$,
\begin{align}\label{eq:eps1}
\mathbf{P}\left(\mu^{\beta,0}([S^{\beta,0}(T_K^-)-\delta,S^{\beta,0}(T_K^-)])\leq \eta\right)\leq\eps.
\end{align}
Let $D:=\{t:\:S^{\beta,0}(t)\neq S^{\beta,0}(t^-)\}$ denote the set of discontinuities of $S^{\beta,0}$. To check \eqref{eq:eps1}, we first note that, conditional on $(M_K,T_K)$, it holds that
\[\left((t,S^{\beta,0}(t)-S^{\beta,0}(t^-)\right)_{t\in D\cap[1,K]\backslash\{T_K\}},\]
is a Poisson point process on $[1,K]\times (0,M_K)$ with intensity $dt \times C_\beta\rho x^{-\rho-1}\mathbf{1}_{\{x\in(0,M_K)\}}dx$ (cf.\ \cite[Proposition 2.4]{Res}). Moreover, under the conditional law, this is independent of
\[\left((t,S^{\beta,0}(t)-S^{\beta,0}(t^-)\right)_{t\in D\cap[0,1)},\]
which is a Poisson point process on $[0,1)\times (0,\infty)$ with intensity $dt \times C_\beta\rho x^{-\rho-1}\mathbf{1}_{\{x>0\}}dx$.
Since
\begin{align*}
\lefteqn{S^{\beta,0}(T_K^-)-S^{\beta,0}(T_K-\eta)}\\
&=\sum_{t\in D\cap[1\vee(T_K-\eta),T_K)}\left(S^{\beta,0}(t)-S^{\beta,0}(t^-)\right)
+\sum_{t\in D\cap[1\wedge(T_K-\eta),1)}\left(S^{\beta,0}(t)-S^{\beta,0}(t^-)\right),
\end{align*}
we readily deduce that $\mathbf{P}(S^{\beta,0}(T_K^-)-S^{\beta,0}(T_K-\eta)> \delta\:\vline\:M_K,\:T_K)$ is given by
\[\mathbf{P}\left(S^{\beta,0}(\eta)-\sum_{t\in D\cap[0,(T-1)\wedge \eta]}\left(S^{\beta,0}(t)-S^{\beta,0}(t^-)\right)\mathbf{1}_{\{S^{\beta,0}(t)-S^{\beta,0}(t^-)\geq M\}}>\delta\right)\]
with $(M,T)=(M_K,T_K)$. Clearly, regardless of the value of $(M_K,T_K)$, the above expression is bounded above by $\mathbf{P}(S^{\beta,0}(\eta)>\delta)$, and taking expectations thus yields
\[\mathbf{P}\left(S^{\beta,0}(T_K^-)-S^{\beta,0}(T_K-\eta)> \delta\right)\leq \mathbf{P}\left(S^{\beta,0}(\eta)>\delta\right).\]
It follows that there exists an $\eta>0$ such that, for large enough $K$,
\[\mathbf{P}\left(S^{\beta,0}(T_K^-)-S^{\beta,0}(T_K-\eta)>\delta\right)< \varepsilon.\]
On the complement of the event in the probability above, we have that
\[\mu^{\beta,0}\left([S^{\beta,0}(T_K^-)-\delta,S^{\beta,0}(T_K^-)]\right)\geq
\mu^{\beta,0}\left([S^{\beta,0}(T_K-\eta),S^{\beta,0}(T_K^-)]\right)=\mathbf{E}\left(c^{\beta,0}(\omega_0)\right)\eta,\]
and thus, after replacing $\eta$ by $\eta/\mathbf{E}(c^{\beta,0}(\omega_0))$, we obtain \eqref{eq:eps1}.

Next, we introduce some notation for the Brownian motion $B$. Let
\begin{align*}
\sigma&:=\inf\{t>0:B_t=S^{\beta,0}(T_K)\},\\
L^B_t[x,y]&:=\inf_{z\in[x,y]} L^B_{t}(z).
\end{align*}
Observe that $\sigma$ is the first time $B$ has crossed the gap $[S^{\beta,0}(T_K^-),S^{\beta,0}(T_K)]$, and that  $L^B_\sigma[S^{\beta,0}(T_K^-)-\delta,S^{\beta,0}(T_K^-)]$ is a lower bound for the local time accumulated on the left side of the gap before this first crossing time. Let $(W_t)_{t\geq 0}$ denote a two-dimensional Brownian motion, independent of $S^{\beta,0}$, and observe that, conditionally on $S^{\beta,0}$ with $S^{\beta,0}(T^-_K)-\delta>0$, by the strong Markov property and the Ray-Knight theorem,
\begin{align}\label{eq:local}
L^B_{\sigma}[S^{\beta,0}(T^-_K)-\delta,S^{\beta,0}(T_K^-)]\isDistr
\inf_{t\in M_K+[0,\delta]} \|W_t\|_2^2.
\end{align}
We can therefore choose $m$ large enough that, on $\{S^{\beta,0}(T^-_K)-\delta>0,\:M_K\geq m\}$,
\begin{align}\label{eq:eps3}
\mathbf{P}\left(L^B_{\sigma}[S^{\beta,0}(T^-_K)-\delta,S^{\beta,0}(T_K^-)]\leq u\eta^{-1}\:\vline\:S^{\beta,0}\right)\leq\eps.
\end{align}
Finally, we observe that
\begin{equation}\label{eq:conclusion1}
\begin{split}
\{\tau_{K}(Z)> u\}\supseteq &\{S^{\beta,0}(T^-_K)-\delta>0\}\cap\{\mu^{\beta,0}([S^{\beta,0}(T_K^-)-\delta,S^{\beta,0}(T_K^-)])>\eta\}\\
&\cap\{M_K> m\}\cap\{L^B_{\sigma}[S^{\beta,0}(T^-_K)-\delta,S^{\beta,0}(T_K^-)]>u\eta^{-1}\},
\end{split}
\end{equation}
and hence \eqref{eq:todo} follows from \eqref{eq:eps0}--\eqref{eq:eps3}.

We next consider $\lambda>0$. Note that, using well-known harmonic properties of Brownian motion,
\begin{align*}
\mathbb{P}^{\beta,\lambda}\left(\tau_K(Z)\wedge\tau_{-K}(Z)=\tau_{-K}(Z)\:\vline\:S^{\beta,\lambda}\right)=\frac{|S^{\beta,\lambda}(-K)^{-1}|}{|S^{\beta,\lambda}(K)^{-1}|+|S^{\beta,\lambda}(-K)^{-1}|}.
\end{align*}
Since $(S^{\beta,\lambda}(t))_{t\in\R}$ is bounded for $t\to\infty$ and unbounded for $t\to-\infty$, we can choose $K$ large enough that
\begin{align*}
\mathbb{P}^{\beta,\lambda}(\tau_K(Z)\wedge\tau_{-K}(Z)=\tau_{-K}(Z))<\eps.
\end{align*}
The proof of $\mathbb{P}^{\beta,\lambda}(\tau_K(Z)\leq u)<4\eps$ follows the same lines as in the case $\lambda=0$. The main difference is that since $S^{\beta,\lambda}(u)$ is bounded for $u\to\infty$, we have to consider the largest gaps with suitable weighting, and replace the interval to the left of the gap by a smaller one. Since the speed-measure grows exponentially, the process will still accumulate a large local time in the interval before crossing the gap. More precisely, we replace $\sigma$ by
\begin{align*}
\widehat\sigma&:=\inf\{t>0:B_t=S^{\beta,\lambda}(T_K)\},
\end{align*}
and instead of \eqref{eq:conclusion1}, we will bound the probability of the right-hand side of
\begin{equation*}
\begin{split}
\{\tau_{K}(Z)\geq u\}\supseteq &\{S^{\beta,0}(T^-_K)-\delta>0\}\cap\{M_K> m\}\\
&\cap\{\mu^{\beta,\lambda}([S^{\beta,\lambda}(T_K^-)-\delta e^{-2\lambda T_K/\rho},S^{\beta,\lambda}(T_K^-)])\geq\eta e^{2\lambda T_K/\rho}\}\\
&\cap\{L^B_{\widehat \sigma}[S^{\beta,\lambda}(T^-_K)-\delta e^{-2\lambda T_K/\rho},S^{\beta,\lambda}(T_K^-)]>u\eta^{-1}e^{-2\lambda T_K/\rho}\}.
\end{split}
\end{equation*}
Let $\delta$, $M_K$ and $T_K$ be defined as before (i.e., using $S^{\beta,0}$), so that
\begin{align*}
S^{\beta,\lambda}(T_K)-S^{\beta,\lambda}(T_K^-)=M_Ke^{-2\lambda T_K/\rho}.
\end{align*}
Moreover, arguing similarly to \eqref{eq:eps1}, it is possible to check that there exists an $\eta>0$ such that
\[\mathbf{P}\left(\mu^{\beta,\lambda}([S^{\beta,\lambda}(T_K^-)-\delta e^{-2\lambda T_K/\rho},S^{\beta,\lambda}(T_K^-)])\leq \eta e^{2\lambda T_K/\rho}\right)\leq\eps.\]
Then the observation at \eqref{eq:local} becomes
\begin{align*}
L^B_{\widehat \sigma}[S^{\beta,\lambda}(T_K^-)-\delta e^{-2\lambda T_K/\rho},S^{\beta,\lambda}(T_K^-)]&
\isDistr \inf_{t\in [M_K,M_K+\delta]e^{-2\lambda T_K/\rho}}\frac 12\|W_t\|_2^2.
\end{align*}
We can therefore choose $m$ large enough that, on $\{S^{\beta,\lambda}(T^-_K)-\delta>0,\:M_K\geq m\}$,
\begin{align*}
\mathbf{P}\left(L^B_{\widehat\sigma}[S^{\beta,\lambda}(T_K^-)-\delta e^{-2\lambda T_K/\rho},S^{\beta,\lambda}(T_K^-)]\leq u \eta^{-1} e^{-2\lambda T_K/\rho}\:\vline\:S^{\beta,\lambda}\right)<\eps.
\end{align*}
Combined with \eqref{eq:eps0}--\eqref{eq:eps1}, the desired bound for the right-hand side of \eqref{eq:conclusion1} follows.
\end{proof}

\begin{proof}[Proof of Lemma \ref{zbllem}]
The limiting process $Z^{\beta,\lambda}$ is a one-dimensional bi-generalized diffusion process in the sense of \cite{MR984748}, with scale function $S^{\beta,\lambda}$ and speed measure $e^{2\lambda r/\rho}\dd r$. Thus the continuity claim follows directly from \cite[Corollary 3.1, 2)]{MR984748} since $S^{\beta,\lambda}$ is strictly increasing and $e^{2\lambda r/\rho}\dd r$ is a continuous measure with full support.

The Markov property claim for $Z^{\beta,\lambda}$ is also essentially contained in \cite{MR984748}. However, since a part of the proof is left to the reader, we choose to give a direct proof. In the following, all `almost-sure' properties of $B_{H_t}^{\beta,\lambda,K}$ or $Z^{\beta,\lambda,K}_t$ are with respect to the conditional law $P^{\beta,\lambda,K}(\cdot|S^{\beta,\lambda})$. We start by considering the restricted process $Z^{\beta,\lambda,K}$. To this end, we recall that  $B_{H^{\beta,\lambda,K}_\cdot}$ is a Markov process and observe that, by \cite[Theorem 10.4]{Kigres}, it admits a (continuous) transition density with respect to $\mu^{\beta,\lambda,K}$. Hence, for any fixed time $t>0$, the law of $B_{H^{\beta,\lambda,K}_t}$ is absolutely continuous with respect to $\mu^{\beta,\lambda,K}$. Since $D:=\{t:\:S^{\beta,\lambda}(t)-S^{\beta,\lambda}(t^-)\}$, the set of discontinuities of $S^{\beta,\lambda}$, is almost-surely a countable set, it follows that, for fixed $t>0$, it is almost-surely the case that
\begin{align*}
B_{H^{\beta,\lambda,K}_t}\in\overline{S^{\beta,\lambda}(\mathbb{R})}\backslash (\cup_{s\in D}\{S_{s^-},S_s\}\cup\{S_\infty^{\beta,\lambda}\}),
\end{align*}
where $S_\infty^{\beta,\lambda}:=\lim_{t\rightarrow\infty}S^{\beta,\lambda}_t$. Now, it is a straightforward exercise to check that the map
\begin{align*}
S^{\beta,\lambda}:\mathbb{R}\backslash D\rightarrow\overline{S^{\beta,\lambda}(\mathbb{R})}\backslash(\cup_{s\in D}\{S_{s^-},S_s\}\cup\{S_\infty^{\beta,\lambda}\})
\end{align*}
is a bijection. Hence $Z^{\beta,\lambda,K}_t\in\mathbb{R}\backslash D$ almost-surely, and appealing to the Markov property of $B_{H^{\beta,\lambda,K}_\cdot}$ yields, for $s<t$, $z\in \mathbb{R}\backslash D$ and measurable $A\subseteq\mathbb{R}$,
\begin{align*}
&\mathbb{P}^{\beta,\lambda,K}\left(Z_t\in A\:\vline\:Z_s=z,\:(Z_r)_{r\leq s},\:S^{\beta,\lambda}\right)\\
&= \mathbb{P}^{\beta,\lambda,K}\left(S^{\beta,\lambda}\left(B_{H^{\beta,\lambda,K}_t}\right)\in A\:\vline\:B_{H^{\beta,\lambda,K}_s}=S^{\beta,\lambda}(z),\:(Z_r)_{r\leq s},\:S^{\beta,\lambda}\right)\\
&=\mathbb{P}^{\beta,\lambda,K}\left(S^{\beta,\lambda}\left(B_{H^{\beta,\lambda,K}_t}\right)\in A\:\vline\:B_{H^{\beta,\lambda,K}_s}=S^{\beta,\lambda}(z),\:S^{\beta,\lambda}\right)\\
&=\mathbb{P}^{\beta,\lambda,K}\left(Z_t\in A\:\vline\:Z_s=z,\:S^{\beta,\lambda}\right),
\end{align*}
which establishes the Markov property for $Z^{\beta,\lambda,K}$. Combining this result with Lemma \ref{lem:nonexpl}, one readily obtains the corresponding result for $Z^{\beta,\lambda}$.
\end{proof}

\section{Incorporation of heavy-tailed holding times}\label{sec:btm}

We now explain how to deal with the random variables $(\tau_i)_{i\in \mathbb{Z}}$, and thereby prove Theorem \ref{thm:main2}. Since the changes needed are relatively minor, we will be brief with the details. A first observation is that the time change does not affect the effective resistance of the discrete model. Thus we still have the conclusion of Theorem \ref{thm:approx}. This suggests that it is enough to prove the analogue of Theorem \ref{thm:measure} for the new invariant measure $\tilde\mu^{\beta,\lambda/n,Kn}$, defined on $\{\oomega_{-Kn},...,\oomega_{Kn}\}$ by setting
\begin{align*}
\tilde\mu^{\beta,\lambda/n,Kn}(\oomega_i):=\tau_i c^{\beta,\lambda/n,Kn}(\oomega_i).
\end{align*}
A minor technical obstacle is that in order to apply the Skorohod representation theorem as in Theorem \ref{thm:metric} we need to prove that the effective resistance and invariant measure converge in law jointly. (In the model without $(\tau_i)_{i\in \mathbb{Z}}$, that the limiting invariant measure was deterministic meant that such joint convergence was equivalent to the convergence of the marginals.) Specifically, together with our previous arguments, the following result yields Theorem \ref{thm:main2}.

\begin{proposition}
Jointly with \eqref{eq:convergence}, it holds that
\begin{align*}
\lefteqn{\left(n^{-1/\kappa}\tilde \mu^{\beta,\lambda/n,Kn}\left(\{\oomega_{\floor{-Kn}},...\oomega_{\floor{un}}\}\right)\right)_{-K\leq u\leq K}}\\
&\qquad\qquad\xrightarrow[n\to\infty]d\left(\tilde \mu^{\beta,\lambda}\left([S^{\beta,\lambda}(-K),S^{\beta,\lambda}(u)]\right)\right)_{-K\leq u\leq K}
\end{align*}
with respect to the Skorohod $J_1$-topology.
\end{proposition}

\begin{proof}
We first assume $\lambda=0$. Recall that
\[\B_n=\left\{i\in\{-Kn,\dots,Kn-1\}:r^{0,0}(\omega_i,\omega_{i+1})\geq n^{3/(4\rho)}\right\}\]
denotes the sites with a big nearest-neighbor resistance, and that $\chiu_n(i)$ (see \eqref{eq:chiu}) is a random variable depending only on $\{\omega_{i+1+j}-\omega_{i+1}:1\leq j\leq a\log (n)\}$ and $\{\omega_{i}-\omega_{i-j}:1\leq j\leq a\log (n)\}$, and is such that
\begin{align}
n^{-1/\rho}\sup_{-Kn\leq i\leq j\leq Kn}\left|R^{\beta,0}(\oomega_i,\oomega_j)-\sum_{k=i}^{j}\chiu_n(k) \1_{k\in\B_n}r^{0,0}(\oomega_k,\oomega_{k+1})\right|\xrightarrow[n\to\infty]{\mathbf{P}}0,\label{eq:probconvrho}
\end{align}
\begin{align}
\lefteqn{\left(n^{-1/\rho}\sum_{j=\floor{-Kn}}^{\floor{un}} \chiu_n(j)\1_{j\in\B_n}r^{0,0}(\oomega_j,\oomega_{j+1})\right)_{-K\leq u\leq K}}\nonumber\\
&\qquad \qquad\xrightarrow[n\to\infty]d \left(S^{\beta,0}(u)-S^{\beta,0}(-K)\right)_{-K\leq u\leq K}.\label{eq:distconvrho}
\end{align}
To define corresponding approximations for the measure, we introduce
\begin{align*}
\widehat\B_n&:=\{i\in\{-Kn,\dots,Kn\}:\:\tau_i\geq n^{3/(4\kappa)}\}\\
\widehat{c}_n(i)&:=\sum_{j:|j-i|\leq \log (n)} c^{\beta,0,Kn}(\oomega_i,\oomega_j),
\end{align*}
and consider the event that sites having either a big resistance or a big holding time are well-separated:
\[A_{6,n}:=\left\{|i-j|\geq n^{1/4}\mbox{ for all distinct } i,j\in \B_n\cup\widehat \B_n\right\}\cap\left\{\B_n\cap\widehat \B_n= \emptyset\right\}.\]
A rerun of the proof of Lemma \ref{lem:aux} shows that
\begin{align}\label{eq:An}
\mathbf{P}(A_{6,n}^c)\xrightarrow[n\to\infty]{} 0.
\end{align}
Recall that $r^{0,0}(\oomega_i,\oomega_{i+1})$ is independent of $\chiu_n(i)$, and note that, by construction, $\tau_i$ is independent of $\widehat c_n(i)$. Let $(\widetilde\chi_n(i),\widetilde \chi_\infty(i))_{i\in\Z}$ and $(\widetilde c_n(i),\widetilde c_\infty(i))_{i\in\Z}$ denote independent, i.i.d.\ copies of $(\chiu_n(i),\chi^{\beta,0}(i))_{i\in\Z}$ and $(\widehat c_n(i), c^{\beta,0}(\omega_i))_{i\in\Z}$. Following the same argument as in the proof of \eqref{eq:PMprobconv}, we obtain
\begin{align}
n^{-1/\kappa}\sup_{-K\leq a<b\leq K} \left|\sum_{j=\floor{an}}^{\floor{bn}} \widetilde c_n(j)\1_{j\in\widehat B_n}\tau_j-\sum_{j=\floor{an}}^{\floor{bn}} \widetilde c_\infty(j)\tau_j\right|\xrightarrow[n\to\infty]{\mathbf{P}}0.\label{eq:probconvtau1}
\end{align}
Since $(\widetilde c_\infty(j)\tau_j)_{j\in\Z}$ is an i.i.d.\ collection with
\begin{align*}
\mathbf{P}(\widetilde c_\infty(0)\tau_0\geq u)\sim\mathbf{E}\left(c^{\beta,0}(\omega_0)^\kappa\right)u^{-\kappa},
\end{align*}
(cf.\ \eqref{xxx},) we get
\begin{align}
\left(n^{-1/\kappa}\sum_{j=\floor{-Kn}}^{\floor{un}} \widetilde c_n(j)\1_{j\in\widehat B_n}\tau_j\right)_{-K\leq u\leq K}\xrightarrow[n\to\infty]d \left(\widehat \mu^{\beta,0}([S^{\beta,0}(-K),S^{\beta,0}(u)]\right)_{-K\leq u\leq K}. \label{eq:distconvtau}
\end{align}
Moreover, because of the truncation in the definition of $\chiu_n$ and $\widehat c_n$, we have, conditional on $r^{0,0}$ and $\tau$, and on the event $A_{6,n}$,
\begin{equation}
 \begin{split}
&\left(n^{-1/\rho}\sum_{j=\floor{un}}^{\floor{vn}} \chiu_n(j)\1_{j\in\B_n}r^{0,0}(\oomega_j,\oomega_{j+1}),n^{-1/\kappa}\sum_{j=\floor{an}}^{\floor{bn}} \widehat c_n(j)\1_{j\in\widehat B_n}\tau_j\right)_{\substack{-K\leq u<v\leq K\\-K\leq a<b\leq K}}\\
&\isDistr\left(n^{-1/\rho}\sum_{j=\floor{un}}^{\floor{vn}} \widetilde \chi_n(j)\1_{j\in\B_n}r^{0,0}(\oomega_j,\oomega_{j+1}),n^{-1/\kappa}\sum_{j=\floor{an}}^{\floor{bn}} \widetilde c_n(j)\1_{j\in\widehat B_n}\tau_j\right)_{\substack{-K\leq u<v\leq K\\-K\leq a<b\leq K}}.
\end{split}
\label{eq:second}
\end{equation}
Since the coordinates in the second vector are independent, we see, using \eqref{eq:distconvrho} and \eqref{eq:An}, that \eqref{eq:distconvtau} holds jointly with
\begin{align*}
\lefteqn{\left(n^{-1/\rho}\sum_{j=\floor{-Kn}}^{\floor{un}} \widetilde \chi_n(j)\1_{j\in\B_n}r^{0,0}(\oomega_j,\oomega_{j+1})\right)_{-K\leq u\leq K}}\\
&\qquad\qquad \xrightarrow[n\to\infty]d \left(S^{\beta,0}(u)-S^{\beta,0}(-K)\right)_{-K\leq u\leq K}.
\end{align*}
The claim therefore follows from \eqref{eq:second} and \eqref{eq:An}, together with \eqref{eq:probconvrho} and the fact that \eqref{eq:probconvtau1} holds with $\widetilde c_n$ and $\widetilde c_\infty$ replaced by $\widehat c_n$ and $c^{\beta,0}$.

Finally, in the case $\lambda>0$, we argue as in Proposition \ref{prop:limits}, noting that outside of an event with vanishing probability, we have, for all $-Kn\leq i\leq Kn$,
\begin{align*}
\lefteqn{\left(1-4\lambda n^{-1/4}/\rho\right)e^{2\lambda i/(\rho n)}c^{\beta,0,Kn}(\oomega_i)}\\
&\qquad\qquad\leq c^{\beta,\lambda/n,Kn}(\oomega_i)\leq \left(1+4n^{-1/4}\lambda/\rho\right)e^{2\lambda i/(\rho n)}c^{\beta,0,Kn}(\oomega_i).
\end{align*}

\end{proof}

\section{Quenched fluctuations}\label{sec:quenched}

\begin{proof}[Proof of Proposition \ref{prop:quenched}]
Recall that we write $P^{\beta,0,n}$ for the quenched law of a process $(Y(t))_{t\geq 0}$ on the truncated state space $\{\oomega_{-n},...,\oomega_n\}$ with generator defined in \eqref{eq:truncated_generator}, and
\begin{equation*}
\sigma_n=\inf\{t\geq 0:X(t)\in\{...,\omega_{-n-1},\omega_{-n}\}\cup\{\omega_n,\omega_{n+1},...\}\}.
\end{equation*}
Clearly $(X(t))_{0\leq t<\sigma_n}$ has the same law as $(Y(t))_{0\leq t<\sigma_n}$. It is thus enough to prove that, $\mathbf P$-a.s.,
\begin{align*}
\limsup_{n\to\infty}P^{\beta,0,n}\left(T_n\leq \frac{Mn^{\frac{1}{\rho}+1}}{\log\log^{\frac{1}{\rho}-1} n}\right)>0,
\end{align*}
where $T_n:=\inf\{t\geq 0:Y(t)=\oomega_n\}$. From Rayleigh's monotonicity law, we know that the effective resistance between $\oomega_0$ and $\oomega_n$ is smaller than the corresponding nearest-neighbor resistance, i.e.,
\begin{align*}
R^{\beta,0,n}(\oomega_0,\oomega_n)\leq \sum_{i=0}^{n-1}r^{\beta,0}(\oomega_i,\oomega_{i+1}).
\end{align*}
Since $(r^{\beta,0}(\oomega_i,\oomega_{i+1}))_{i=0,...,n-1}$ is an i.i.d.\ sequence, we can apply the LIL for random variables in the domain of attraction of a stable distribution, see \cite[Theorem 1]{lipschutz}. We obtain that there exists a deterministic constant $c(\rho)\in(0,\infty)$ such that $\mathbf P$-a.s.,
\begin{align*}
\liminf_{n\to\infty}\frac{\log\log^{1/\rho-1} n}{n^{1/\rho}}\sum_{i=0}^{n-1}r^{\beta,0}(\oomega_i,\oomega_{i+1})\leq c(\rho).
\end{align*}
Next, by applying the commute time identity (see \cite[Theorem 4.27]{barlowdof}) to the finite graph $\{\oomega_{-n},...,\oomega_n\}$, we obtain
\begin{align*}
E^{\beta,0}[T_n|Y(0)=\oomega_0]+E^{\beta,0}[T_0|Y(0)=\oomega_n]= \mu^{\beta,0}(\oomega_{-n},...,\oomega_n)R^{\beta,0}(\oomega_0,\oomega_n),
\end{align*}
and so
\begin{align*}
E^{\beta,0}[T_n|Y(0)=\oomega_0]&\leq\mu^{\beta,0}(\oomega_{-n},...,\oomega_n)R^{\beta,0}(\oomega_0,\oomega_n).
\end{align*}
Finally, we recall from Theorem \ref{thm:measure} that $\mathbf P$-a.s.,
\begin{align*}
\lim_{n\to\infty}\frac 1n\mu^{\beta,0}(\oomega_{-n},...,\oomega_n)=\mathbf E[c^{\beta,0}(\oomega_0)].
\end{align*}
Combining all these observations and choosing $M$ large enough, we get, $\mathbf P$-a.s.,
\begin{align*}
\liminf_{n\to\infty}P^{\beta,0}\left(T_n>\frac{Mn^{\frac{1}{\rho}+1}}{\log\log^{\frac{1}{\rho}-1} n}\right)&\leq \liminf_{n\to\infty}\frac{\log\log^{1/\rho-1} n}{Mn^{1/\rho+1}}E^{\beta,0}[T_n]\\
&\leq \frac 1{M}c(\rho)\mathbf E[c^{\beta,0}(\oomega_0)]\\
&<1,
\end{align*}
which yields the result.
\end{proof}

\begin{appendix}

\section*{Homogenisation via a resistance scaling limit}\label{homogsec}

The purpose of this appendix is to provide further details for the homogenisation claim made in Remark \ref{homogrem}. In particular, we will establish the following result (cf.\ \eqref{homog}). Whilst we believe that the vanishing drift ($\lambda>0$) case may also be dealt with via resistance arguments, with the limit being Brownian motion with drift, we restrict to the driftless ($\lambda=0$) case in order to present a relatively simple argument (based on Kingman's sub-additive ergodic theorem).

\begin{theorem}\label{thm:diffusive}
For every $\rho>1$ and $\beta \geq 0$, for $\mathbf{P}$-a.e.\ realisation of $(\omega,E)$, it holds that as $n\to\infty$,
\[{P}^{\beta,0}\left((n^{-1}X_{n^{2}t})_{t\geq 0}\in\cdot\right)\]
converge weakly as probability measures on $D([0,\infty),\mathbb{R})$ to the law of $(B_{\sigma^2 t})_{t\geq 0}$, where $(B_{t})_{t\geq 0}$ is standard Brownian motion, and $\sigma^2>0$ is a deterministic constant.
\end{theorem}

Before proving this result, we give a lemma concerning the scaling limit of the resistance, which replaces Theorem \ref{thm:approx} in the present parameter regime. In this section, we use the abbreviation $R^\beta:=R^{\beta,0}$.

\begin{lemma}\label{reslemdif} For every $\rho>1$ and $\beta \geq 0$, for $\mathbf{P}$-a.e.\ realisation of $(\omega,E)$, it holds that
\[\left(n^{-1}R^{\beta}(\omega_{\lfloor un\rfloor},\omega_{\lfloor vn\rfloor})\right)_{u,v\in\mathbb{R}}\xrightarrow[n\to\infty]{}\left(R_\infty|u-v|\right)_{u,v\in\mathbb{R}}\]
uniformly on compacts, where $R_\infty\in(0,\infty)$ is a deterministic constant.
\end{lemma}
\begin{proof} We start by noting that, for $i<j$,
\begin{equation}\label{rrest}
R^\beta(\omega_i,\omega_j)\leq\sum_{k=i}^{j-1}r^{\beta,0}(\omega_k,\omega_{k+1})\leq C\sum_{k=i}^{j-1}e^{\omega_{k+1}-\omega_k}.
\end{equation}
Since $(e^{\omega_{k+1}-\omega_k})_{k\in\mathbb{Z}}$ are i.i.d.\ and have a first moment when $\rho>1$, it follows from the functional law of large numbers for the partial sums of $(e^{\omega_{k+1}-\omega_k})_{k\in\mathbb{Z}}$ (see \cite[Theorem 1.1]{RS}, for example) that the rescaled resistances $(n^{-1}R^{\beta}(\omega_{\lfloor un\rfloor},\omega_{\lfloor vn\rfloor}))_{u,v\in\mathbb{R}}$ are $\mathbf{P}$-a.s.\ sequentially compact. It thus remains to characterise the limit as that given in the statement of the theorem. Given the stationarity of the model under the shift from $(\omega_i,E_i)_{i\in\mathbb{Z}}$ to $(\omega_{i+1}-\omega_1,E_{i+1})_{i\in\mathbb{Z}}$, to do this it will be enough to show that, $\mathbf{P}$-a.s.,
\[n^{-1}R^{\beta}(\omega_{0},\omega_{n})\rightarrow R_\infty\in(0,\infty).\]
To this end, we first observe that the triangle inequality for the resistance metric tells us that $R^\beta(\omega_0,\omega_{i+j})\leq R^\beta(\omega_0,\omega_{i})+R^\beta(\omega_i,\omega_{i+j})$. Moreover, from \eqref{rrest}, we obtain the integrability of the random variables $R^\beta(\omega_{0},\omega_{n})$. Since the shift map $(\omega_i,E_i)_{i\in\mathbb{Z}}\mapsto(\omega_{i+1}-\omega_1,E_{i+1})_{i\in\mathbb{Z}}$ is not only stationary, but also ergodic, we consequently obtain from Kingman's subadditive ergodic theorem that
\[n^{-1}R^{\beta}(\omega_{0},\omega_{n})\rightarrow \inf_{n}n^{-1}\mathbf{E}\left(R^{\beta}(\omega_{0},\omega_{n})\right)\in[0,\infty).\]
To complete the proof, we therefore only need to show that $R_\infty:=\inf_{n}n^{-1}\mathbf{E}(R^{\beta}(\omega_{0},\omega_{n}))$ is strictly positive.

Now, by definition
\begin{align*}
R^{\beta}(\omega_{0},\omega_{n})^{-1}&=\inf\left\{\sum_{i,j\in\mathbb{Z}}c^{\beta,0}(\omega_i,\omega_j)\left(f(\omega_i)-f(\omega_j)\right)^2:\:f:\omega\rightarrow[0,1],\:\omega_0=0,\:\omega_n=1\right\}\\
&\leq \Sigma(n):=\sum_{i,j\in\mathbb{Z}}c^{\beta,0}(\omega_i,\omega_j)\left(f_n(i)-f_n(j)\right)^2,
\end{align*}
where $f_n(i):=(\frac{i}{n}\vee0)\wedge 1$. Hence, for any $C>0$,
\begin{equation}\label{rrrest}
R_\infty\geq\inf_n n^{-1}\mathbf{E}\left(\Sigma(n)^{-1}\right)\geq C^{-1}\inf_n \mathbf{P}\left(n\Sigma(n)\leq C\right).
\end{equation}
By Markov's inequality, we have that
\begin{align*}
\mathbf{P}\left(n\Sigma(n)\geq C\right)&\leq C^{-1}n\mathbf{E}\Sigma(n)\\
&\leq C^{-1}n\sum_{i,j\in\mathbb{Z}}\mathbf{E}\left(e^{-|\omega_i-\omega_j|}\right)\left(f_n(i)-f_n(j)\right)^2\\
&\leq C^{-1}n\sum_{i,j\in\mathbb{Z}}c^{|i-j|}\left(f_n(i)-f_n(j)\right)^2\\
&\leq 2C^{-1}n\sum_{i\leq 0}\sum_{j\geq 0}c^{j-i}\left(\frac{j}{n}\right)^2+2C^{-1}n\sum_{i=1}^n\sum_{j\geq i}c^{j-i}\left(\frac{j}{n}-\frac{i}{n}\right)^2\\
&\leq 2C^{-1}\sum_{i\geq 0}c^i\sum_{j\geq 0}c^{j}j^2+2C^{-1}\sum_{j\geq 0}c^{j}j^2,
\end{align*}
where $c:=\mathbf{E}(e^{-\omega_1})\in(0,1)$. In particular, uniformly in $n$, the upper bound above converges to 0 as $C\rightarrow \infty$. It follows that, for $C$ suitably large,
\[\inf_n \mathbf{P}\left(n\Sigma(n)\leq C\right)\geq \frac{1}{2},\]
which, together with \eqref{rrrest}, means we are done.
\end{proof}

\begin{proof}[Proof of Theorem \ref{thm:diffusive}] The argument of Theorem \ref{thm:measure} still applies when $\rho>1$, and can be used to check that, $\mathbf{P}$-a.s., for every $a<b$,
\[\mu^{\beta}_n\left(\{\omega_{an},\dots,\omega_{bn}\}\right)\xrightarrow[n\to\infty]{} \mathbf{E}\left(c^{\beta,0}(\omega_0)\right)(b-a),\]
where $\mu^{\beta}_n(\omega_i):=n^{-1}c^{\beta,0}(\omega_i)$. It readily follows from this convergence and Lemma \ref{reslemdif} that, $\mathbf{P}$-a.s.,
\[\left(\omega,R_n^{\beta},\mu^{\beta}_n,\omega_0,\Phi_n\right)\xrightarrow[n\to\infty]{} \left(\mathbb{R},R_\infty d_E,\mathbf{E}(c^{\beta,0}(\omega_0))\text{Leb},0,\Phi\right),\]
where $R_n^{\beta}:=n^{-1}R^{\beta}$, $d_E$ is the Euclidean distance on $\mathbb{R}$, $\text{Leb}$ is the Lebesgue measure on $\mathbb{R}$, $\Phi_n(\omega_i):=n^{-1}\omega_i$ and $\Phi(x):=x/\rho$. Here convergence is stated with respect to the locally compact version of the topology introduced in Section \ref{sec:mms} (see \cite[Section 7]{croydon2018} for details). Associating processes with these spaces in the way described at the start of Section \ref{sec:mr}, the result with $\sigma^2=1/\rho^2R_\infty \mathbf{E}(c^{\beta,0}(\omega_0))$ will therefore follow from \cite[Theorem 7.2]{croydon2018} if we can check the following non-explosion condition: $\mathbf{P}$-a.s.,
\[\lim_{K\rightarrow\infty}\liminf_{n\rightarrow\infty}R_n^{\beta}\left(\omega_0,\{\dots,\omega_{-Kn-1},\omega_{-Kn}\}\cup\{\omega_{Kn},\omega_{Kn+1},\dots\}\right)=\infty.\]
(This is a corrected version of \cite[Assumption 1.1(b)]{croydon2018}; in the reference, the liminf was wrongly written as limsup.) To check the latter statement, we estimate the relevant resistance in a similar way to the proof of Lemma \ref{reslemdif}. Namely,
\begin{align*}
\lefteqn{R_n^{\beta}\left(\omega_0,\{\dots,\omega_{-Kn-1},\omega_{-Kn}\}\cup\{\omega_{Kn},\omega_{Kn+1},\dots\}\right)^{-1}}\\
&\leq n\tilde{\Sigma}(n):=n\sum_{i,j\in\mathbb{Z}}c^{\beta,0}(\omega_i,\omega_j)\left(\tilde{f}_n(i)-\tilde{f}_n(j)\right)^2,
\end{align*}
where $\tilde{f}_n(i):=\frac{|i|}{Kn}\wedge 1$. We write
\begin{align*}
n\tilde{\Sigma}(n)&\leq 2n\sum_{i=-Kn}^{Kn} \sum_{j\in\Z}c^{\beta,0}(\omega_i,\omega_j)\left(\tilde{f}_n(i)-\tilde{f}_n(j)\right)^2\notag\\
&\leq \frac{2}{K^2n}\sum_{i=-Kn}^{Kn} \sum_{j\in\Z}c^{\beta,0}(\omega_i,\omega_{i+j})j^2\\
&=:\frac{2}{K^2n}\sum_{i=-Kn}^{Kn} \xi_i.
\end{align*}
Clearly $\mathbf E[\xi_0]=\sum_{j\in\Z}j^2c^{|j|}<\infty$, where we recall $c=\mathbf E[c^{\beta,0}(\omega_0,\omega_1)]\in(0,1)$, and so, by the ergodic theorem, we have that, $\mathbf{P}$-a.s.,
\begin{align*}
\lim_{n\to\infty} \frac{2}{K^2n}\sum_{i=-Kn}^{Kn} \xi_i=\frac {4\mathbf E[\xi_0]}K.
\end{align*}
In conclusion, we have shown that, $\mathbf{P}$-a.s.,
\begin{align*}
\liminf_{n\rightarrow\infty}R_n^{\beta}\left(\omega_0,\{\dots,\omega_{-Kn-1},\omega_{-Kn}\}\cup\{\omega_{Kn},\omega_{Kn+1},\dots\}\right)\geq \frac{K}{4 \mathbf E[\xi_0]}\xrightarrow[K\to\infty]{}\infty.
\end{align*}
\end{proof}

\end{appendix}

\begin{acks}[Acknowledgments]
The authors would like to thank Takashi Kumagai for his contributions in the early part of the discussions that led to this article. They also thank a referee for their very careful reading of an earlier version of the paper and pointing out an important error in the argument.
\end{acks}

\begin{funding}
This research was supported by JSPS Grant-in-Aid for Scientific Research (C) 19K03540, JSPS Grant-in-Aid for Scientific Research (A) 17H01093, a JSPS Postdoctoral Fellowship for Research in Japan, Grant-in-Aid for JSPS Fellows 19F19814, and the Research Institute for Mathematical Sciences, an International Joint Usage/Research Center located in Kyoto University.
\end{funding}


\end{document}